\documentclass[reqno, dvipdfmx]{amsart}
\usepackage{amsthm,amsmath,amssymb,amsthm}
\usepackage[left=3.5cm, right=3.5cm, top=3cm, bottom=3cm, includefoot]{geometry}
\usepackage[all]{xy}
\usepackage{color}
\usepackage{appendix}
\usepackage{tikz}
\usepackage{tikz-cd}
\usepackage{bbm}
\usepackage{mathtools}

\newtheoremstyle{mystyle}
  {}
  {}
  {}
  {}
  {\bfseries}
  {.}
  { }
  {}

\theoremstyle{mystyle}
\theoremstyle{definition}
\newtheorem{theorem}{Theorem}[section]
\newtheorem{proposition}{Proposition}[section]
\newtheorem{lemma}{Lemma}[section]
\newtheorem{corollary}{Corollary}[section]

\newtheorem{example}{Example}[section]
\newtheorem{definition}{Definition}[section]
\newtheorem{remark}{Remark}[section]
\numberwithin{equation}{section}

\begin{document}
\title[Dynamical relationship between CAR algebras and determinantal point processes]{Dynamical relationship between CAR algebras and determinantal point processes: point processes at finite temperature and stochastically positive KMS systems}
\author[R. Sato]{Ryosuke SATO}
\address{Department of Mathematics, Faculty of Science, Hokkaido University, Kita 10, Nishi 8, Kita-ku, Sapporo, Hokkaido, 060-0810, Japan}
\email{r.sato@math.sci.hokudai.ac.jp}

\begin{abstract}
    The aim of this paper is threefold. Firstly, we develop the author's previous work on the dynamical relationship between determinantal point processes and CAR algebras. Secondly, we present a novel application of the theory of stochastic processes associated with KMS states for CAR algebras and their quasi-free states. Lastly, we propose a unified theory of algebraic constructions and analysis of stationary processes on point configuration spaces with respect to determinantal point processes. As a byproduct, we establish an algebraic derivation of a determinantal formula for space-time correlations of stochastic processes, and we analyze several limiting behaviors of these processes.
\end{abstract}

\maketitle

\allowdisplaybreaks{
\section{Introduction}
In this paper, we develop the author's previous work \cite{S24} exploring the dynamical relationship between operator algebras and point processes. Let $\mathfrak{X}$ be a locally compact, second-countable, Hausdorff space. A \emph{point configuration} in $\mathfrak{X}$ is a locally finite collection of its points, and we denote by $\mathcal{C}(\mathfrak{X})$ the set of all such configurations. A probability measure on $\mathcal{C}(\mathfrak{X})$ is called a \emph{point process} on $\mathfrak{X}$. \emph{Correlation functions} play a fundamental role in analysis of point processes. In particular, the \emph{$n$-point} correlation function, defined on $\mathfrak{X}^n$, describes the statistical property of $n$ points from a point process. If the $n$-point correlation function can be expressed as the determinant of a so-called \emph{correlation kernel} defined on $\mathfrak{X}\times \mathfrak{X}$ for all $n\geq 1$, then a point process is called a \emph{determinantal} point process (DPP). 

Points in a DPP can be viewed as having fermionic property, that is, they interact through repulsive forces. From the perspective of operator algebras, a \emph{CAR algebra}, generated by \emph{fermionic} creation and annihilation operators, naturally describes the behavior of fermions in quantum statistical mechanics (see e.g. \cite{BR2}). Thus, developing an operator-algebraic approach to DPPs is a quite natural research direction (see \cite{L02, LM07, Olshanski20,Koshida, S24}). Throughout this paper, we always assume that the state space $\mathfrak{X}$ is a countable set. Then, the space of all continuous functions on $\mathcal{C}(\mathfrak{X})$ is embedded into a CAR algebra. Hence, by the Riesz--Markov--Kakutani theorem, an appropriate linear functional on the CAR algebra, referred to as a \emph{state} in operator algebra theory, gives rise to a point process on $\mathfrak{X}$. Furthermore, the determinantal property of the point process can be characterized in terms of an algebraic property of the corresponding state. More precisely, a \emph{quasi-free} state on the CAR algebra produces a DPP.

The aim of the paper is to extend the \emph{static} relationship between DPPs and quasi-free states on CAR algebras to the dynamical setting. In many contexts, stationary processes with respect to DPPs have attracted significant interest. For instance, such stochastic processes provide a mathematical description of the time evolution of random interacting particle systems in statistical mechanics. In random matrix theory, dynamical models of random matrices naturally give rise to stochastic processes on DPPs (see e.g. \cite{Katori13, K20} for expositions on this field). In this paper, we are especially interested in stochastic processes on DPPs with representation-theoretic origin (see \cite{BO06,BO12,BO13, C18, BB19}, etc).

The essential idea of our work relies heavily upon the pioneering work \cite{KL81} by Klein and Landau, who studied stochastic processes associated with KMS (Kubo--Martin--Schwinger) states on operator algebras. KMS states are known to characterize thermal equilibrium states in quantum statistical mechanics (see e.g. \cite{BR2}). Motivated by constructive quantum field theory, many researchers have made efforts to connect two distinct types of time evolutions, one from quantum mechanics and the other from probability theory. More specially, Klein and Landau invented a framework for constructing stochastic processes via analytic continuation of quantum time evolution under KMS states. In this paper, we apply their framework for CAR algebras and quasi-free states. This forms our central approach to construct and analyze stochastic processes that are stationary with respect to DPPs. 

As mentioned above, KMS states characterize thermal equilibrium states in quantum system \emph{at finite temperature}. Thus, we can naturally regard DPPs associated with quasi-free KMS states on a CAR algebra as interacting particle system at finite temperature. More precisely, the fundamental setting in this paper is as follows: Let $H$ be a self-adjoint operator on $\ell^2(\mathfrak{X})$, which plays the role of a Hamiltonian of a quantum system. Then, there exists a DPP $\mathbb{P}_{H; \beta}$ on $\mathfrak{X}$ whose correlation kernel is given by $e^{-\beta H}(1+e^{-\beta H})^{-1}$. Here, $\beta>0$ is an inverse temperature. In the main result of this paper (see Theorem \ref{thm:dynamical_correlation}), under certain conditions, we construct a stochastic process on the configuration space $\mathcal{C}(\mathfrak{X})$ that is stationary with respect to $\mathbb{P}_{H; \beta}$. Moreover, the resulting stochastic process satisfies a determinantal formula for the space-time correlations. The algebraic nature of our construction also allows us to analyze limiting behavior of stochastic processes when $\beta$ tends to $\infty$ (i.e., the zero temperature limit) or $H$ converges to another self-adjoint operator in the strong resolvent sense (see Section \ref{sec:zero_temp_limit}, \ref{sec:LT_by_strong_res_conv}).

Let us compare our works in this paper and the previous paper \cite{S24}. In the previous one, we considered DPPs satisfying the following properties:
\begin{itemize}
    \item The correlation kernel is given as the spectral projection of a self-adjoint operator.
    \item There exists a representation of a (gauge-invariant) CAR algebra on the $L^2$-space of the DPP, and the vector state corresponding to the constant function $\mathbbm{1}$ coincides with the associated quasi-free state.
\end{itemize}

These conditions will be relaxed in this paper. We consider DPPs whose correlation kernels are given by functional calculus of self-adjoint operators, and we examine the limiting behavior of these kernels toward spectral projections. The second condition states that the $L^2$-space of a DPP admits the GNS (Gelfand--Naimark--Segal)-representation associated with a quasi-free state. Although it induces a Markov semigroup on the $L^2$-space of a DPP from a linear semigroup on a CAR algebra, as discussed in \cite{S24}, we clarify that it is \emph{not} essential for constructing stochastic processes. However, it is related to the Markov property of stochastic processes.

\medskip
This paper is organized as follows: In Section \ref{sec:construction}, we discuss an abstract framework for the construction of stochastic processes. While such constructions typically rely on Kolmogorov's extension theorem, we reformulate the approach in terms of function spaces of stochastic processes, making it more adaptable to our algebraic construction. In Section \ref{sec:DPP_CAR}, we review the static relationship between DPPs and CAR algebras. We will also demonstrate that our framework is applicable for the infinite wedge representation of CAR algebras and cylindric Plancherel process in \cite{BB19}. In Section \ref{sec:stoc_proc}, we examine stochastically positive KMS systems, which were introduced by Klein and Landau \cite{KL81}, and we present our central idea for constructing stochastic processes via an operator algebraic method. Section \ref{sec:density_op_fermion_Fock_sp} constitutes the main part of the paper, where we study stochastically positive KMS systems arising from CAR algebras and quasi-free states. These systems yield stochastic processes on configuration space that are stationary with respect to DPPs. In Sections \ref{sec:6} \ref{sec:stoc_from_op}, we explore concrete examples of DPPs related to orthogonal polynomials. Finally, the appendix addresses the notion of \emph{perfectness} for DPPs, introduced by Olshanski \cite{Olshanski20}.

\medskip
The author gratefully acknowledges Professor Makoto Katori for his suggestive comments on the draft of this paper and throughout the entire duration of this project. The author would also like to thank Dr. Shinji Koshida and Professor Syota Esaki for suggesting a deeper exploration of parts of this work. Finally, the author appreciate Professor Vadim Gorin for informing the reference \cite{BR05}. This work was supported by JSPS Research Fellowship for Young Scientists PD (KAKENHI Grand Number 22J00573).


\section{Abstract construction}\label{sec:construction}
The aim of this section is to rephrase the usual construction of stochastic processes, relying on Kolmogorov's extension theorem, in terms of function spaces. Thus, this section does not contain a new result. Klein and Landau \cite{KL81} discussed such a construction in an operator algebraic setting. We will follow their idea but keep our focus in a probability theoretic setting.

Let $E$ be a compact, second countable space and $\mathcal{B}$ its Borel $\sigma$-algebra. We fix an interval $I\subset \mathbb{R}$ and define $I^n_\leq:=\{(t_1\leq \cdots \leq t_n)\in I^n\}$. Let $\{\mathcal{S}_{f_1, \dots, f_n}\}_{f_1, \dots, f_n\in C(E), n\geq 1}$ be a family of $\mathbb{C}$-valued continuous functions on $I^n_\leq$ labeled by every $n$-tuple of continuous functions $f_1, \dots, f_n$ on $E$ for every $n\geq 1$. We assume the following five conditions:
\begin{description}
    \item[(multi-linearlity)] $\mathcal{S}_{f_1, \dots, f_n}(t_1, \dots, t_n)$ is multi-linear with respect to $f_1, \dots, f_n\in C(E)$ for any $(t_1, \dots, t_n)\in I^n_\leq$,
    \item[(boundedness)] for any $(t_1, \dots, t_n)\in I^n_\leq$, there exists a positive constant $C>0$ such that $|\mathcal{S}_{f_1, \dots, f_n}(t_1, \dots, t_n)|\leq C \|f_1\|\cdots \|f_n\|$, where $\|\cdot\|$ denotes the sup norm,
    \item[(consistency)] $\mathcal{S}_{f_1, \dots, f_n}(t_1, \dots, t_n)=\mathcal{S}_{f_1, \dots, f_{i-1}, f_{i+1}, \dots, f_n}(t_1, \dots, t_{i-1}, t_{i+1}, \dots, t_n)$ holds if $f_i=\mathbbm{1}_E$ for some $i=1, \dots, n$,
    \item[(normalization)] $\mathcal{S}_{\mathbbm{1}_E}(t)\equiv 1$ for any $t\in I$,
    \item[(positivity)] $\mathcal{S}_{f_1, \dots, f_n}(t_1, \dots, t_n)\geq 0$ for any $(t_1, \dots, t_n)\in I^n_\leq$ if $f_1, \dots, f_n\geq 0$.
\end{description}

By the consistency and the normalization, $\mathcal{S}_{\mathbbm{1}_E, \dots, \mathbbm{1}_E}(t_1, \dots, t_n)\equiv 1$ holds for every $n\geq 1$.

\medskip
The goal of this section is the following:
\begin{theorem}\label{thm:construction_process}
    There exists a $E$-valued stochastic process $(X_t)_{t\in I}$ such that
    \[\mathbb{E}[f_1(X_{t_1})\cdots f_n(X_{t_n})]=\mathcal{S}_{f_1, \dots, f_n}(t_1, \dots, t_n)\]
    for any $f_1, \dots, f_n\in C(E)$, $(t_1, \dots, t_n)\in I^n_\leq$, and $n\geq 1$.
\end{theorem}

Let $(E_t, \mathcal{B}_t):=(E, \mathcal{B})$ for each $t\in I$ and $(E^{\underline{t}}, \mathcal{B}^{\otimes \underline{t}}):=(E_{t_1}\times \cdots \times E_{t_n}, \mathcal{B}_{t_1}\otimes \cdots \otimes \mathcal{B}_{t_n})$ for any $\underline{t}=(t_1, \dots, t_n)\in I^n_\leq$. We remark that $\mathcal{B}^{\otimes \underline{t}}$ coincides with the Borel $\sigma$-algebra of $E^{\underline{t}}$ since $E$ is second countable.

\medskip
The following lemma is essential for the proof of Theorem \ref{thm:construction_process}:
\begin{lemma}\label{lem:const_prob_meas}
    For any $\underline{t}=(t_1, \dots, t_n)\in I^n_\leq$ there exists a unique Borel probability measure $\mathbb{P}_{\underline{t}}$ on $(E^{\underline{t}}, \mathcal{B}^{\otimes \underline{t}})$ such that
    \[\int_{E^{\underline{t}}} f_1(x_1)\cdots f_n(x_n) d\mathbb{P}_{\underline{t}}(x_1, \dots, x_n)=\mathcal{S}_{f_1, \dots, f_n}(t_1, \dots, t_n)\]
    for all $f_1, \dots, f_n\in C(E)$.
\end{lemma}
\begin{proof}
    We remark that $\mathbb{P}_{\underline{t}}$ is unique if it exist since $C(E)^{\otimes \underline{t}}$ is dense in $C(E^{\underline{t}})$.

    Let $\rho_{\underline{t}}\colon C(E)^{\otimes \underline{t}}\to \mathbb{C}$ be a linear functional given by
    \[\rho_{\underline{t}}(f_1\otimes \cdots \otimes f_n):=\mathcal{S}_{f_1, \dots, f_n}(t_1, \dots, t_n).\]
    By the multi-linearlity, $\rho_{\underline{t}}$ is well defined. To show the existence of $\mathbb{P}_{\underline{t}}$, we confirm that $\rho_{\underline{t}}$ can be extended to a positive linear form on $C(E^{\underline{t}})$. Let $M_b(E)$ denote the space of bounded Borel measurable functions on $E$. Firstly, $f_1$ can be extended to $M_b(E)$. In fact, since any continuous functions can be decomposed into a sum of four nonnegative continuous functions, we may assume that $f_2, \dots, f_n$ are nonnegative. Thus, $f_1\in C(E)\mapsto \rho_{\underline{t}}(f_1\otimes \cdots \otimes f_n)$ is a positive continuous linear form. By Riesz's representation theorem, it is represented by a integration of a unique Borel (Radon) measure on $(E, \mathcal{B})$, and $\rho_{\underline{t}}$ can be extended to $M_b(E)\otimes C(E)^{\otimes \underline{t}\backslash\{t_1\}}$. By the same argument repeatedly, $\rho_{\underline{t}}$ can be eventually extended on $M_b(E)^{\otimes \underline{t}}$. If $g_1, \dots, g_n\geq 0$ in $M_b(E)$, then we have 
    \[0\leq \rho_{\underline{t}}(g_1\otimes \cdots g_n)\leq C\|g_1\|\cdots \|g_n\|,\]
    where $C>0$ is a constant independent of $g_1, \dots, g_n$.

    Next, we discuss a positivity of $\rho_{\underline{t}}$. Let $\Xi(E)$ denote the space of simple functions, i.e., $\Xi(E)$ is the linear span of characteristic functions of Borel sets in $E$. Obviously, we have $\Xi(E)\subset M_b(E)$. Let $f=\sum_{k=1}^m\alpha_k \mathbbm{1}_{A_{t_1}^{(k)}}\otimes \cdots \otimes \mathbbm{1}_{A_{t_n}^{(k)}}\in \Xi(E)^{\otimes \underline{t}}$ such that $\prod_{j=1}^n A_{t_j}^{(k)}$ ($k=1, \dots m$) are disjoint. We suppose that $f$ is nonnegative on $E^{\underline{t}}$, i.e., $\alpha_k>0$ for each $k=1, \dots, m$. Thus, we have
    \[0\leq \rho_{\underline{t}}(f) \leq \|f\| \rho_{\underline{t}}\left(\mathbbm{1}_E\otimes \cdots \otimes \mathbbm{1}_E\right)\leq C\|f\|.\]
    By the Cauchy--Schwartz inequality, 
    \[|\rho_{\underline{t}}(f)|^2 = |\rho_{\underline{t}}(f\cdot \mathbbm{1})|^2\leq \rho_{\underline{t}}(|f|^2) \rho_{\underline{t}}(\mathbbm{1})\leq C\|f\|^2\] 
    holds for any $f\in \Xi(E)^{\otimes \underline{t}}$. Thus, $\rho_{\underline{t}}$ can be extended to a positive linear form on the norm closure of $\Xi(E)^{\otimes \underline{t}}$.

    Since $E$ is compact, $C(E)$ is contained in the norm closure of $\Xi(E)$. Therefore, we have
    \[C(E^{\underline{t}})=\overline{C(E)^{\otimes \underline{t}}} \subset \overline{\Xi(E)^{\otimes \underline{t}}}.\]
    In particular, $\rho_{\underline{t}}$ provides a positive linear form on $C(E^{\underline{t}})$. By the Riesz representation theorem, we obtain a desired probability measure on $(E^{\underline{t}}, \mathcal{B}^{\otimes \underline{t}})$.
\end{proof}

\begin{proof}[Proof of Theorem \ref{thm:construction_process}]
    Let $\underline{t}\in I^n_\leq $ and $\underline{s}\in I^m_\leq$ such that $m< n$ and $\underline{s}\subset \underline{t}$ as sets. We define $\Theta^{\underline{t}}_{\underline{s}}\colon E^{\underline{t}}\to E^{\underline{s}}$ by $\Theta^{\underline{t}}_{\underline{s}}(x):=(x_s)_{s\in \underline{s}}$ for any $x=(x_t)_{t\in \underline{t}}\in E^{\underline{t}}$. By the consistency condition, the probability measures $\mathbb{P}_{\underline{t}}$ on $(E^{\underline{t}}, \mathcal{B}^{\otimes \underline{t}})$ in Lemma \ref{lem:const_prob_meas} are consistent with respect to $(\Theta^{\underline{t}}_{\underline{s}})_{\underline{s}\subset \underline{t}}$, that is, $\mathbb{P}_{\underline{s}}\circ \Theta^{\underline{t}}_{\underline{s}}=\mathbb{P}_{\underline{t}}$ holds. Therefore, by Kolmogorov's extension theorem, we obtain a probability measure $\mathbb{P}$ on $(E^{I}, \mathcal{B}^{\otimes I})$ such that $\mathbb{P}_{\underline{t}}\circ \Theta_{\underline{t}}=\mathbb{P}$, where $\Theta_{\underline{t}}\colon E^{I}\to E^{\underline{t}}$ is defined by $\Theta_{\underline{t}}(x):=(x_t)_{t\in \underline{t}}$. Finally, we obtain a desired stochastic process $(X_t)_{t\in I}$ by $X_t:=\Theta_{(t)}$ for every $t\in I$.
\end{proof}

\section{Discrete point processes and CAR algebras}\label{sec:DPP_CAR}
In this section, we review the static relationship between point processes on a discrete space and an operator algebra, called a CAR algebra. See \cite{Olshanski20,Koshida,S24}. We will extend the static relationship to a dynamical level in subsequent sections.

\subsection{Discrete point processes}
Let $\mathfrak{X}$ be a countable set with discrete topology, which is referred as a space of particle positions throughout the paper. The set $\mathcal{C}(\mathfrak{X})$ of (\emph{simple}) \emph{point configurations} in $\mathfrak{X}$ is defined as $\mathcal{C}(\mathfrak{X}):=\{0, 1\}^\mathfrak{X}$ equipped with the product topology. Every point configuration $\omega=(\omega_x)_{x\in \mathfrak{X}\in \mathcal{C}(\mathfrak{X})}$ naturally represents a state that $x\in \mathfrak{X}$ is occupied by a particle if $\omega_x=1$; otherwise, $x$ is empty. A Borel probability measure $\mathbb{P}$ on $\mathcal{C}(\mathfrak{X})$ is called a \emph{point process} on $\mathfrak{X}$. For every $n\geq 1$ we define the \emph{$n$-th correlation function} $\rho^{(n)}_\mathbb{P}\colon \mathfrak{X}^n\to \mathbb{R}$ by
\[\rho^{(n)}_\mathbb{P}(x_1, \dots, x_n):=\begin{dcases}\mathbb{P}(\{\omega\in \mathcal{C}(\mathfrak{X})\mid \omega_{x_i}=1 \text{ for }i=1, \dots, n\}) & x_i\text{'s are distinct}, \\ 0 & \text{otherwise}.\end{dcases}\]
If there exists a function $K$ on $\mathfrak{X}\times \mathfrak{X}$ such that
\[\rho^{(n)}_\mathbb{P}(x_1, \dots, x_n)=\det[K(x_i, x_j)]_{i, j=1}^N\]
holds for every $n\geq 1$, then $\mathbb{P}$ is said to be \emph{determinantal}, and $K$ is called a \emph{correlation kernel}. Under a certain condition, a point process is uniquely determined by its correlation functions. In particular, a determinantal point process is uniquely determined by its correlation kernel (see \cite{Lenard, Soshnikov}). However, a correlation kernel is \emph{not} unique for given determinantal point process. 

The concrete examples will be discussed in Section \ref{sec:6}.

\subsection{CAR algebras and gauge-invariant quasi-free states}
Here, we summarize basic facts about CAR algebras and their gauge-invariant quasi-free states, which are intimately related to determinantal point processes.

Let $\mathcal{H}$ be a separable complex Hilbert space. The \emph{CAR algebra } $\mathfrak{A}(\mathcal{H})$ over $\mathcal{H}$ is the universal unital $C^*$-algebra generated by $\{a(h) \mid h\in \mathcal{H}\}$ such that the mapping $h \mapsto a^*(h):=a(h)^*$ is complex linear and the \emph{canonical anti-commutation relation} (CAR) holds true, that is,
\[a(h)a(k)+a(k)a(h)=0, \quad a^*(h)a(k)+a(k)a^*(h)=\langle h, k\rangle 1\quad (h, k\in \mathcal{H}),\]
where the inner product is complex linear in the left argument.

A state $\varphi$ on $\mathfrak{A}(\mathcal{H})$ is said to be (\emph{gauge-invariant}) \emph{quasi-free} if we have
\[\varphi(a^*(h_n)\cdots a^*(h_1)a(k_1)\cdots a(k_m))=\delta_{n, m}\det[\varphi(a^*(h_i)a(k_j))]_{i, j=1}^n\]
for every $h_1, \dots, h_n, k_1, \dots, k_m\in\mathcal{H}$ and $n, m\geq1$. By the Riesz representation theorem, there is a unique bounded linear operator $K$ on $\mathcal{H}$ such that $0\leq K\leq 1$ and $\varphi(a^*(h)a(k))=\langle Kh, k\rangle$ for every $h, k\in\mathcal{H}$. Conversely, if $0\leq K\leq 1$, there exists a quasi-free state $\varphi_K$ such that $\varphi_K(a^*(h)a(k))=\langle Kh, k\rangle$ for any $h, k\in \mathcal{H}$.

For instance, if $K=0$, the associated quasi-free state $\varphi_0$ is called the \emph{Fock state} of $\mathfrak{A}(\mathcal{H})$, and it can be realized as follows: Let $\mathcal{F}_a(\mathcal{H})$ denote the \emph{anti-symmetric Fock space} over $\mathcal{H}$, that is, $\mathcal{F}_a(\mathcal{H}):=\bigoplus_{n=0}^\infty\mathcal{H}^{\wedge n}$, where $\mathcal{H}^{\wedge n}$ is the $n$-th anti-symmetric tensor product of $\mathcal{H}$ for every $n\geq 1$, and $\mathcal{H}^{\wedge 0}$ is the linear span of a distinguished unit vector $\Omega$, called a \emph{vacuum vector}. For any $h\in \mathcal{H}$ the bounded linear operator $\pi_0(a^*(h))$ on $\mathcal{F}_a(\mathcal{H})$ is defined as
\[\pi_0(a^*(h))\Omega:=h, \quad \pi_0(a^*(h))h_1\wedge \cdots \wedge h_n:=h\wedge h_1\wedge \cdots \wedge h_n\]
for any $h_1, \dots, h_n\in \mathcal{H}$ and $n\geq 1$. In addition, we naturally define $\pi_0(a(h)):= \pi_0(a^*(h))^*$. Then, these operators satisfy the CAR. Thus, by the universality of $\mathfrak{A}(\mathcal{H})$, we obtain a $*$-algebra representation $(\pi_0, \mathcal{F}_a(\mathcal{H}))$ of $\mathfrak{A}(\mathcal{H})$, called the \emph{Fock representation}. Moreover, the vector state of $\Omega$ is nothing but the Fock state $\varphi_0$, that is, $\varphi_0(A)=\langle \pi_0(A)\Omega, \Omega\rangle$ for any $A\in \mathfrak{A}(\mathcal{H})$. It is known that the Fock representation is faithful, and hence $\mathfrak{A}(\mathcal{H})$ can be identified its representation $\pi_0(\mathfrak{A}(\mathcal{H}))$. Moreover, $a^*(h)$ and $a(h)$ are called the \emph{creation} and \emph{annihilation} operators associated with $h\in \mathcal{H}$.

Similar to the Fock state case, if a bounded linear operator $K$ on $\mathcal{H}$ satisfies $0\leq K\leq 1$, then there exists a $*$-algebra representation $(\pi_K, \mathcal{H}_K)$ with distinguished unit vector $\Omega_K\in \mathcal{H}_K$ such that $\varphi_K(A)=\langle \pi_K(A)\Omega_K, \Omega_K\rangle$ for any $A\in \mathfrak{A}(\mathcal{H})$. In what follows, we always assume that $\pi_K(\mathfrak{A}(\mathcal{H}))\Omega_K$ is dense in $\mathcal{H}_K$, i.e., $(\pi_K, \mathcal{H}_K, \Omega_K)$ is the GNS-triple associated with $\varphi_K$. Up to unitary equivalence, the GNS-triple is unique.

\subsection{Relationship to discrete point processes}
Let us assume that $\mathcal{H}:=\ell^2(\mathfrak{X})$, where $\mathfrak{X}$ is a countable set as in the previous section. The CAR algebra $\mathfrak{A}(\ell^2(\mathfrak{X}))$ is simply written as $\mathfrak{A}(\mathfrak{X})$. We define a $C^*$-subalgebra $\mathfrak{D}(\mathfrak{X})$ generated by $\{\rho_x:=a^*_xa_x \mid x\in \mathfrak{X}\}$, where $a_x:=a(\delta_x)$ and $(\delta_x)_{x\in \mathfrak{X}}$ is the canonical orthonormal basis for $\ell^2(\mathfrak{X})$,  By the CAR, $\mathfrak{D}(\mathfrak{X})$ is a commutative algebra. Moreover, $\mathfrak{D}(\mathfrak{X})$ is isomorphic to the $C^*$-algebra $C(\mathcal{C}(\mathfrak{X}))$ of continuous functions on $\mathcal{C}(\mathfrak{X})$. Then, $\rho_x$ is assigned to the characteristic function of $\{\omega\in \mathcal{C}(\mathfrak{X})\mid \omega_x=1\}$. See, e.g., \cite[Section 2]{S24} for more details. In what follows, we freely identify $\mathcal{D}(\mathfrak{X})$ and $C(\mathcal{C}(\mathfrak{X}))$.

Let $\varphi$ be a state on $\mathfrak{A}(\mathfrak{X})$. Since $\mathfrak{D}(\mathfrak{X})\cong C(\mathcal{C}(\mathfrak{X}))$, by the Riesz--Markov--Kakutani theorem, the restriction $\varphi|_{\mathfrak{D}(\mathfrak{X})}$ gives rise to a point process on $\mathfrak{X}$, denoted by $\mathbb{P}_\varphi$. Furthermore, by the definition of correlation functions, for every $n\geq 1$ and $x_1, \dots, x_n\in \mathfrak{X}$ we have
\begin{equation}\label{eq:corr_state}
    \rho^{(n)}_{\mathbb{P}_\varphi}(x_1, \dots, x_n)=\varphi(a^*_{x_n}\cdots a^*_{x_1} a_{x_1}\cdots a_{x_n}).
\end{equation}
In fact, $a_{x_n}^*\cdots a_{x_1}^*a_{x_1}\cdots a_{x_n}= \rho_{x_1}\cdots \rho_{x_n}$ holds if $x_1, \dots ,x_n$ are distinct; otherwise, it is equal to 0. In particular, if $\varphi$ is a quasi-free state, then $\mathbb{P}_\varphi$ is determinantal whose correlation kernel is given as $K(x, y):=\varphi(a_x^*a_y)$. Moreover, we have $K(x, y)=\langle Ke_x, e_y\rangle$ for any $x, y\in \mathfrak{X}$ if $\varphi=\varphi_K$ for some positive contraction operator $K$ on $\ell^2(\mathfrak{X})$.

\begin{remark}
    Any quasi-free states on $\mathfrak{A}(\mathfrak{X})$ produces a determinantal point processes on $\mathfrak{X}$. However, this correspondence is not one-to-one.
\end{remark}

Let $\varphi$ be a state on $\mathfrak{A}(\mathfrak{X})$ and $(\pi_\varphi, \mathcal{H}_\varphi, \Omega_\varphi)$ the associated GNS-triple. We obtain an isometrically embedding $L^2(\mathcal{C}(\mathfrak{X}), \mathbb{P}) \hookrightarrow\mathcal{H}_\varphi$ by $f\in C(\mathcal{C}(\mathfrak{X}))\mapsto \pi_K(f)\Omega_\varphi$. In general, the range of this embedding does not coincides with $\mathcal{H}_\varphi$. In a sense, $\mathfrak{A}(\mathfrak{X})$ is too large, and hence, so is $\mathcal{H}_\varphi$. 

Let $\mathfrak{I}(\mathfrak{X})$ denote the $C^*$-subalgebra of $\mathfrak{A}(\mathfrak{X})$ generated by $\{a^*(h)a(k)\mid h, k\in \ell^2(\mathfrak{X})\}$. In other words, $\mathfrak{I}(\mathfrak{X})$ consists of elements fixed under the \emph{gauge action} $\gamma\colon \mathbb{T}\curvearrowright\mathfrak{A}(\mathfrak{X})$ defined by $\gamma_z(a(h)):=a(zh)=\bar z a(h)$ for any $h\in \ell^2(\mathfrak{X})$ and $z\in \mathbb{T}$. Here, $\mathbb{T}:=\{z\in\mathbb{C}\mid |z|=1\}$ is the torus group. Namely, we have 
\[\mathfrak{I}(\mathfrak{X})=\{A\in \mathfrak{A}(\mathfrak{X})\mid \gamma_z(A)=A \text{ for all } z\in \mathbb{T}\}.\]
For that reason, $\mathfrak{I}(\mathfrak{X})$ is called the \emph{gauge-invariant} CAR (GICAR) algebra.

By definition, we have $\mathfrak{D}(\mathfrak{X})\subset \mathfrak{I}(\mathfrak{X})$. Thus, $\pi_\varphi(\mathfrak{D}(\mathfrak{X}))\Omega_\varphi \subset \pi_\varphi(\mathfrak{I}(\mathfrak{X}))\Omega_\varphi$. However, even for $\mathfrak{I}(\mathfrak{X})$, the embedding of $L^2(\mathcal{C}(\mathfrak{X}), \mathbb{P}_\varphi)$ does not generally coincides with $\overline{\pi_\varphi(\mathfrak{I}(\mathfrak{X}))\Omega_\varphi}$. On the other hand, if the embedding of $L^2(\mathcal{C}(\mathfrak{X}), \mathbb{P}_\varphi)$ coincides with $\overline{\pi_\varphi(\mathfrak{I}(\mathfrak{X}))\Omega_\varphi}$, i.e., if the GNS-representation space for $\varphi|_{\mathfrak{I}(\mathfrak{X})}$ can be identified with $L^2(\mathcal{C}(\mathfrak{X}), \mathbb{P}_\varphi)$, then many properties of $\mathbb{P}_\varphi$ can be characterized by $\varphi$. This research policy was proposed by Olshanski \cite{Olshanski20} when $\varphi$ is a quasi-free state, i.e., $\mathbb{P}_\varphi$ is determinantal. Moreover, Koshida \cite{Koshida} studied the case of Pfaffian point processes and quasi-free states on \emph{self-dual} CAR algebras.

In this paper, we investigate the dynamical relationship between point processes and states on CAR algebras. However, as well as in the static relation, the size of $\pi_\varphi(\mathfrak{D}(\mathfrak{X}))\Omega_\varphi$ in the GNS-representation space is important for us, too. See Corollary \ref{cor:suf_cond_Markov}, Propositions \ref{prop:Markov_property}, \ref{prop:Markov_2}.

\subsection{Systems in Section \ref{sec:construction} from the infinite wedge representation}
Our fundamental idea for constructing stochastic processes on $\mathcal{C}(\mathfrak{X})$ is to produce systems in Section \ref{sec:construction} from the CAR algebra $\mathfrak{A}(\mathfrak{X})$. In the subsequent sections, we will discuss such a system manufactured from quasi-free states on $\mathfrak{A}(\mathfrak{X})$ satisfying the so-called KMS condition. Before that, we point out that a similar system has also been discussed for \emph{periodic Schur measures} and the infinite wedge representation of a CAR algebra. See \cite{BB19} for more details.

First, we prepare some necessary facts about symmetric functions and their realization via the infinite wedge space. See \cite{M, Kac,Okounkov01} for more details.

Let $\mathsf{Sym}$ be the $\mathbb{C}$-algebra of symmetric functions. This has two kinds of algebraically independent generators, called the \emph{power sum} symmetric functions $\{p_k\}_{k=1}^\infty$ and the \emph{complete homogeneous} symmetric functions $\{h_k\}_{k=1}^\infty$. They are formally defined by 
\[p_k(x):=\sum_{i=1}^\infty x_i^k, \quad h_k(x):=\sum_{i_1\leq \cdots \leq i_k}x_{i_1}\cdots x_{i_k}.\]

A sequence $\lambda=(\lambda_1\geq \lambda_2\geq \cdots)\in \mathbb{Z}_{\geq 0}^\infty$ is called a \emph{partition} of $n$ if $|\lambda|:=\lambda_1+\lambda_2+\cdots = n$ holds. The integer $\ell(\lambda)$ such that $\lambda_{\ell(\lambda)}\neq 0$ and $\lambda_{\ell(\lambda)+1}=0$ is called the \emph{length} of $\lambda$. As the convention, we set $\emptyset :=(0, 0, \dots)$ and $\ell(\emptyset)=0$. Let $\mathcal{P}$ denote the set of all partitions of integers. 

For every $\lambda\in \mathcal{P}$, the \emph{Schur} function $s_\lambda$ is defined by 
\[s_\lambda:=\det[h_{\lambda_i-i+j}]_{i, j=1}^{\ell(\lambda)},\]
where $h_0:=1$ and $h_k:=0$ if $k<0$. Moreover, for any $\lambda, \mu\in \mathcal{P}$, the \emph{skew Schur} function $s_{\lambda/\mu}$ is defined by 
\[ s_{\lambda/\mu}:=\begin{dcases} \det[h_{\lambda_i-\mu_j-i+j}]_{i, j=1}^{\ell(\lambda)} & \lambda \supset \mu, \\ 0 & \text{otherwise},\end{dcases}\]
where we write $\lambda\supset \mu$ if $\lambda_i\geq \mu_i$ for every $i=1, 2, \dots$. We remark that $\{s_\lambda\}_{\lambda\in \mathcal{P}}$ form a linear basis of $\mathsf{Sym}$. Moreover, every skew Schur function is expanded to a sum of Schur functions with non-negative integer coefficients.

A $\mathbb{C}$-algebra homomorphism $\rho\colon \mathsf{Sym}\to \mathbb{C}$ is called a \emph{specialization}. Following the convention, we write $f(\rho)$ instead of $\rho(f)$ for any $f\in \mathsf{Sym}$. For arbitrary specialization $\rho$, the Miwa variables 
\[t_k:=\frac{1}{k}p_k(\rho) \quad (k=1, 2, \dots)\] 
uniquely determine the specialization $\rho$ since $\{p_k\}_{k=1}^\infty$ are algebraically independent. A specialization $\rho$ is said to be \emph{Schur-positive} if $s_\lambda(\rho)\geq 0$ for every $\lambda\in \mathcal{P}$. It implies that $s_{\lambda/\mu}(\rho)\geq 0$ for every $\lambda, \mu \in \mathcal{P}$. For instance, for any $\gamma>0$ the following specialization $\mathrm{ex}_\gamma$ is Schur-positive:
\[t_k:=\begin{dcases} \gamma & k=1, \\ 0 & k\geq 2.
\end{dcases}\]

Let $\mathcal{F}$ be a complex Hilbert space with orthonormal basis $\{v_\lambda\}_{\lambda\in \mathcal{P}}$. We call $\mathcal{F}$ the \emph{infinite wedge space}\footnote{More precisely, $\mathcal{F}$ is a \emph{charge zero} subspace of the infinite wedge space, and it realized by the infinite wedge vectors. See \cite{Kac,Okounkov01} for more details.}. For a specialization $\rho\colon \mathsf{Sym}\to \mathbb{C}$, we define the \emph{vertex operators} $\Gamma_\pm(\rho)$ on $\mathcal{F}$ by 
\[\Gamma_-(\rho)v_\mu:=\sum_{\lambda\in \mathcal{P}}s_{\lambda/\mu}(\rho) v_\lambda, \quad \Gamma_+(\rho)v_\lambda := \sum_{\mu\in \mathcal{P}}s_{\lambda/\mu}(\rho)v_\mu.\]
Here, the right-hand side of the second equality is a finite sum. By the skew Cauchy identity,
\[\sum_{\lambda\in \mathcal{P}}s_{\lambda/\mu}(\rho)^2=Z_{\rho, \rho} \sum_{\kappa\in \mathcal{P}}s_{\mu/\kappa}(\rho)^2<\infty\]
holds when $Z_{\rho, \rho}<\infty$, where 
\[Z_{\rho, \rho'}:=\sum_{\lambda\in \mathcal{P}}s_\lambda(\rho)s_\lambda(\rho')=\exp\left(\sum_{k=1}^\infty kt_k t'_k\right)\]
for arbitrary specializations $\rho$ and $\rho'$ with Miwa variables $(t_k)_{k=1}^\infty$ and $(t'_k)_{k=1}^\infty$. Thus, $\Gamma_-(\rho)$ and $\Gamma_+(\rho)$ are well defined. By definition, we have $s_{\lambda/\mu}(\rho)=\langle \Gamma_-(\rho)v_\mu, v_\lambda\rangle = \langle \Gamma_+(\rho)v_\lambda, v_\mu\rangle$. 

For any two specializations $\rho, \rho'$, the skew Cauchy identity also implies 
\begin{equation}\label{eq:vertex_op}\Gamma_+(\rho)\Gamma_-(\rho')=Z_{\rho, \rho'}\Gamma_-(\rho')\Gamma_+(\rho), \quad \Gamma_\pm(\rho)\Gamma_\pm(\rho')=\Gamma_\pm(\rho+\rho').
\end{equation}
Here, $\rho+\rho'$ is the specialization whose Miwa variables is given by $(t_k+t'_k)_{k=1}^\infty$.

\medskip
Let $\mathfrak{S}(\lambda):=\{\lambda_i-i+1/2\}_{i=1}^\infty\subset \mathbb{Z}':=\mathbb{Z}+1/2$ for any $\lambda\in \mathcal{P}$. Hence, every partition $\lambda\in \mathcal{P}$ gives the point configuration $\mathfrak{S}(\lambda)\in \mathcal{C}(\mathbb{Z}')$ in $\mathbb{Z}'$.

A representation of $\mathfrak{A}(\mathbb{Z}')$ on $\mathcal{F}$ is given as follows: For every $x\in \mathbb{Z}'$, the creation operators $a^*_x$ acts on $\mathfrak{F}$ by 
\[a^*_xv_\lambda:=\begin{dcases}(-1)^i v_\nu & \lambda_i-i+1/2 < x < \lambda_{i+1}-i-1/2 \text{ for some }i\geq 1, \\ 0 & \text{ otherwise},\end{dcases}\]
where $\nu\in \mathcal{P}$ satisfies $\mathfrak{S}(\nu)=\mathfrak{S}(\lambda)\cup \{x\}$. Similarly, the annihilation operator $a_x$ acts on $\mathfrak{F}$ by 
\[a_xv_\lambda:=\begin{dcases}(-1)^{i-1} v_\mu & x=\lambda_i-i+1/2 \text{ for some }i=1, 2, \dots, \\ 0 & \text{ otherwise},\end{dcases}\]
where $\mu\in \mathcal{P}$ satisfies $\mathfrak{S}(\lambda)=\mathfrak{S}(\mu)\backslash\{x\}$. Since the above operators satisfy the CAR, the representation of $\mathfrak{A}(\mathbb{Z}')$ on $\mathcal{F}$ is well defined\footnote{In the usual convention, the creation and annihilation operator on $\mathcal{F}$ are denoted by $\psi_x$ and $\psi^*_x$ ($x\in \mathbb{Z}'$).}. By definition, for any $x\in \mathbb{Z}'$ and $\lambda\in \mathcal{P}$ we have \
\begin{equation}\label{eq:num_op_on_inf_wedge}
    \rho_xv_\lambda=\begin{dcases}v_\lambda & x \in \mathfrak{S}(\lambda), \\ 0 & \text{otherwise}.\end{dcases}
\end{equation}

\medskip
We define the \emph{energy operator} $H$ on $\mathcal{F}$ by $Hv_\lambda:=|\lambda|v_\lambda$ ($\lambda\in \mathcal{P}$). We have
\begin{equation}\label{eq:vertex_energy}
    \Gamma_\pm(\rho)u^H=u^H\Gamma_\pm(u^\pm\cdot \rho),
\end{equation}
where a new specialization $u\cdot \rho$ is defined by $p_k(u\cdot \rho)=u^k p_k(\rho)$ for every $k\geq 1$. 

\medskip
Following \cite{BB19}, for any $\vartheta, t\geq 0$ we define
\[T^\vartheta_t:=e^{\vartheta^2(e^{-t}-1)}\Gamma_-(\mathrm{ex}_{\vartheta(1-e^{-t})})e^{-tH}\Gamma_+(\mathrm{ex}_{\vartheta(1-e^{-t})}).\]
By Equations \eqref{eq:vertex_op}, \eqref{eq:vertex_energy}, we have
\[T^\vartheta_t T^\vartheta_s=T^\vartheta_{t+s}\quad (t, s\geq0).\]
Moreover, for every $\lambda, \mu\in \mathcal{P}$ we have 
\begin{equation}\label{eq:matrix_entries_transition_op}
    \langle T^\vartheta_t v_\mu, v_\lambda\rangle = e^{\vartheta^2(e^{-t}-1)}\sum_{\kappa \in \mathcal{P}}e^{-t |\kappa|} s_{\lambda/\kappa}(\mathrm{ex}_{\vartheta(1-e^{-t})}) s_{\mu/\kappa}(\mathrm{ex}_{\vartheta(1-e^{-t})}).
\end{equation}
Thus, $\langle T^\vartheta_tv_\mu, v_\lambda\rangle\geq 0$ since $\mathrm{ex}_{\vartheta(1-e^{-t})}$ is Schur-positive.

For any $n\geq 1$ and $A_1, \dots, A_n\in \mathfrak{A}(\mathbb{Z}')$ we define $\mathcal{S}_{A_1, \dots, A_n}\in C([0, \beta]^n_\leq)$ by 
\[\mathcal{S}_{A_1, \dots, A_n}(t_1, \dots, t_n):=\frac{1}{\mathrm{Tr}(T^\vartheta_\beta)}\mathrm{Tr}(T^\vartheta_{t_1}A_1T^\vartheta_{t_2-t_1}A_2\cdots T^\vartheta_{t_n-t_{n-1}}A_n T^\vartheta_{\beta-t_n}).\]
Since $\mathfrak{D}(\mathbb{Z}')\cong C(\mathcal{C}(\mathbb{Z}'))$, we can obtain $\mathcal{S}_{f_1, \dots, f_n} \in C([0, \beta]^n_\leq)$ for every $f_1, \dots, f_n\in C(\mathcal{C}(\mathbb{Z}'))$. By definition, the multi-linearlity, the consistency, and the normalization conditions holds true. 

The following is a consequence of Equations \eqref{eq:num_op_on_inf_wedge}, \eqref{eq:matrix_entries_transition_op}. See also \cite[Section 6]{BB19}.
\begin{lemma}
    For any $n\geq 1$, $f_1, \dots, f_n\in C(\mathcal{C}(\mathbb{Z}'))$, and $(t_1, \dots, t_n)\in [0, \beta]^n_\leq$ we have 
    \begin{align*}
        \mathcal{S}_{f_1, \dots, f_n}(t_1, \dots, t_n)
        &=\frac{1}{\mathrm{Tr}(T^\vartheta_\beta)}\sum_{\lambda^0, \dots, \lambda^n \in \mathcal{P}}\prod_{j=1}^{n+1} f_j(\mathfrak{S}(\lambda^j)) \langle T^\vartheta_{t_i-t_{i-1}} v_{\lambda^j}, v_{\lambda^{j-1}}\rangle,
    \end{align*}
    where $t_0=0$, $t_{n+1}=\beta$, and $\lambda^{n+1}=\lambda^0$. In particular, the positivity and the normalization conditions holds.
\end{lemma}

Therefore, by Theorem \ref{thm:construction_process}, we obtain a stochastic process $(X_t)_{t\in [0, \beta]}$ on $\mathcal{P}$($\subset \mathcal{C}(\mathbb{Z}')$) such that for any $t_0=0\leq t_1\leq \cdots \leq t_n\leq t_{n+1}=\beta$ and $\lambda_0,  \dots, \lambda_{n+1} \in \mathcal{P}$ such that 
\[\mathbb{P}[X_0=\lambda_0, X_{t_1}=\lambda_1, \dots, X_\beta=\lambda_{n+1}]=\frac{1}{\mathrm{Tr}(T_\beta^\vartheta)}\prod_{j=1}^{n+1}\langle T^\vartheta_{t_j-t_{j-1}}v_{\lambda^j}, v_{\lambda^{j+1}}\rangle.\]
This stochastic process is named the \emph{stationary cylindric Plancherel process} of period $\beta$ and intensity $\vartheta$. See \cite{BB19} fore more details.

\medskip
Let us recall that our construction of stochastic processes essentially consists of two steps (and it is the same as the usual argument in probability theory). The first step (see Lemma \ref{lem:const_prob_meas}) is to give finite-dimensional distributions of a desired stochastic process. In the second step, we apply Kolmogorov's extension theorem. Here, the set $\mathcal{P}$ of integer partitions is countable, and so is $\mathcal{P}^n$ for every $n\geq 1$. Thus, a system in Section \ref{sec:construction} is not necessarily required, that is, we can directly give the finite-dimensional distributions of $(X_t)_{t\in [0, \beta]}$. However, as will be seen later, the observation in this section suggests a unified algebraic approach to construct and analyze stochastic processes on a point configuration space. The common idea is as follows:
\begin{itemize}
    \item a CAR algebra contains a commutative $C^*$-subalgebra isomorphic to the algebra of continuous functions on a point configuration space.
    \item a state on a CAR algebra gives a point process through its restriction to the commutative $C^*$-subalgebra (the static relationship between CAR algebras and point processes).
    \item a ``time evolution'' on a CAR algebra gives rise to a system in Section \ref{sec:construction}, and hence, it produces a stochastic process on a point configuration space.
\end{itemize}

\section{Stochastically positive KMS systems and associated stochastic processes}\label{sec:stoc_proc}
In this section, we introduce the notion of \emph{stochastically positive KMS systems} due to Klein and Landau \cite{KL81}, which play an essential role in producing systems in Section \ref{sec:construction} and in constructing and analyzing stochastic processes. In this section, we discuss an abstract setting of $C^*$-algebras and then return to the setting of the CAR algebra $\mathfrak{A}(\mathfrak{X})$ in the next section.

\subsection{Stochastically positive KMS systems}
Let $\mathfrak{A}$ be a separable unital $C^*$-algebra. We assume that $\alpha\colon \mathbb{R}\curvearrowright \mathfrak{A}$ is an $\mathbb{R}$-flow, i.e., $t\mapsto \alpha_t$ is a group homomorphism to the group of $*$-algebra automorphisms on $\mathfrak{A}$ such that $t\mapsto \alpha_t(A)$ is continuous for any $A\in \mathfrak{A}$. 

A state $\varphi$ on $\mathfrak{A}$ is called a \emph{$\alpha$-KMS state at $\beta>0$} if for every $A, B\in \mathfrak{A}$ there exists a bounded continuous function $F_{A, B}$ on the strip region $\{z\in \mathbb{C} \mid 0\leq \mathrm{Im}(z)\leq \beta\}$, which is analytic in $\{z\in \mathbb{C} \mid 0 < \mathrm{Im}(z) <\beta\}$, such that
\[F_{A, B}(t)=\varphi(A\alpha_t(B)), \quad F_{A, B}(t+\mathrm{i}\beta)=\varphi(\alpha_t(B)A) \quad (t\in \mathbb{R}).\]
If $\beta\to \infty$, we obtain a \emph{$\alpha$-ground state}. More precisely, a state $\varphi$ on $\mathfrak{A}$ is called a $\alpha$-ground state if for any $A, B\in \mathfrak{A}$ there exists a continuous function $F_{A, B}$ on the region $\{z\in \mathbb{C}\mid \mathrm{Im}(z)\geq 0\}$, which is analytic in $\{z\in \mathbb{C}\mid \mathrm{Im}(z)>0\}$, such that $F_{A, B}(t)=\varphi(A\alpha_t(B))$ holds for all $t\in \mathbb{R}$. In this sense, we can refer to a ground state as a KMS state at $\beta=\infty$. See \cite[Proposition 5.3.19]{BR2} for more details.

In the context of quantum statistical mechanics, an $\mathbb{R}$-flow on a $C^*$-algebra is regarded as a time evolution of a quantum system. Moreover, a KMS state describes a thermal equilibrium state, like Gibbs states. Roughly speaking, the above definition says that under a KMS state, quantum time evolution along the real axis can be continued to the strip region. The extendability of time evolution is quite essential for our purpose. In fact, we will later observe that stochastic dynamics would appear as the time evolution along the imaginary axis.

\medskip
We remark that the KMS condition is extended to multiple time. Let $I_\beta(n):=\bigcup_{k=0}^n I_\beta(n)_k$ for every $n\geq 1$, where $I_\beta(n)_k$ is the set of all $n$ complex numbers $(z_1, \dots, z_n)\in \mathbb{C}^n$ such that 
\[-\beta/2< \mathrm{Im}(z_1)< \dots < \mathrm{Im}(z_k)< 0 < \mathrm{Im}(z_{k+1})<\cdots < \mathrm{Im}(z_n)< \beta/2.\]
Let $\varphi$ be a $\alpha$-KMS state at $\beta\in \mathbb{R}_{>0}$ and $(\pi_\varphi, \mathcal{H}_\varphi, \Omega_\varphi)$ denote the associated GNS-triple. Namely, $(\pi_\varphi, \mathcal{H}_\varphi)$ is a representation of $\mathfrak{A}$ and $\Omega_\varphi \in \mathcal{H}_\varphi$ is a unit vector such that $\varphi(A)=\langle \pi_\varphi(A)\Omega_\varphi, \Omega_\varphi\rangle$ for all $A\in \mathfrak{A}$, and $\pi_\varphi(\mathfrak{A})\Omega_\varphi$ is dense in $\mathcal{H}_\varphi$. Since $\varphi$ is $\alpha$-invariant (see \cite[Proposition 5.3.3]{BR2}), there exists a unique self-adjoint operator $H$ on $\mathcal{H}_\varphi$ such that $e^{\mathrm{i}t H}\pi_\varphi(A)\Omega_\varphi=\pi_\varphi(\alpha_t(A))\Omega_\varphi$ for all $A\in \mathfrak{A}$ and $t\in \mathbb{R}$. By \cite[Theorem 3.2]{Araki73} (see also \cite[Theorem 2.1]{KL81}), for any $A_1, \dots, A_n\in \mathfrak{A}$, the vector-valued function 
\[(z_1, \dots, z_n) \mapsto e^{\mathrm{i}z_1 H}\pi_\varphi(A_1)e^{\mathrm{i}(z_2-z_1)H}\pi_\varphi(A_2)\cdots e^{\mathrm{i}(z_n-z_{n-1})H}\pi_\varphi(A_n)\Omega_\varphi\]
is analytic in $I_\beta(n)_0$ and continuous, uniformly bounded by $\|A_1\|\cdots \|A_n\|$ on $\overline{I_\beta(n)_0}$. Thus, the following function $\mathcal{W}_{A_1, \dots, A_n}$ on $\overline{I_\beta(n)}$ is well defined:
\begin{align}\label{eq:W}
    &\mathcal{W}_{A_1, \dots, A_n}(z_1, \dots, z_n) \nonumber\\
    &:=\langle e^{\mathrm{i}z_{k+1} H}\pi_\varphi(A_{k+1})\cdots e^{\mathrm{i}(z_n-z_{n-1})H}\pi_\varphi(A_n)\Omega_\varphi, e^{-\mathrm{i}z_k H}\pi_\varphi(A_k^*)\cdots e^{-\mathrm{i}(z_2-z_1)H}\pi_\varphi(A_1^*)\Omega_\varphi\rangle
\end{align}
for any $(z_1, \dots, z_n)\in \overline{I_\beta(n)_k}$. Moreover, $\mathcal{W}_{A_1, \dots, A_n}$ is analytic in $I_\beta(n)$ and continuous, uniformly bounded by $\|A_1\|\cdots \|A_n\|$ on $\overline{I_\beta(n)}$. By definition, for any $(t_1, \dots, t_n)\in \mathbb{R}^n$, we have
\[\mathcal{W}_{A_1, \dots, A_n}(t_1, \dots, t_n)=\varphi(\alpha_{t_1}(A_1)\cdots \alpha_{t_n}(A_n)).\]

Let us recall that $A\in \mathfrak{A}$ is said to be \emph{entire analytic for $\alpha$} if the mapping $t\in \mathbb{R}\mapsto \alpha_t(A)\in \mathfrak{A}$ extends to the whole complex region $\mathbb{C}$, and $\omega(\alpha_z(A))$ is analytic in $z\in \mathbb{C}$ for any $\omega\in \mathfrak{A}^*$. If $A_1, \dots, A_n\in \mathfrak{A}$ be entire analytic for $\alpha$, then, we have
\begin{equation}\label{eq:analytic}
    \mathcal{W}_{A_1, \dots, A_n}(z_1, \dots, z_n) = \varphi(\alpha_{z_1}(A_1)\cdots \alpha_{z_n}(A_n))
\end{equation}
holds for any $(z_1, \dots, z_n)\in I_n(\beta)$. 

\begin{remark}\label{rem:approximation}
    It is known that the set of entire analytic elements for $\alpha$ in the unit ball of $\mathfrak{A}$ is weakly dense in the unit ball of $\mathfrak{A}$ (see \cite[Proposition 2.5.22]{BR1}). Moreover, by the Hahn--Banach separation theorem, the former is norm dense in the unit ball of $\mathfrak{A}$. Thus, for any $A_1, \dots, A_n\in \mathfrak{A}$, there exist sequences $(A_{i, k})_{k=1}^\infty$ of entire analytic elements for $\alpha$ such that $\|A_{i, k}\|\leq \|A_i\|$ and $\lim_{k\to \infty}A_{i, k}=A_i$ in the norm topology. By \cite[Theorem 2.2]{KL81} and Equation \eqref{eq:analytic}, we have 
    \[\mathcal{W}_{A_1, \dots, A_n}(z_1, \dots, z_n)=\lim_{k\to\infty} \varphi(\alpha_{z_1}(A_{1, k})\cdots \alpha_{z_n}(A_{n, k}))\]
    for every $(z_1, \dots, z_n)\in I_n(\beta)$.
\end{remark}

By the KMS condition, for any $A, B\in \mathfrak{A}$, we have
\begin{equation}\label{eq:beta_half}
    \mathcal{W}_{A, B}(-\mathrm{i}\beta/2, \mathrm{i}\beta/2)=\varphi(BA).
\end{equation}

We assume that $\varphi$ is a $\alpha$-ground state. By \cite[Proposition 5.3.19]{BR2}, the associated self-adjoint operator $H$ on $\mathcal{H}_\varphi$ is positive. Thus, $e^{\mathrm{i}z H}$ is bounded and $\|e^{\mathrm{i}zH}\|\leq 1$ for any $z\in \mathbb{C}$ with $\mathrm{Im}(z)>0$. Thus, for any $A_1, \dots, A_n\in \mathfrak{A}$,
\[\mathcal{W}_{A_1, \dots, A_n}(z_1, \dots, z_n):=\langle e^{\mathrm{i}z_1 H}\pi_\varphi(A_1)e^{\mathrm{i}(z_2-z_1) H}\pi_\varphi(A_2)\cdots e^{\mathrm{i}(z_n-z_{n-1})}\pi_\varphi(A_n)\Omega_\varphi, \Omega_\varphi\rangle\]
is well define and continuous for any $(z_1, \dots, z_n)\in \overline{I_\infty(n)_0}$. Moreover, $\mathcal{W}_{A_1, \dots, A_n}$ is bounded above by $\|A_1\|\cdots \|A_n\|$. If $A_1, \dots, A_n$ are entire analytic, then Equation \eqref{eq:analytic} holds true.

Let $\beta\in \mathbb{R}_{>0}\cup \{\infty\}$ and $[-\beta/2, \beta/2]^n_\leq :=\{(t_1\leq \dots\leq t_n)\in [-\beta/2, \beta/2]^n\}$ for every $n\geq 1$. We define
\begin{equation}\label{eq:S_func_from_KMS}
    \mathcal{S}_{A_1, \dots, A_n}(t_1, \dots, t_n):=\mathcal{W}_{A_1, \dots, A_n}(\mathrm{i}t_1, \dots, \mathrm{i}t_n)
\end{equation}
for every $(t_1, \dots, t_n)\in [-\beta/2, \beta/2]^n_\leq$ and $A_1, \dots, A_n\in \mathfrak{A}$.

\medskip
Let $\mathfrak{D}$ be a commutative $C^*$-algebra with common unit $1\in \mathfrak{A}$. The following notion was introduced by Klein and Landau \cite{KL81}.

\begin{definition}
    A $\alpha$-KMS state $\varphi$ on $\mathfrak{A}$ at $\beta$ is said to be \emph{stochastically positive} with respect to a commutative $C^*$-subalgebra $\mathfrak{D}$ if
    $\mathcal{S}_{A_1, \dots, A_n}(t_1, \cdots, t_n)\geq 0$ holds for any $A_1, \dots, A_n\in \mathfrak{D}_+$, $(t_1, \cdots, t_n)\in [-\beta/2, \beta/2]^n_\leq$, and $n\geq 1$. In this case, the quadruplet $(\mathfrak{A}, \mathfrak{D}, \alpha, \varphi)$ is called a \emph{stochastically positive KMS system}.
\end{definition}

\begin{remark}
    In this paper, we have slightly relaxed the condition for stochastically positive KMS systems compared to the original paper \cite{KL81}. However, it is sufficient for our purpose, which is to construct stochastic processes.
\end{remark}

We give a few properties of $\{\mathcal{S}_{A_1, \dots, A_n}\}_{A_1, \dots, A_n\in \mathfrak{A}}$, which are related to properties of the associated stochastic process given in Theorem \ref{thm:stochastic_process}.
\begin{lemma}\label{lem:properties_S}
    Let $A_1, \dots, A_n\in \mathfrak{A}$ and $(t_1, \dots, t_n)\in [-\beta/2, \beta/2]^n_\leq$. The following holds true:
    \begin{enumerate}
        \item (stationarity) If $s\in \mathbb{R}$ and $(t_{k+1}+s-\beta, \dots, t_n+s-\beta, t_1+s, \dots, t_k+s)\in [-\beta/2, \beta/2]^n_\leq $ holds, then we have
        \[\mathcal{S}_{A_{k+1}, \dots, A_n, A_1, \dots, A_k}(t_{k+1}+s-\beta, \dots, t_n+s-\beta, t_1+s, \dots, t_n+s)=\mathcal{S}_{A_1, \dots, A_n}(t_1, \dots, t_n).\]
        \item (symmetry) $\overline{\mathcal{S}_{A_n^*, \dots, A_1^*}(-t_n, \dots, -t_1)}=\mathcal{S}_{A_1, \dots, A_n}(t_1, \dots, t_n)$ holds.
        \item (OS(Osterwalder--Schrader) positivity) Let $t_1\geq 0$. For any $m\geq 1$, $A_{k, 1}, \dots, A_{k, n}\in \mathfrak{A}$ ($k=1, \dots, m$), and $c_1, \dots, c_m\in \mathbb{C}$, we have 
        \[\sum_{k, l=1}^m \overline{c_k}c_l\mathcal{S}_{A_{k, n}^*, \dots, A_{k, 1}^*, A_{l, 1}, \dots, A_{l, n}}(-t_n, \dots, -t_1, t_1, \dots, t_n)\geq 0.\]
    \end{enumerate}
\end{lemma}
\begin{proof}
    By Remark \ref{rem:approximation}, we may assume that $A_1, \dots, A_N \in \mathfrak{A}$ are entire analytic. Since $\varphi$ is a $\alpha$-KMS state at $\beta$ (see \cite[Proposition 5.3.7]{BR2}), we have 
    \begin{align*}
        &\mathcal{S}_{A_{k+1}, \dots, A_n, A_1, \dots, A_k}(t_{k+1}+s-\beta, \dots, t_n+s-\beta, t_1+s, \dots, t_n+s)\\
        &=\varphi(\alpha_{-\mathrm{i}\beta}(\alpha_{\mathrm{i}(t_{k+1}+s)}(A_{k+1})\cdots \alpha_{\mathrm{i}(t_n+s)}(A_n))\alpha_{\mathrm{i}(t_1+s)}(A_1)\cdots \alpha_{\mathrm{i}(t_k+s)}(A_k))\\
        &= \varphi(\alpha_{\mathrm{i}(t_1+s)}(A_1)\cdots \alpha_{\mathrm{i}(t_n+s)}(A_n))\\
        &=\mathcal{S}_{A_1, \dots, A_n}(t_1, \dots, t_n).
    \end{align*}
    For the last equation, we used $\alpha$-invariance of $\varphi$ (see \cite[Proposition 5.3.3]{BR2}). Next, we have 
    \begin{align*}
        \mathcal{S}_{A_n^*, \dots, A_1^*}(-t_n, \dots, -t_1)
        &= \varphi(\alpha_{-\mathrm{i}t_n}(A_n^*)\cdots \alpha_{-\mathrm{i}t_1}(A_1^*))\\
        &= \varphi((\alpha_{\mathrm{i}t_1}(A_1)\cdots \alpha_{\mathrm{i}t_n}(A_n))^*)\\
        &=\overline{\mathcal{S}_{A_1, \dots, A_n}(t_1, \dots, t_n)}.
    \end{align*}
    Finally, let $A:=\sum_{k=1}^mc_k \alpha_{\mathrm{i}t_1}(A_{k, 1})\cdots \alpha_{\mathrm{i}t_n}(A_{k, n})$, and we have 
    \[\sum_{k, l=1}^m \overline{c_k}c_l\mathcal{S}_{A_{k, n}^*, \dots, A_{k, 1}^*, A_{l, 1}, \dots, A_{l, n}}(-t_n, \dots, -t_1, t_1, \dots, t_n) = \varphi(A^*A)\geq 0.\]
\end{proof}

\subsection{Stochastic processes from stochastically positive KMS systems}

We keep the same notations in the previous section. Let $E$ denote the Gelfand spectrum of $\mathfrak{D}$, i.e., $E$ is a unique compact space, up to homeomorphism, such that $C(E)\cong \mathfrak{D}$. In what follows, we freely identify $\mathfrak{D}$ with $C(E)$. Moreover, we assume that $E$ is second countable.

We say that two stochastic processes are \emph{equivalent} if these finite dimensional distributions are the same. We obtain a $E$-valued stochastic process if $\varphi$ is stochastically positive.

\begin{theorem}[{see also \cite[Theorem 6.1]{KL81}}]\label{thm:stochastic_process}
    Assume that $(\mathfrak{A}, \mathfrak{D}, \alpha, \varphi)$ is a stochastically positive KMS-system, where $\varphi$ is a $\alpha$-KMS state at $\beta\in \mathbb{R}_{>0}\cup\{\infty\}$. There exists a stochastic process $(X_t)_{t\in [-\beta/2, \beta/2]}$ on the spectrum $E$ of $\mathfrak{D}$ such that
    \begin{equation}\label{eq:time_dependent_correlation}
        \mathbb{E}[f_1(X_{t_1})\cdots f_n(X_{t_n})]=\mathcal{S}_{f_{t_1}, \dots, f_{t_n}}(t_1, \dots, t_n)
    \end{equation}
    for every $(t_1, \dots, t_n)\in [-\beta/2, \beta/2]^n_\leq$ and $f_1, \dots, f_n\in C(E)$. Moreover, the following holds:
    \begin{enumerate}
        \item (periodicity) $X_{-\beta/2}=X_{\beta/2}$ holds almost surely if $\beta<\infty$.
        \item (stationarity) $(\widetilde{X}_t)_{t\in \mathbb{R}}$ is a stationary process, where $\widetilde{X}_t:=X_{\tilde t}$ and $t\mapsto \tilde t$ denotes the canonical map from $\mathbb{R}$ to $\mathbb{R}/\beta\mathbb{Z}\cong [-\beta/2, \beta/2]$. Namely, $(\widetilde X_{t+s})_{t\in \mathbb{R}}$ is equivalent to $(\widetilde X_t)_{t\in \mathbb{R}}$ as stochastic processes for all $s\in \mathbb{R}$.
        \item (symmetry) $(X_{-t})_{t\in [-\beta/2, \beta/2]}$ is equivalent to $(X_t)_{t\in [-\beta/2, \beta/2]}$ as stochastic processes.
        \item (OS positivity) $\mathbb{E}[\overline{F(X_{-t_1}, \dots, X_{-t_n})} F(X_{t_1}, \dots, X_{t_n})]\geq 0$ holds for any $(t_1, \dots, t_n)\in [0, \beta/2]^n_\leq $, bounded measurable function $F$ on $E^n$, and $n\geq 1$.
    \end{enumerate}
\end{theorem}
\begin{proof}
    The existence of a stochastic process $(X_t)_{t\in [-\beta/2, \beta/2]}$ satisfying Equation \eqref{eq:time_dependent_correlation} follows from Theorem \ref{thm:construction_process}. We show the periodicity. For any $f\in C(E)$, we have
    \begin{align*}
    &\mathbb{E}[|f(X_{-\beta/2})-f(X_{\beta/2})|^2]\\
    &=\mathcal{S}_{f, f^*}(-\beta/2, -\beta/2)-\mathcal{S}_{f, f^*}(-\beta/2, \beta/2)-\mathcal{S}_{f^*, f}(-\beta/2, \beta/2)+\mathcal{S}_{f^*, f}(\beta/2, \beta/2)
    \end{align*}
    Since $\varphi$ is $\alpha$-invariant, $\mathcal{S}_{f, f^*}(-\beta/2, -\beta/2)=\varphi(ff^*)$ and $\mathcal{S}_{f^*, f}(\beta/2, \beta/2)=\varphi(f^*f)$ hold. By Equation \eqref{eq:beta_half}, we have $\mathcal{S}_{f. f^*}(-\beta/2, \beta/2)=\varphi(f^*f)$ and $\mathcal{S}_{f^*, f}(-\beta/2, \beta/2)=\varphi(ff^*)$. Thus, in conclusion, $\mathbb{E}[|f(X_{-\beta/2})-f(X_{\beta/2})|^2]=0$ holds for any $f\in C(E)$, and hence, $X_{-\beta/2}=X_{\beta/2}$ almost surely. 
    
    Other properties are immediately follows from Lemma \ref{lem:properties_S}.
\end{proof}

\medskip
Let $(\mathfrak{A}, \mathfrak{D}, \alpha, \varphi)$ be a stochastically positive KMS system and $X=(X_t)_{t\in [-\beta/2, \beta/2]}$ the associated $E$-valued stochastic process. By the quotient map $\mathbb{R}\to \mathbb{R}/\beta\mathbb{Z}\cong [-\beta/2, \beta/2]$, we also obtain the periodic stochastic process, denoted by the same symbol $(X_t)_{t\in \mathbb{R}}$. Here, we discuss the known result about its Markov property. See also \cite{KL81}.

Let $(Q, \mathcal{F}, \mathbb{P})$ be an underlying probability space of $X$ and $\mathcal{F}_I$ the $\sigma$-algebra on $Q$ generated by $X_t$ ($t\in I$) for any $I\subset \mathbb{R}$. Throughout this section, we always suppose that $\mathcal{F}=\mathcal{F}_{[-\beta/2, \beta/2]}$. By the monotone class theorem, the linear span of 
\[\{F(X_{t_1}, \dots, X_{t_n})\mid F \text{ is bounded measurable on } E^n, (t_1, \dots, t_n)\in I^n, n\geq 1\}\]
is dense $L^2(Q, \mathcal{F}_I, \mathbb{P})$ for any $I\subset \mathbb{R}$. Here, the function $F(X_{t_1}, \dots, X_{t_n})$ on $Q$ is defined by 
\[F(X_{t_1}, \dots, X_{t_n})(q):=F(X_{t_1}(q), \dots, X_{t_n}(q)) \quad (q\in Q).\]
The stationarity and symmetry induce two unitary operators $U(t)$ ($t\in \mathbb{R}$) and $R$ on $L^2(Q, \mathcal{F}, \mathbb{P})$:
\[U(t)F(X_{t_1}, \dots, X_{t_n}):=F(X_{t_1+t}, \dots, X_{t_n+t}), \quad RF(X_{t_1}, \dots, X_{t_n}):=F(X_{-t_1}, \dots, X_{-t_n})\]
for any bounded measurable function $F$ on $E^n$ and $(t_1, \dots, t_n)\in [-\beta/2, \beta/2]^n$. We call $U(t)$ ($t\in \mathbb{R}$) and $R$ the \emph{time-shift} and \emph{time-reversal} operators, respectively. By definition, we have $U(t)U(s)=U(t+s)$ and $R^2=I$. If $\beta<\infty$, then $U(t+\beta)=U(t)$ holds by the periodicity of $X$.

For any $I\subset \mathbb{R}$, we denote by $E_I$ the orthogonal projection onto $L^2(Q, \mathcal{F}_I, \mathbb{P})\subset L^2(Q, \mathcal{F}, \mathbb{P})$. We remark that $E_I$ is nothing but the conditional expectation with respect to $\mathcal{F}_I$, that is,
\[E_If = \mathbb{E}[f| \mathcal{F}_I]\]
holds for any $f\in L^2(Q, \mathcal{F}, \mathbb{P})$. Moreover, $U(t)E_I = E_{I+t} U(t)$ holds for all $t\in \mathbb{R}$.

First, we examine the periodic case ($\beta<\infty$). Here, the usual Markov property dose not make sense. Thus, we should consider the \emph{two-sided Markov property}. The stochastic process $X$ is said to be \emph{two-sided Markov} if 
\[E_{[t-\beta/2, t]}E_{[t, t+\beta/2]}=E_{\{t, t+\beta/2\}}\]
holds for any $t\in[-\beta/2, \beta/2]$. By the stationarity, $X$ is two-sided Markov if and only if the above equality holds for $t=0$, i.e., $E_{[-\beta/2, 0]}E_{[0, \beta/2]}=E_{\{0, \beta/2\}}$. By the periodicity, we have $RE_{\{0, \beta/2\}}=E_{\{0, \beta/2\}}=E_{\{0, \beta/2\}} R$. Thus, $X$ is two-sided Markov if and only if 
\[E_{[0, \beta/2]}RE_{[0, \beta/2]}=E_{\{0, \beta/2\}}\]
holds true (see \cite[Proposition 11.1]{KL81}).

Let $(\pi_\varphi, \mathcal{H}_\varphi, \Omega_\varphi)$ denote the GNS-triple associated with $\varphi$, and $H$ denotes the self-adjoint operator on $\mathcal{H}_\varphi$ so that $e^{\mathrm{i}t H}\pi_\varphi(A)\Omega_\varphi=\pi_\varphi(\alpha_t(A))\Omega_\varphi$ for all $t\in \mathbb{R}$. For any $f_1, \dots, f_n\in C(E)$, $(t_1, \dots, t_n)\in [0, \beta/2]^n_\leq$, and $n\geq 1$, we define 
\[V(f_1\otimes \cdots f_n)(X_{t_1}, \dots, X_{t_n}):=e^{-t_1 H}\pi_\varphi(f_1) e^{-(t_2-t_1) H}\pi_\varphi(f_2)\cdots e^{-(t_n-t_{n-1})H} \pi_\varphi(f_n)\Omega_\varphi.\]
By Equation \eqref{eq:time_dependent_correlation}, $V$ is a well-defined linear map $L^2(Q, \mathcal{F}_{[0, \beta/2]}, \mathbb{P})$ satisfying that 
\[\langle VF, VG\rangle=(RF, G) \quad (F, G\in L^2(Q, \mathcal{F}_{[0, \beta/2]}, \mathbb{P})).\] 
Here, $\langle\,,\,\rangle$ and $(\,,\,)$ denote the inner products of $\mathcal{H}_\varphi$ and $L^2(Q, \mathcal{F}_{[0, \beta/2]}, \mathbb{P})$, respectively. 

Let $\mathcal{H}_{\varphi, I}$ denote the closure of $VL^2(Q, \mathcal{F}_I, \mathbb{P})$ in $\mathcal{H}_\varphi$ for any $I\subset [0, \beta/2]$. The following characterization is known (see \cite[Theorem 11.2]{KL81}):
\begin{proposition}
    $X$ is two-sided Markov if and only if $\mathcal{H}_{\varphi, [0, \beta/2]}=\mathcal{H}_{\varphi, {\{0, \beta/2\}}}$.
\end{proposition}
\begin{proof}
    Let $e_{\{0, \beta\}}$ be the orthogonal projection onto $\mathcal{H}_{\varphi, \{0, \beta/2\}}\subset \mathcal{H}_{\varphi, [0, \beta/2]}$. We have 
    \[\langle VF, VG\rangle = (RF, G) =(E_{\{0, \beta/2\}}R F, G) =(RE_{\{0, \beta\}}F, G)=\langle VE_{\{0, \beta/2\}}F, VG\rangle\]
    for any $F\in L^2(Q, \mathcal{F}_{[0, \beta/2]}, \mathbb{P})$ and $G\in L^2(Q, \mathcal{F}_{\{0, \beta/2\}}, \mathbb{P})$. Hence, $e_{\{0, \beta\}}V=VE_{\{0, \beta\}}$ holds on $L^2(Q, \mathcal{F}_{[0, \beta/2]}, \mathbb{P})$. Moreover, for any $F, G\in L^2(Q, \mathcal{F}_{[0, \beta/2]}, \mathbb{P})$, we have 
    \[\langle e_{\{0, \beta/2\}} VF, VG\rangle = (RE_{\{0, \beta/2\}}F, G)= (E_{\{0, \beta/2\}}F, G).\]
    In particular, if $X$ is two-sided Markov, i.e., $E_{[0, \beta/2]}RE_{[0, \beta/2]}=E_{\{0, \beta/2\}}$ holds, then we have 
    \[\langle e_{\{0, \beta/2\}} VF, VG\rangle = (RF, G)=\langle VF, VG\rangle\]
    for any $F, G\in L^2(Q, \mathcal{F}_{[0, \beta/2]}, \mathbb{P})$. Namely, $\mathcal{H}_{\varphi, [0, \beta/2]}=\mathcal{H}_{\varphi, {\{0, \beta/2\}}}$ holds. 
    
    Conversely, if $\mathcal{H}_{\varphi, [0, \beta/2]}=\mathcal{H}_{\varphi, {\{0, \beta/2\}}}$, then we have 
    \[(RF, G)=\langle VF, VF\rangle =\langle e_{\{0, \beta/2\}}VF, VG\rangle =(E_{\{0, \beta/2\}}F, G)\]
    for any $F, G\in L^2(Q, \mathcal{F}_{[0, \beta/2]}, \mathbb{P})$, and it implies $E_{[0, \beta/2]}RE_{[0, \beta/2]}=E_{\{0, \beta/2\}}$.
\end{proof}

Next, we discuss the non-periodic ($\beta=\infty$) case. Let us recall that $X$ is said to be \emph{Markov} if 
\[E_{(-\infty, t]}E_{[t, \infty)}=E_t\]
holds for all $t\in \mathbb{R}$. Similarly to the above case, $X$ is Markov if and only if $E_{(-\infty, 0]}E_{[0, \infty)}=E_0$ since $X$ is stationary. Moreover, it holds if and only if $E_{[0, \infty)}RE_{[0, \infty)}=E_0$.

By the same argument, we obtain the following characterization of Markov property.
\begin{proposition}
    $X$ is Markov if and only if $\mathcal{H}_{\varphi, [0, \infty)}=\mathcal{H}_{\varphi, \{0\}}$.
\end{proposition}

Thus, we obtain a sufficient condition for (two-sided) Markov property:
\begin{corollary}\label{cor:suf_cond_Markov}
    Let us consider the following conditions:
    \begin{description}
        \item [periodic ($\beta<\infty$) case] $\{\pi_\varphi(f_1)e^{-\beta/2H}\pi_\varphi(f_2)\Omega_\varphi \mid f_1, f_2\in C(E) \}$ is dense in $\mathcal{H}_\varphi$,
        \item [non-periodic ($\beta=\infty$) case] $\{\pi_\varphi(f)\Omega_\varphi\mid f\in C(E)\}$ is dense in $\mathcal{H}_\varphi$.
    \end{description}
    These conditions implies the (two-sided) Markov property of $X$.
\end{corollary}

\begin{remark}
    The above sufficient condition for (two-sided) Markov property suggests that the \emph{perfectness} of DPPs in \cite{Olshanski20} play a crucial role in the dynamical relationship between DPPs and quasi-free states on (gauge-invariant) CAR algebras. We will return to this point of view later. See Propositions \ref{prop:Markov_property}, \ref{prop:Markov_2}.
\end{remark}

\medskip
In the rest of this section, we discuss limit transitions of stochastically positive KMS states. Let $\varphi_N$ be an $\alpha$-KMS state on $\mathfrak{A}$ at $\beta_N>0$ for every $N\geq 1$. We define $T_N:=[-\beta_N/2, \beta_N/2]$ and $\mathcal{S}^{(N)}_{A_1, \dots, A_n}\in C({T_N}^n_\leq)$ for any $A_1, \dots, A_n\in \mathfrak{A}$ and $n\geq1$ by Equation \eqref{eq:S_func_from_KMS} with $\varphi_N$. Let $T:=[-\beta/2, \beta/2]$ for $\beta\in \mathbb{R}_{>0}\cup \{\infty\}$.

\begin{lemma}\label{lem:convergence}
    Let $\varphi$ be a state on $\mathfrak{A}$ so that $\varphi=\lim_{N\to\infty}\varphi_N$ in the weak${}^*$ topology and assume that $\beta_N\nearrow \beta\in \mathbb{R}\cup \{\infty\}$. Then, $\varphi$ is an $\alpha$-KMS state at $\beta$, and 
    \[\lim_{N\to\infty}\mathcal{S}^{(N)}_{A_1, \dots, A_n}(t_1, \dots, t_n)= \mathcal{S}_{A_1, \dots, A_n}(t_1, \dots, t_n)\]
    holds for any $A_1, \dots, A_n \in \mathfrak{A}$, $(t_1, \dots, t_n)\in T^n_\leq$, and $n\geq1$, where $\mathcal{S}_{A_1, \dots, A_n}$ is defined by Equation \eqref{eq:S_func_from_KMS} with $\varphi$, and $N$ is large enough so that $t_1,\dots, t_n\in T$. In particular, if $\varphi_N$ are stochastically positive with respect to $\mathfrak{D}$ for all $N\geq1$, then so is $\varphi$.
\end{lemma}
\begin{proof}
    By \cite[Proposition 5.3.23]{BR2}, $\varphi$ is an $\alpha$-KMS state at $\beta$. By Equation \eqref{eq:analytic}, we have
    \[\lim_{N\to\infty}\mathcal{S}^{(N)}_{A_1, \dots, A_n}(t_1, \dots, t_n)= \mathcal{S}_{A_1, \dots, A_n}(t_1, \dots, t_n)\]
    for any $(t_1, \dots, t_n)\in T^n_\leq$ if $A_1, \dots, A_n\in \mathfrak{A}$ are entire analytic for $\alpha$. By the argument in Remark \ref{rem:approximation}, the statement holds true for all $A_1, \dots, A_n\in \mathfrak{A}$.
\end{proof}

We assume that $\varphi_N$ is stochastically positive with respect to $\mathfrak{D}$ for all $N\geq 1$, and $(X^{(N)}_t)_{t\in T_N}$ denotes the associated $E$-valued stochastic process by Theorem \ref{thm:stochastic_process}. In addition, by Lemma \ref{lem:convergence}, we obtain a $E$-valued stochastic process $(X_t)_{t\in T}$ associated with $\varphi$.

\begin{proposition}\label{prop:convergence}
    $(X^{(N)}_t)_{t\in T_N}$ converges to $(X_t)_{t\in T}$ as $N\to \infty$ in the following sense: for any $n\geq 1$ and $t_1\leq \cdots \leq t_n$, the distribution of $(X^{(N)}_{t_1}, \dots, X^{(N)}_{t_n})$ weakly converges to the distribution of $(X_{t_1}, \dots, X_{t_n})$, where $N$ is large enough that $t_1, \dots, t_n\in T_N$.
\end{proposition}
\begin{proof}
    By Lemma \ref{lem:convergence}, we have $\lim_{N\to\infty}\mathbb{E}[f_1(X^{(N)}_{t_1})\cdots f_n(X^{(N)}_{t_n})]=\mathbb{E}[f_1(X_{t_1})\cdots f_n(X_{t_n})]$ for any $f_1, \dots, f_n\in C(E)$, $t_1\leq \cdots \leq t_n$, and $n\geq1$.
\end{proof}

\section{Stochastically positive quasi-free states on CAR algebras}\label{sec:density_op_fermion_Fock_sp}

We already discussed our perspective to study DPPs on a discrete countable space $\mathfrak{X}$ in Section \ref{sec:DPP_CAR}. Namely, the CAR algebra $\mathfrak{A}(\mathfrak{X})$ determined by $\ell^2(\mathfrak{X})$ contains the commutative $C^*$-subalgebra $\mathfrak{D}(\mathfrak{X})$ isomorphic to $C(\mathcal{C}(\mathfrak{X}))$. Moreover, a quasi-free state on $\mathfrak{A}(\mathfrak{X})$ gives rise to a DPP on $\mathfrak{X}$ with the correlation kernel given by $(x, y)\in \mathfrak{X}\times \mathfrak{X}\mapsto \varphi(a^*_xa_y)$. In this section, we discuss quasi-free states on $\mathfrak{A}(\mathfrak{X})$ that are stochastically positive for $\mathfrak{D}(\mathfrak{X})$. By the previous results, they produce stationary processes with respect to the associated DPPs (see Section \ref{sec:6}).

\subsection{Quasi-free states and associated stochastic processes}

Let $H$ be a self-adjoint (possibly, unbounded) operator on $\ell^2(\mathfrak{X})$. We define 
\begin{equation}\label{eq:H_K}
    K_{H, \beta}:=e^{-\beta H}(1+e^{-\beta H})^{-1}
\end{equation} 
Since $K_{H, \beta}$ is positive contractive for all $\beta>0$, we obtain a quasi-free state $\varphi_{H, \beta}$ on $\mathfrak{A}(\mathfrak{X})$ such that $\varphi_{H, \beta}(a^*(h)a(k))=\langle K_{H, \beta} h, k\rangle$ for any $h, k\in \ell^2(\mathfrak{X})$. In particular, the associated determinantal point process, denoted by $\mathbb{P}_{H, \beta}$, has the correlation kernel $K_{H, \beta}(x, y):=\langle K_{H, \beta}\delta_x, \delta_y\rangle$. It is nothing but the so-called $L$-ensemble with $L=e^{-\beta H}$.

It is known that $\varphi_{H, \beta}$ is a $\alpha^H$-KMS state at $\beta$ (see \cite[Exapmle 5.3.2]{BR2}). Here, the $\mathbb{R}$-flow $\alpha^H\colon \mathbb{R}\curvearrowright \mathfrak{A}(\mathfrak{X})$ is defined by $\alpha^H_t(a(h)):=a(e^{\mathrm{i}t H} h)$ for any $h\in \ell^2(\mathfrak{X})$.

\begin{example}
    For all $t\in \mathbb{R}$, a unitary operator $\Gamma(e^{\mathrm{i}t H})$ on the anti-symmetric Fock space $\mathcal{F}_a(\ell^2(\mathfrak{X}))$ is defined by
    \[\Gamma(e^{\mathrm{i}t H})\Omega:=\Omega, \quad \Gamma(e^{\mathrm{i}t H})h_1\wedge \cdots \wedge h_n:=(e^{\mathrm{i}t H}h_1)\wedge \cdots \wedge (e^{\mathrm{i}t H}h_n)\]
    for any $h_1, \dots, h_n\in \ell^2(\mathfrak{X})$ and $n\geq 1$. By Stone's theorem, there exists a unique self-adjoint operator $d\Gamma(H)$ on $\mathcal{F}_a(\mathcal{H})$, called the \emph{second quantization} of $H$, such that $e^{\mathrm{i}t d\Gamma(H)}=\Gamma(e^{\mathrm{i}t H})$ for all $t\in \mathbb{R}$. If $e^{-\beta H}$ is trace class, then $e^{-\beta d\Gamma(H)}$ is also trace class on $\mathcal{F}_a(\ell^2(\mathfrak{X}))$ (see \cite[Proposition 5.2.22]{BR2}). Then, by \cite[Proposition 5.2.23]{BR2}, we obtain the quasi-free state $\varphi_{H, \beta}$ as follows
    \begin{equation}\label{eq:qf_state_density_op}
        \varphi_{H, \beta}(A):=\frac{\mathrm{Tr}(e^{-\beta d\Gamma(H)} A)}{\mathrm{Tr}(e^{-\beta d\Gamma(H)})} \quad (A\in \mathfrak{A}(\mathfrak{X})).
    \end{equation}
    Here, $\mathfrak{A}(\mathfrak{X})$ is identified with its Fock representation on $\mathcal{F}_a(\ell^2(\mathfrak{X}))$.
\end{example}

\begin{remark}\label{rem:limit_trans_projection}
    By \cite[Theorem VIII.5]{RS1}, we have $\lim_{\beta\to\infty}K_{H, \beta}=K_H^{(-)}$ strongly, where $K_H^{(-)}$ denotes the spectral projection onto $\{H < 0\}$. In this sense, the determinantal point process with correlation kernel $K_{H, \beta}$ can be regarded as a \emph{finite temperature} version of the determinantal point process with correlation kernel $K_H^{(-)}$.
\end{remark}

In what follows, $\mathcal{W}$ and $\mathcal{S}$ denote functions defined by \eqref{eq:W}, \eqref{eq:S_func_from_KMS} for $\varphi_{H, \beta}$.
\begin{lemma}\label{lem:qf_state_at_analytic}
    For any $h_1, \dots, h_n, k_1, \dots, k_n\in \ell^2(\mathfrak{X})$,
    \[\mathcal{S}_{a^*(h_1)a(k_1), \dots, a^*(h_n)a(k_n)}(t_1, \dots, t_n)=\det[B_{i, j}]_{i, j=1}^n\]
    holds, where $(t_1, \dots, t_n)\in [-\beta/2, \beta/2]^n_\leq$ and
    \[B_{i, j}:=\begin{dcases} \mathcal{S}_{a^*(h_i), a(k_j)}(t_i, t_j) & i\leq j, \\ -\mathcal{S}_{a(k_j), a^*(h_i)}(t_j, t_i) & i>j.
    \end{dcases}\]
\end{lemma}
\begin{proof}
    By the same argument in Remark \ref{rem:approximation}, we may assume that $h_1, \dots, h_n, k_1, \dots, k_n$ are entire analytic for $(e^{\mathrm{i}t H})_{t\in \mathbb{R}}$. Thus, $a^*(h_1), \dots a^*(h_n), a(k_1), \dots, a(k_n)$ are also entire analytic for $\alpha^{H}$, and we have 
    \[\alpha^{H}_z(a^*(h_i))=a^*(e^{\mathrm{i}z H}h_i), \quad \alpha^{H}_z(a(k_j))=\alpha^{H}_{\bar z}(a^*(k_j))^*=a(e^{\mathrm{i}\bar z H}k_j)\]
    for any $z\in \mathbb{C}$. Hence, by Equation \eqref{eq:analytic} and Lemma \ref{lem:qf_st_calc}, we have 
    \begin{align*}
        &\mathcal{S}_{a^*(h_1)a(k_1), \dots, a^*(h_n)a(k_n)}(t_1, \dots, t_n) \\
        &= \varphi_{H, \beta}(\alpha_{\mathrm{i}t_1}(a^*(h_1)a(k_1)) \cdots \alpha^{H}_{\mathrm{i}t_n}(a^*(h_n)a(k_n)))\\
        &= \varphi_{H, \beta}(a^*(e^{-t_1 H}h_1)a(e^{t_1 H}k_1)\cdots a^*(e^{-t_n H}h_n)a(e^{t_nH}k_n)) \\
        &= \det[B_{i, j}]_{i, j=1}^n.
    \end{align*}
\end{proof}

We give a sufficient condition for the stochastic positivity of $\varphi_{H, \beta}$ with respect to $\mathfrak{D}(\mathfrak{X})$. Let us recall that $\rho_x:=a_x^*a_x$ for every $x\in \mathfrak{X}$.

\begin{lemma}\label{lem:cond_stoc_positivity_abst}
    $\varphi_{H, \beta}$ is stochastically positive with respect to $\mathfrak{D}(\mathfrak{X})$ if $\mathcal{S}_{\rho_{x_1}, \dots, \rho_{x_n}}(t_1, \dots, t_n)\geq 0$ holds for any $x_1, \dots, x_n\in \mathfrak{X}$, $(t_1, \dots, t_n)\in [-\beta/2, \beta/2]^n_\leq$, and $n\geq 1$.
\end{lemma}
\begin{proof}
    We remark that for any $A_1, \dots, A_n, B\in \mathfrak{A}(\mathfrak{X})$ and $j=1, \dots, n$
    \[\mathcal{S}_{A_1, \dots, A_jB, A_{j+1}, \dots, A_n}(t_1, \dots, t_n)= \mathcal{S}_{A_1, \dots, A_j, B, A_{j+1}, \dots, A_n}(t_1, \dots, t_j, t_j, \dots, t_n).\]
    Thus, by assumption, $\mathcal{S}_{A_1, \dots, A_n}(t_1, \dots, t_n)\geq 0$ holds when $A_1, \dots, A_n$ are products of $\{\rho_x\}_{x\in \mathfrak{X}}$. Since any nonnegative continuous functions on $\mathcal{C}(\mathfrak{X})$ can be approximated by linear combinations of products of $\{\rho_x\}_{x\in \mathfrak{X}}$ with positive coefficients, $\varphi_{H, \beta}$ is stochastically positive.
\end{proof}

We suppose that $e^{-\beta H}$ is trace class on $\ell^2(\mathfrak{X})$. Hence, $e^{-\beta L}$ is also trace class on $\mathcal{F}_a(\ell^2(\mathfrak{X}))$, where $L:=d\Gamma(H)$. By \cite[Theorem 17.5]{KL81}, we have
\begin{align}\label{eq:schwinger_func_trace}
     & \mathcal{S}_{a_1, \dots, a_n}(t_1, \dots, t_n) \nonumber                                                                                                            \\
     & =\frac{1}{\mathrm{Tr}(e^{-\beta L})}\mathrm{Tr}(e^{-(t_1+\beta/2)L} a_1 e^{-(t_2-t_1)L} a_2\cdots e^{-(t_n-t_{n-1})L} a_n e^{-(\beta/2 -t_n)L})
\end{align}
for any $(t_1, \dots, t_n)\in [-\beta/2, \beta/2]^n_\leq$ and $a_1, \dots, a_n\in \mathfrak{A}(\mathfrak{X})$.

In addition, we assume that $\mathfrak{X}$ is equipped with a linear order $\leq$ and fix the orthonormal basis $\{\Omega\}\cup \{\delta_{x_1}\wedge \cdots \wedge \delta_{x_n}\}_{x_1< \cdots < x_n; n\geq 1}$ for $\mathcal{F}_a(\ell^2(\mathfrak{X}))$. 
\begin{lemma}\label{lem:cond_stoc_positivity}
    $\varphi_{H, \beta}$ is stochastically positive with respect to $\mathfrak{D}(\mathfrak{X})$ if 
    \[\langle e^{-tL}\delta_{x_1}\wedge \cdots \wedge \delta_{x_n}, \delta_{y_1}\wedge \cdots \wedge \delta_{y_n}\rangle=\det[\langle e^{-t H}\delta_{x_i}, \delta_{y_j}\rangle]_{i, j=1}^n\geq 0\] 
    for any $x_1<\cdots<x_n$, $y_1<\cdots < y_n$ in $\mathfrak{X}$ and $0\leq t\leq \beta$.
\end{lemma}
\begin{proof}
    We remark that $\langle \rho_x \delta_{x_1}\wedge \cdots \wedge \delta_{x_n}, \delta_{y_1}\wedge \cdots \wedge \delta_{y_n}\rangle$ is equal to 1 if $x\in \{x_1, \dots, x_n\}$ and $\{x_1, \dots, x_n\}=\{y_1, \dots, y_n\}$ for any $x\in \mathfrak{X}$; otherwise, it is equal to 0. Thus, by assumption,
    \[\mathrm{Tr}(e^{-(t_1+\beta/2)L}\rho_{x_1}e^{-(t_2-t_1)L} \rho_{x_2} \cdots e^{-(t_n-t_{n-1})L} \rho_{x_n}e^{-(\beta/2-t_n)L})\geq 0\]
    holds true for any $x_1, \dots, x_n\in \mathfrak{X}$ and $(t_1, \dots, t_n)\in [-\beta/2, \beta/2]^n_\leq$. Moreover, by Equation \eqref{eq:schwinger_func_trace} and Lemma \ref{lem:cond_stoc_positivity_abst}, $\varphi_{H, \beta}$ is stochastically positive.
\end{proof}

Let us return to the general setting ($e^{-\beta H}$ is not necessarily trace class). The following is our first main result:
\begin{theorem}\label{thm:dynamical_correlation}
    Assume that $\varphi_{H, \beta}$ is stochastically positive with respect to $\mathfrak{D}(\mathfrak{X})$ and let $(X^{(\beta)}_t)_{t\in [-\beta/2, \beta/2]}$ denote the associated stochastic process on $\mathcal{C}(\mathfrak{X})$ by Theorem \ref{thm:stochastic_process}. For any $n\geq 1$, $x_1, \dots, x_n\in \mathfrak{X}$, and $-\beta/2\leq t_1\leq \cdots \leq t_n\leq \beta/2$, we have
    \[\mathbb{P}[(X^{(\beta)}_{t_1})_{x_1}=1, \dots, (X^{(\beta)}_{t_n})_{x_n}=1]=\det[R_{H, \beta}(x_i, t_i; x_j, t_j)]_{i, j=1}^n,\]
    where
    \begin{equation}\label{eq:zero_temp_limit}
        R_{H, \beta}(x, t; y, s)= \begin{dcases} \langle e^{-(\beta-s+t)H}(1+e^{-\beta H})^{-1}\delta_x, \delta_y\rangle & t\leq s, \\ -\langle e^{-(t-s)H}(1+e^{-\beta H})^{-1}\delta_x, \delta_y\rangle & t>s.\end{dcases}
    \end{equation}
\end{theorem}
\begin{proof}
    The statement immediately follows Equation \eqref{eq:time_dependent_correlation} and Lemma \ref{lem:qf_state_at_analytic}. In fact, if $h, k\in \ell^2(\mathfrak{X})$ are entire analytic for $(e^{\mathrm{i}t H})_{t\in \mathbb{R}}$, then $a^*(h)$ and $a(k)$ are also entire analytic for $\alpha^{H}$, and for any $-\beta/2 \leq t\leq s\leq \beta/2$, we have 
    \[\mathcal{S}_{a^*(h), a(k)}(t, s)=\varphi_{H, \beta}( a^*(e^{-t H}h)a(e^{s H}k))=\langle e^{-(\beta-s+t)H}(1+e^{-\beta H})^{-1} h, k\rangle.\]
    Here, we remark that $e^{-(\beta-s+t)H}(1+e^{-\beta H})^{-1}$ is bounded since $0\leq s-t\leq \beta$. Similarly, for any $-\beta/2\leq s<t\leq \beta/2$, we have 
    \[\mathcal{S}_{a(k), a^*(h)}(s, t)=\varphi_{H, \beta}(a(e^{s H}k)a^*(e^{-t H}h))=\langle e^{-(t-s)H}(1+e^{-\beta H})^{-1} h, k\rangle.\]
    Since the space of entire analytic vectors for $(e^{\mathrm{i} t H})_{t\in \mathbb{R}}$ is dense, we obtain Equation \eqref{eq:zero_temp_limit}.
\end{proof}

For fixed $-\beta/2\leq t_1\leq \cdots t_n\leq \beta/2$, the theorem states that the point process $X^{(\beta)}_{t_1}\cup \cdots \cup X^{(\beta)}_{t_n}$ on $\mathfrak{X}\cup\cdots \cup \mathfrak{X}$ is determinantal with correlation kernel $R_{H; \beta}$. Such a statement is commonly known as the Eynard--Metha type theorem (see \cite{BR05}).

\medskip
An explicit expression of space-time correlation $R_{H,\beta}$ will be computed in Section \ref{sec:stoc_from_op}.

\subsection{Limiting behavior as temperature tends to zero}\label{sec:zero_temp_limit}
Let us keep the same notations as in the previous section. As mentioned in Remark \ref{rem:limit_trans_projection}, we have 
\[\lim_{\beta \to \infty} K_{H, \beta}=K_{H}^{(-)}\]
strongly, where $K_{H}^{(-)}$ denotes the spectral projection onto $\{H<0\}$. Thus, as $\beta\to\infty$, the state $\varphi_{H, \beta}$ weakly converges to $\varphi_{K_{H}^{(-)}}$. 

\begin{proposition}\label{prop:dynam_corr_zero_temp}
    Assume that $\varphi_{H, \beta}$ is stochastically positive with respect to $\mathfrak{D}(\mathfrak{X})$ for all $\beta>0$. Then, so is $\varphi_{K_{H}^{(-)}}$. Moreover, the associated stochastic process $(X_t)_{-\infty < t< \infty}$ on $\mathcal{C}(\mathfrak{X})$ satisfies 
    \[\mathbb{P}[(X_{t_1})_{x_1}=1, \dots, (X_{t_n})_{x_n}=1]=\det[R_{H}(x_i, t_i; x_j, t_j)]_{i, j=1}^n\]
    for any $x_1, \dots, x_n\in \mathfrak{X}$ and $t_1\leq \cdots \leq t_n$, where 
    \[R_{H}(x, t; y, s):=\begin{dcases}\langle e^{(s-t)H}K_{H}^{(-)}\delta_x, \delta_y\rangle & t\leq s, \\ -\langle e^{-(t-s)H}(1-K_{H}^{(-)})\delta_x, \delta_y\rangle & t>s.\end{dcases}\]
\end{proposition}
\begin{proof}
    By Lemma \ref{lem:convergence}, $\varphi_{K_{H}^{(-)}}$ is also stochastically positive. Moreover, by \cite[Theorem VIII.5]{RS1},
    \[\lim_{\beta\to\infty; t<\beta} e^{-(\beta-t)H}(1+e^{-\beta H})^{-1}=e^{t H}K_{H}^{(-)}\]
    in the strong operator topology for any $t\in \mathbb{R}$. Here, $e^{-(\beta-t)H}(1+e^{-\beta H})^{-1}$ is bounded when $t<\beta$. Similarly,
    \[\lim_{\beta\to \infty} e^{-t H}(1+e^{-\beta H})^{-1}= e^{-t H}(1-K_{H}^{(-)})\]
    strongly for any $t\geq 0$. Thus, by Theorem \ref{thm:dynamical_correlation} and Equation \eqref{eq:zero_temp_limit}, the assertion holds true.
\end{proof}

\subsection{Limiting behavior by convergences of self-adjoint operators}\label{sec:LT_by_strong_res_conv}
Let $(H_N)_{N=1, 2, \dots}$ be a sequence of self-adjoint operators on $\ell^2(\mathfrak{X})$. We assume the following twofold:
\begin{itemize}
    \item $\varphi_{H_N, \beta}$ is stochastically positive with respect to $\mathfrak{D}(\mathfrak{X})$ for all $N\geq1$.
    \item the sequence $(H_N)_{N=1}^\infty$ converges to the self-adjoint operator $H$ in the strong resolvent sense, i.e., $\lim_{N\to\infty}(\lambda-H_N)^{-1}= (\lambda-H)^{-1}$ strongly for any $\lambda\in \mathbb{C}\backslash\mathbb{R}$.
\end{itemize}

By the first assumption and Theorem \ref{thm:stochastic_process}, we obtain the stochastic process $(X^{(N)}_t)_{t\in [-\beta/2, \beta/2]}$ from $(\mathfrak{A}(\mathfrak{X}), \mathfrak{D}(\mathfrak{X}), \alpha^{H_N}, \varphi_{H_N, \beta})$. Moreover, the second assumption and \cite[Theorem VIII.20]{RS1} imply that $K_{H_N, \beta}$ strongly converges to $K_{H, \beta}$ as $N\to \infty$, where $K_{H_N, \beta}$ and $K_{H, \beta}$ are defined by Equation \eqref{eq:H_K} for $H_N$ and $H$. Thus, $\lim_{N\to \infty}\varphi_{H_N, \beta}=\varphi_{H, \beta}$ in the weak${}^*$ topology. 

We denote by $\mathcal{S}^{(N)}$ and $\mathcal{S}$ the functions defined by Equation \eqref{eq:S_func_from_KMS} for $\varphi_{H_N, \beta}$ and $\varphi_{H, \beta}$, respectively. Here, we have to pay attention that $\alpha^{H_N}$ depends on $N$ as well as $\varphi_{H_N, \beta}$. Nevertheless, we obtain the following statement:

\begin{proposition}\label{prop:limit_by_strong_resolvent_convergence}
    $\varphi_{H, \beta}$ is stochastically positive with respect to $\mathfrak{D}(\mathfrak{X})$. Let $(X_t)_{t\in [-\beta/2, \beta/2]}$ be the associated stochastic process on $\mathcal{C}(\mathfrak{X})$. Then, $(X^{(N)}_t)_{t\in [-\beta/2, \beta/2]}$ converges to $(X_t)_{t\in [-\beta/2, \beta/2]}$ in the sense of Proposition \ref{prop:convergence}. Moreover, for any $x_1, \dots, x_n\in \mathfrak{X}$ and $-\beta/2\leq t_1\leq \cdots \leq t_n\leq \beta/2$
    \[\mathbb{P}[(X_{t_1})_{x_1}=1, \dots, (X_{t_n})_{x_n}=1]=\det[R_{H, \beta}(x_i, t_i; x_j, t_j)]_{i, j=1}^n\]
    holds, where $R_{H, \beta}$ is defined by Equation \eqref{eq:zero_temp_limit}.
\end{proposition}
\begin{proof}
    First, we show $\mathcal{S}_{\rho_{x_1}, \dots, \rho_{x_n}}(t_1, \dots, t_n)\geq 0$ for any $x, \dots, x_n\in \mathfrak{X}$ and $(t_1, \dots, t_n)\in [-\beta/2, \beta/2]^n_\leq$. Since $H_N$ converges to $H$ in the strong resolvent sense, we have 
    \begin{align*}
        \mathcal{S}_{\rho_{x_1}, \dots, \rho_{x_n}}(t_1, \dots, t_n)
        &=\det[R_{H, \beta}(x_i, t_i; x_j, t_j)]_{i, j=1}^n \quad (\text{by Theorem \ref{thm:dynamical_correlation}})\\
        &=\lim_{N\to \infty} \det[R_{H_N, \beta}(x_i, t_i; x_j, t_j)]_{i, j=1}^n\\
        &=\lim_{N\to \infty} \mathcal{S}^{(N)}_{\rho_{x_1}, \dots, \rho_{x_n}}(t_1, \dots, t_n) \quad (\text{by Theorem \ref{thm:dynamical_correlation}}) \\
        &\geq 0.
    \end{align*}
    Thus, by Lemma \ref{lem:cond_stoc_positivity_abst}, $\varphi_{H, \beta}$ is stochastically positive.

    Let $f_1, \dots, f_n$ be in the $\mathbb{C}$-algebra generated by $\{\rho_x\mid x\in \mathfrak{X}\}$. By above convergence (see also the proof of Lemma \ref{lem:cond_stoc_positivity_abst}), we have $\lim_{N\to \infty}\mathcal{S}^{(N)}_{f_1, \dots, f_n}(t_1, \dots, t_n)=\mathcal{S}_{f_1, \dots, f_n}(t_1, \dots, t_n)$ for any $(t_1, \dots, t_n)\in [-\beta/2, \beta/2]^n_{\leq}$. Furthermore, by the Stone--Weierstrass theorem, the $\mathbb{C}$-algebra generated by $\{\rho_x\mid x\in \mathfrak{X}\}$ is norm dense in $\mathfrak{D}(\mathfrak{X})$, and hence, the same convergence holds for all $f_1, \dots, f_n\in \mathfrak{D}(\mathfrak{X})$. Namely, $(X^{(N)}_t)_{t\in [-\beta/2, \beta/2]}$ converges to $(X_t)_{t\in [-\beta/2, \beta/2]}$ in the sense of Proposition \ref{prop:convergence}.
\end{proof}

\section{Determinantal point processes related to orthogonal polynomials}\label{sec:6}
Orthogonal polynomial ensembles are one of the most fundamental classes of determinantal point processes. In this section, we introduce their ``finite temperature'' version. Moreover, we also discuss their limiting behavior. We will study stationary processes with respect to these DPPs in the next section.

\subsection{Orthogonal polynomial ensembles}\label{sec:ope}

Let $\mathfrak{X}\subset \mathbb{R}$ and $w\colon \mathfrak{X}\to \mathbb{R}_{>0}$ be a weight function such that $\sum_{x\in \mathfrak{X}}x^n w(x)<\infty$ for all $n\geq 0$. By assumption, $1, x, x^2, \dots$ belong to $L^2(\mathfrak{X}, w)$, and hence, we obtain orthogonal monic polynomials $(\tilde p_n(x))_{n\geq 0}$ with respect to $w$. Namely, $\tilde p_n(x)$ is a polynomial with degree $n$ in which the leading coefficient of $x^n$ is 1 for each $n\geq 1$, and $\sum_{x\in \mathfrak{X}}\tilde p_n(x)\tilde p_m(x)w(x)=0$ if $n\neq m$. Moreover, we define $p_n(x):=\tilde p_n(x)w(x)^{1/2}/\|\tilde p_n\|_{L^2(\mathfrak{X}, w)}$ for every $n\geq 0$. By this normalization, we have $p_n\in \ell^2(\mathfrak{X})$, and they form an orthonormal basis.

Let us fix $N\geq 1$ and define the \emph{normalized Christoffel--Darboux kernel} $K_{w, N}$ by
\begin{equation}\label{eq:CD_kernel}
    K_{w, N}(x, y):=\sum_{n=0}^{N-1}p_n(x)p_n(y) \quad (x, y\in \mathfrak{X}).
\end{equation}
The integral operator of $K_{w, N}$,  denoted by the same symbol $K_{w, N}$, is the orthogonal projection onto the linear span of $p_0, \dots, p_{N-1}$. Thus, $K_{w, N}$ is a positive contraction operator on $\ell^2(\mathfrak{X})$. Therefore, as mentioned in Section \ref{sec:DPP_CAR}, we obtain the quasi-free state $\varphi_{w, N}$ on the CAR algebra $\mathfrak{A}(\mathfrak{X})$ such that $\varphi_{w, N}(a^*(h)a(k))=\langle K_{w, N}h, k\rangle$ for any $h, k\in \ell^2(\mathfrak{X})$. Moreover, $\varphi_{w, N}$ gives the determinantal point process $\mathbb{P}_{w, N}$ on $\mathfrak{X}$ with correlation kernel $K_{w, N}$.

We call $\mathbb{P}_{w, N}$ the \emph{orthogonal polynomial ensemble} with wight $w$. It can be realized explicitly as follows: Let $\mathcal{C}_N(\mathfrak{X}):=\{\omega\in \mathcal{C}(\mathfrak{X})\mid \sum_{x\in \mathfrak{X}}\omega_x=N\}$. Since $w$ is strictly positive and its all moments are finite, we have
\[0< Z_{w, N} := \sum_{\omega\in \mathcal{C}_N(\mathfrak{X})} \left(\prod_{1\leq i < j\leq N}(x_i-x_j)^2\right)\prod_{i=1}^N w(x_i)<\infty, \]
where $x_1, \dots, x_N\in \mathfrak{X}$ satisfy $\omega_{x_i}=1$ for each $i=1, \dots, N$. We remark that each summand in $Z_{w, N}$ is well defined as a function of $\omega$, i.e., it does not depend on the order $x_1, \dots, x_N$. Then, $\mathbb{P}_{w, N}$ is given by
\begin{equation}\label{eq:op_ensemble}
    \mathbb{P}_{w, N}(\omega)=\frac{1}{Z_{w, N}}\left(\prod_{1\leq i< j \leq N}(x_i-x_j)^2\right)\prod_{i=1}^Nw(x_i) \quad (\omega\in \mathcal{C}_N(\mathfrak{X})).
\end{equation}

Later, we will deal with concrete examples of orthogonal polynomials.

\subsection{Orthogonal polynomials of hypergeometric type}\label{sec:op_hg}

We summarize necessary facts about orthogonal polynomials of \emph{hypergeometric type} (see \cite[Chapter 2]{NSU} for more details). Let us assume that $\mathfrak{X}$ is a uniform lattice, that is, $\mathfrak{X}=\mathbb{Z}+a$ or $\mathbb{Z}_{\geq 0}+a$ for some $a\in \mathbb{R}$. The following two kinds of difference operators on $\mathfrak{X}$ are defined:
\[[\Delta y](x):=y(x+1)-y(x), \quad [\nabla y](x):=y(x)-y(x-1) \quad (x\in \mathfrak{X}) \]
for any function $y$ on $\mathfrak{X}$, where $y(x-1):=0$ if $x-1\not\in \mathfrak{X}$. We consider the difference equation
\begin{equation}\label{eq:hyp_eq}
    \sigma(x)\Delta\nabla y(x)+\tau(x)\Delta y(x)+\lambda y(x) = 0,
\end{equation}
where $\sigma(x)$ and $\tau(x)$ are polynomials with $\deg \sigma(x)\leq 2$ and $\deg\tau(x)\leq 1$, and $\lambda$ is a constant. Clearly, a constant function solves Equation \eqref{eq:hyp_eq} with $\lambda=0$. Let $m_n:=\tau' n + \sigma '' n(n-1)/2$ for every $n\geq 0$. If $m_n\neq m_k$ for $k=0, \dots, n-1$, then there exists a polynomial solution $y_n(x)$ of Equation \eqref{eq:hyp_eq} with degree $n$ and $\lambda = -m_n$. Furthermore, we assume that the weight function $w$ satisfies that $\Delta[\sigma(x)w(x)]=\tau(x)w(x)$ for any $x\in \mathfrak{X}$ and the following boundary condition:
\begin{itemize}
    \item $\mathfrak{X}=\mathbb{Z}+a$: $x^n\sigma(x)w(x)\to 0$ as $x\to \pm \infty$ for $n=0, 1, \dots, 2N-1$,
    \item $\mathfrak{X}=\mathbb{Z}_{\geq 0}+ a$: $a^n\sigma(a) w(a)=0$ and $x^nw(x)\sigma(x)\to 0$ as $x\to \infty$ for $n=0, 1, \dots, 2N-1$.
\end{itemize}

In this case, $(y_n)_{n=0}^{N-1}$ belong to $L^2(\mathfrak{X}, w)$, and they are orthogonal. We call them (discrete) orthogonal polynomials of \emph{hypergeometric type}.

\medskip
We give fundamental examples. See \cite{KS, NSU} for more details. In what follows, we use the Pochhammer symbol $(a)_x:=a(a+1)\cdots (a+x-1)$. In first two cases, we assume that $\mathfrak{X}=\mathbb{Z}_{\geq 0}$.

\begin{example}[Meixner polynomials]\label{ex:Meixner}
    For $c>0$ and $\xi\in (0, 1)$ we define the negative binomial distribution $w_{c, \xi}$ on $\mathbb{Z}_{\geq 0}$ by
    \[w_{c, \xi}(x):=\frac{(c)_x \xi^x}{(1-\xi)^c x!} \quad (x\in \mathbb{Z}_{\geq 0}).\]
    The associated monic orthogonal polynomials, denoted by $\widetilde m^{(c, \xi)}_n(x)$, are called the \emph{Meixner polynomials}. They solve Equation \eqref{eq:hyp_eq} with $\sigma(x):=x$, $\tau(x):=-(1-\xi)x+c\xi$, and $\lambda = (1-\xi)n$ for all $n\geq 0$.
\end{example}

Let $K_{\mathrm{Meixner}(c, \xi), N}$ denote the normalized Christoffel--Darboux kernel on $\mathbb{Z}_{\geq 0}\times \mathbb{Z}_{\geq 0}$ given by Equation \eqref{eq:CD_kernel} with $p_n=m^{(c, \xi)}_n:=\widetilde m^{(c, \xi)}_n w_{c, \xi}^{1/2}/\|\widetilde m^{(c, \xi)}_n\|$ for $n=0, \dots, N-1$.

\begin{example}[Charlier polynomials]\label{ex:Charlier}
    For any $\mu>0$ the Poisson distribution $w_\mu$ on $\mathbb{Z}_{\geq 0}$ is defined by
    \[ w_\mu(x):=\frac{e^{-\mu}\mu^x}{x!} \quad (x\in \mathbb{Z}_{\geq 0}).\]
    The associated monic orthogonal polynomial, denoted by $\widetilde c^{(\mu)}_n(x)$, is called the \emph{Charlier polynomials}. They solve Equation \eqref{eq:hyp_eq} with $\sigma(x):=x$, $\tau(x):=\mu -x$, and $\lambda = n$ for all $n\geq0$.
\end{example}

Let $K_{\mathrm{Charlier}(\mu), N}$ denote the normalized Christoffel--Darboux kernel on $\mathbb{Z}_{\geq 0}\times \mathbb{Z}_{\geq 0}$ given by Equation \eqref{eq:CD_kernel} with $p_n=c^{(\mu)}_n:=\widetilde c^{(\mu)}_n w_\mu^{1/2}/\|\widetilde c^{(\mu)}_n\|$ for $n=0, \dots, N-1$.

\medskip
The following are orthogonal polynomials on a finite set $\mathfrak{X}=\{0, 1, \dots, M\}$.
\begin{example}[Krawtchouk, Hahn, Racah polynomials]
    The weight functions $w$ on $\mathfrak{X}$ for the \emph{Krawtchouk}, \emph{Hahn}, and \emph{Racah} polynomials are given by 
    \[w(x):=\begin{dcases}
        \binom{M}{x}p^x(1-p)^{M-x} & \text{Krawtchouk case}, \\
        \binom{a+x}{x}\binom{M+b-x}{M-x} & \text{Hahn case}, \\
        \frac{(\gamma+\delta+1)_x((\gamma+\delta+3)/2)_x(\alpha+1)_x(\beta+\delta+1)_x(\gamma+1)_x}{x!((\gamma+\delta+1)/2)_x(\gamma+\delta-\alpha+1)_x (\gamma-\beta+1)_x(\delta+1)_x} & \text{Racah case}.
    \end{dcases}\]
    Here, the above parameters satisfy 
    \[p\in (0, 1), \quad a, b>-1,\]
    \[\alpha+1=-M \quad \text{or} \quad \beta+\delta+1=-M \quad \text{or} \quad \gamma+1=-M.\]
    These orthogonal polynomials solve Equation \eqref{eq:hyp_eq} with
    \[\sigma(x):=\begin{dcases}
    (1-p)x & \text{Krawtchouk case}, \\
    x(M+b+1-x) & \text{Hahn case} \\ 
    \frac{x(x+\delta)(\beta-\gamma-x)(x-\alpha+\gamma+\delta)}{(2x+\gamma+\delta+1)(2x+\gamma+\delta)} & \text{Racah case},
    \end{dcases}\]
    \[\tau(x)+\sigma(x):=\begin{dcases}
        p(M-x) & \text{Krawtchouk case}, \\
        (x+a+1)(M-x) & \text{Hahn case}, \\
        \frac{(-\alpha-1-x)(x+\gamma+1)(x+\gamma+\delta+1)(x+\beta+\delta+1)}{(2x+\gamma+\delta+1)(2x+\gamma+\delta+2)} & \text{Racah case}.
    \end{dcases}\]
    Moreover, for the $n$-th Krawtchouk, Hahn, and Racah polynomial ($n=0, \dots, M$), 
    \[\lambda= \begin{dcases}
        n & \text{Krawtchouk case}, \\
        n(n+a+b+1) & \text{Hahn case}, \\
        n(n+\alpha+\beta+1) & \text{Racah case}.
    \end{dcases}\]
\end{example}

Let $N\leq M$. Hence, there exist Krawtchouk, Hahn, and Racah orthogonal polynomials of degrees $0, \dots, N-1$. We denote by $K_{\mathrm{Krawtchouk}_M(p), N}$, $K_{\mathrm{Hahn}_M(a, b), N}$, and $K_{\mathrm{Racah}_M(\alpha, \beta, \gamma, \delta), N}$ the normalized Christoffel--Darboux kernel on $\mathfrak{X}\times \mathfrak{X}$ given by Equation \ref{eq:CD_kernel} with the associated normalized orthonormal polynomials of degrees $0, \dots, N-1$, respectively.

\medskip 
Let $(\widetilde p_n(x))_{n}$ be orthogonal polynomials with respect to $w$ of hypergeometric type with $\sigma(x)$ and $\tau(x)$. Namely, they satisfy that $\mathcal{D}\widetilde p_n=m_n\widetilde p_n$ for each $n$, where $\mathcal{D}:=\sigma(x)\Delta\nabla +\tau(x)\Delta$ and $m_n:=\tau'n+\sigma''n(n-1)/2$. More precisely, the domain of $\mathcal{D}$ is defined as 
\[\mathrm{dom}(\mathcal{D}):=\{y\in L^2(\mathfrak{X}, w)\mid \mathcal{D}y\in L^2(\mathfrak{X}, w)\}.\] 
In what follows, we assume that $\mu_x:=\sigma(x)>0$ on $\mathfrak{X}$, and hence, $\lambda_x:= \sigma(x)+\tau(x)=\sigma(x+1)w(x+1)/w(x)>0$ on $\mathfrak{X}$. Then, we define $D:=w^{1/2} \mathcal{D} w^{-1/2}$ on $\ell^2(\mathfrak{X})$. Namely, for any finitely supported function $f$ on $\mathfrak{X}$,
\begin{equation}\label{eq:difference_operator}
    [Df](x) = \sqrt{\mu_{x+1}\lambda_x} f(x+1) - (\mu_x + \lambda_x) f(x) +\sqrt{\mu_x \lambda_{x-1}}f(x-1) \quad (x\in \mathfrak{X}).
\end{equation}
By \cite[Lemma 3.1]{S24}, $D$ is closed symmetric. Moreover, its spectrum is given by $\{m_n\}_n$. Thus, $D$ is self-adjoint. We remark that $e^{\beta D}$ is of trace class for all $\beta>0$ in the above examples, i.e., $\mathrm{Tr}(e^{\beta \mathcal{D}})=\sum_{n=0}^\infty e^{\beta m_n}<\infty$ holds.

We suppose that $m_0=0>m_1>\cdots$ as in the above examples and set $-\mu \in (m_N, m_{N-1})$. Then, $K_{w, N}$ is equal to the spectral projection onto $\{D+\mu>0\}$.

\medskip
Askey--Lesky polynomials played an important role in the harmonic analysis of the infinite-dimensional unitary group $U(\infty)$ in \cite{Olshanski03,BO05}. As similar to the above two examples, Askey--Lesky polynomials are also eigenfunctions of a self-adjoint operator $D$. However, $D$ can not be diagonalized by these polynomials. In fact, the existence is guaranteed for only a finite number of polynomials, while the state space $\mathfrak{X}$ is infinite. Let us explain below in more detail.

A pair $z, z'\in \mathbb{C}$ is said to be
\begin{itemize}
    \item \emph{principal} if $z, z'\in \mathbb{C}\backslash \mathbb{R}$ and $z' = \bar{z}$,
    \item \emph{complementary} if $z, z'\in \mathbb{R}$ and $k < z, z' < k+1$ for some $k\in \mathbb{Z}$.
\end{itemize}
Either $(z, z')$ is principal or complementary if and only if $(z+k)(z'+k)>0$ for all $k\in \mathbb{Z}$.

\begin{example}[Askey--Lesky polynomials, see \cite{Olshanski03, BO05}]\label{ex:AL}
    Let $\mathfrak{X}=\mathbb{Z}$ and $u, u', w, w'\in \mathbb{C}$ such that each pair $u, u'$ and $w, w'$ are principal or complementary. A weight function $w:=w_{u, u', w, w'}$ on $\mathbb{Z}$, called the \emph{Askey--Lesky weight function}, is defined by
    \[w(x):=\frac{1}{\Gamma(u-x+1)\Gamma(u'-x+1)\Gamma(w+x+1)\Gamma(w'+x+1)}.\]
    By \cite[Lemma 3.2]{Olshanski03}, $w(x)>0$ holds for all $x\in \mathfrak{X}$. Moreover, we have $\Delta[\sigma(x)w(x)]=\tau(x)w(x)$ on $\mathfrak{X}$, where
    \[\sigma(x):=(x+w)(x+w'), \quad \tau(x):=-(u+u'+w+w')x + uu'-ww'.\]
    We suppose that $u+u'+w+w'> 2N+1$. For $n=0, \dots, N-1$, there exists a polynomial solution of Equation \eqref{eq:hyp_eq} with $\lambda=-m_n= (u+u'+w+w')x + uu'-ww'$, and they are orthogonal with respect to $w$. Following \cite{Olshanski03,BO05}, we call them the \emph{Askey--Lesky polynomials}. The associated difference operator $D$ is given by Equation \eqref{eq:difference_operator} with
    \[\mu_x:=\sigma(x)=(x+w)(x+w'), \quad \lambda_x:=\sigma(x)+\tau(x)=(x-u)(x-u').\]
\end{example}

\subsection{Orthogonal polynomial ensembles at finite temperature}
We use the same notation as in the previous section. In particular, we obtain the orthogonal polynomial ensemble $\mathbb{P}_{w, N}$ by Equation \ref{eq:op_ensemble}, and its correlation kernel $K_{w, N}$ is given by Equation \eqref{eq:CD_kernel}. Now, we apply the results in Section \ref{sec:density_op_fermion_Fock_sp} for $H=-D$.

Let $\mu\in \mathbb{R}$ and $\varphi_{-D-\mu, \beta}$ denote the quasi-free state on the CAR algebra $\mathfrak{A}(\mathfrak{X})$ associated with $K_{-D-\mu, \beta}$ (see Equation \eqref{eq:H_K}). Through the inclusion $C(\mathcal{C}(\mathfrak{X}))\cong \mathfrak{D}(\mathfrak{X})\subset \mathfrak{A}(\mathfrak{X})$, the state $\varphi_{-D-\mu, \beta}$ produces a DPP $\mathbb{P}_{-D-\mu, \beta}$ on $\mathfrak{X}$ whose correlation kernel is given by $K_{-D-\mu, \beta}$.

\begin{definition}
    The DPP $\mathbb{P}_{-D-\mu, \beta}$ is called the \emph{orthogonal polynomial ensemble at inverse temperature $\beta$}.
\end{definition}

As mentioned in Remark \ref{rem:limit_trans_projection}, $K_{-D-\mu, \beta}$ strongly converges to the spectral projection onto $\{D+\mu>0\}$ as $\beta$ tends to infinity. Moreover, as observed in the previous section, if the spectrum of $D$ forms a non-increasing sequence $m_0>m_1>\cdots$ and $-\mu\in (m_N, m_{N-1})$, then this spectral projection is nothing but the $K_{w, N}$. Therefore, $\mathbb{P}_{-D-\mu, \beta}$ weakly converges to $\mathbb{P}_{w, N}$. This observation justifies the term ``at inverse temperature.''

\begin{example}
    Let $\mathfrak{X}=\mathbb{Z}_{\geq 0}$ and $w$ the negative binomial distribution or the Poisson distribution. Thus, the associated orthogonal polynomials $(\tilde p_n(x))_{n=0}^\infty$ are the Meixner polynomials or Charlier polynomials (see Examples \ref{ex:Meixner} \ref{ex:Charlier}). In both cases, there exists a self-adjoint operator $D$ who can be diagonalized by the normalized orthogonal polynomials $(p_n)_{n=0}^\infty$, and those eigenvalues $(m_n)_{n=0}^\infty$ form a decreasing sequence. For fixed $N\geq 1$ we take $-\mu\in (m_N, m_{N-1})$. Then, $K_w^{\beta, \mu}$ strongly converges to $K_{w, N}$ as $\beta\to \infty$.
\end{example}

In the Askey--Lesky polynomials case, we can not completely determine the spectrum of $D$. However, we obtain the quasi-free state $\varphi_{-D-\mu, \beta}$ on $\mathfrak{A}(\mathbb{Z})$ for all $\beta, \mu\in \mathbb{R}$. Therefore, the Askey--Lesky polynomial ensemble $\mathbb{P}_{-D-\mu, \beta}$ at inverse temperature $\beta$ can be defined. However, the author can not reveal whether $\mathbb{P}_{-D-\mu, \beta}$ weakly converges to $\mathbb{P}_{w_{u, u', w, w'}, N}$ as $\beta$ tends to $\infty$ for an appropriate $\mu$ and $N\geq 1$.

\subsection{Limiting behavior: point processes from continuous orthogonal polynomials}\label{sec:limit_OP}

In this section, we discuss determinantal point processes on $\mathbb{Z}_{\geq 0}$ derived from orthogonal polynomials with respect to non-atomic measures on $\mathbb{R}$ through the duality in the sense of \cite{BO17,KS22}.

We assume that $w\colon \mathbb{R}\to \mathbb{R}_{\geq 0}$ satisfies the following two conditions:
\begin{itemize}
    \item[(1)] $\int_\mathbb{R} u^n w(u)du <\infty$ for every $n=0, 1, \dots$,
    \item[(2)] the moment problem for $w(u)du$ is determinate\footnote{Namely, $w(u)du$ is a Radon measure uniquely determined by its moments.}.
\end{itemize}
By Condition (1), there exist monic polynomials $(p_x(u))_{x=0}^\infty$ which are orthogonal in $L^2(\mathbb{R}, w)$, that is, $\int_\mathbb{R}p_x(u)p_y(u)w(u)du=0$ if $x\neq y$. Moreover, by Condition (2), $\tilde p_x := p_x/\|p_x\|_{L^2(\mathbb{R}, w)}$ form an orthonormal basis for $L^2(\mathbb{R}, w)$ (see e.g., \cite[Theorem 6.10]{Schmuden}). Thus, we obtain the unitary map $U\colon L^2(\mathbb{R}, w)\to \ell^2(\mathbb{Z}_{\geq 0})$ by $U\tilde p_x:=\delta_x$ for all $x\in \mathbb{Z}_{\geq 0}$, where $(\delta_x)_{x\in \mathbb{Z}_{\geq 0}}$ is the canonical orthonormal basis for $\ell^2(\mathbb{Z}_{\geq 0})$.

Let $I\subset \mathbb{R}$ be an interval and $K_{w, I}$ the orthogonal projection from $\ell^2(\mathbb{Z}_{\geq 0})$ onto $UL^2(I, w|_I)$. It gives rise to the determinantal point process $\mathbb{P}_{w, I}$ on $\mathbb{Z}_{\geq 0}$ with correlation kernel $K_{w, I}$, called the \emph{discrete ensemble associated with $(p_x)_{x\in \mathbb{Z}_{\geq 0}}$}. The kernel $K_{w, I}$ can be expressed explicitly by
\begin{equation}\label{eq:int_kernel}
    K_{w, I}(x, y)=\int_I \tilde p_x(u) \tilde p_y(u)w(u)du.
\end{equation}

The following three-therm recurrence relation is well known
\[u\tilde p_x(u)=a_x\tilde p_{x+1}(u)+b_x \tilde p_x(u) + a_{x-1}\tilde p_{x-1}(u)\]
for some $(a_x)_{x\geq 0}$ and $(b_x)_{x\geq 0}$, where $a_{-1}$ and $p_{-1}(t)$ are supposed to be zero. Thus, an operator $T$ on $\ell^2(\mathbb{Z})$ defined by 
\begin{equation}\label{eq:multi_operator}
    T\delta_x:=a_x\delta_{x+1}+b_x\delta_x+a_{x-1}\delta_{x-1}
\end{equation}
is unitarily equivalent to the multiplication operator by $u$ on $L^2(\mathbb{R}, w)$. Moreover, $T$ is essentially self-adjoint, denoted by $T$ its self-adjoint extension, and the spectrum of $T$ coincides with the support of $w$. For any $r\in \mathbb{R}$ the spectral projection onto $\{\pm(T-r)>0\}$ is given by the kernel $K_{w, (r, \infty)}$ or $K_{w, (-\infty, r)}$, respectively. See \cite[Proposition 5.2]{BO17}.

Similarly to discrete orthogonal polynomial ensembles, we introduce a finite temperature version of discrete ensemble associated with $(p_x)_{x\in \mathbb{Z}_{\geq 0}}$.
\begin{definition}
    A determinantal point process with correlation kernel $e^{\mp\beta(T-r)}(1+e^{\mp(T-r)})^{-1}$ is called the discrete ensemble associated with $(p_x)_{x\in \mathbb{Z}_{\geq 0}}$ \emph{at inverse temperature $\beta$}.
\end{definition}

As $\beta\to\infty$, the discrete ensemble associated with $(p_x)_{x\in \mathbb{Z}_{\geq 0}}$ weakly converges to the discrete ensemble with correlation kernel $K_{w, (r, \infty)}$ and $K_{w, (-\infty, r)}$, respectively.

\medskip
Let us collect several examples from \cite{BO17}.

\begin{example}[discrete Hermite ensembles]
    Let $w(u):=e^{-u^2}$ for all $u\in \mathbb{R}$. The associated monic orthogonal polynomials are called the \emph{Hermite polynomials}, denoted by $(h_x(u))_{x=0}^\infty$, and they satisfy $\|h_x\|_{L^2(\mathbb{R}, w)}^2=\pi^{1/2}2^{-x}x!$ for every $x=0, 1, \dots$. Thus, for any interval $I\subset \mathbb{R}$, the kernel $K_{\mathrm{DHermite}, I}$ in Equation \eqref{eq:int_kernel} is equal to  
    \[K_{\mathrm{DHermite}, I}(x, y)=\frac{1}{(\pi 2^{-x-y}x!y!)^{1/2}}\int_I h_x(u)h_y(u)e^{-u^2}du.\]
    The determinantal point process on $\mathbb{Z}_{\geq 0}$ with correlation kernel $K_{\mathrm{DHermite}, I}$ is called the \emph{discrete Hermite ensemble}.
\end{example}

\begin{example}[discrete Laguerre enesmbles]
    Let $c>0$ and $w(u)=u^{c-1}e^{-u}$ for $u\in [0, \infty)$. The associated monic orthogonal polynomials are called the \emph{Laguerre polynomials}, denoted by $(l^{(c)}_x(u))_{x=0}^\infty$, and they satisfy $\|l^{(c)}_x\|_{L^2([0, \infty), w)}^2=x!\Gamma(x+c)$ for every $x=0, 1, \dots$. Thus, for any interval $I\subset [0, \infty)$, the kernel $K_{\mathrm{DLaguerre}(c), I}$ in Equation \eqref{eq:int_kernel} is equal to 
    \[K_{\mathrm{DLaguerre}(c), I}(x, y)=\frac{1}{(x!\Gamma(x+c)y!\Gamma(y+c))^{1/2}}\int_Il^{(c)}_x(u)l^{(c)}_y(u)u^{c-1}e^{-u}du.\]
    The determinantal point process on $\mathbb{Z}_{\geq 0}$ with correlation kernel $K_{\mathrm{DLaguerre}(c), I}$ is called the \emph{discrete Laguerre ensemble}.
\end{example}

\begin{example}[discrete Jacobi ensembles]
    Let $a, b>-1$ and $w(u):=(1-u)^a(1+u)^b$ for $u\in (-1, 1)$. The associated monic orthogonal polynomials are called the \emph{Jacobi polynomials}, denoted by $(j^{(a, b)}_x(u))_{x=0}^\infty$, and they satisfy
    \[\|j_x^{(a, b)}\|:=\|j_x^{(a, b)}\|_{L^2([-1, 1], w)}^2=\frac{2^{a+b+1}}{2x+a+b+1}\frac{\Gamma(x+a+1)\Gamma(x+b+1)}{x!\Gamma(x+a+b+1)}\]
    for every $x=0, 1, \dots$. Thus, for any interval $I\subset [-1, 1]$, the kernel $K_{\mathrm{DJacobi}(a, b), I}$ in Equation \eqref{eq:int_kernel} is equal to 
    \[K_{\mathrm{DJacobi}(a, b), I}(x, y)=\frac{1}{\|j^{(a, b)}_x\|\|j^{(a, b)}_y\|}\int_Ij^{(a, b)}_x(u)j^{(a, b)}_y(u)(1-u)^a(1+u)^bdu.\]
    The determinantal point process on $\mathbb{Z}_{\geq 0}$ with correlation kernel $K_{\mathrm{DJacobi}(a, b), I}$ is called the \emph{discrete Jacobi ensemble}.
\end{example}

As shown in \cite{BO17}, discrete Hermite/Laguerre/Jacobi ensembles can be obtained as limit transitions of Meixner, Charlier, Krawtchouk, Hahn, and Racah polynomial ensembles at certain regimes. The essential idea is to show strong resolvent convergences of difference operators of orthogonal polynomials. Let us explain this idea briefly in the case of limit transition from a Charlier polynomial ensemble to a discrete Hermite ensemble.

Let $D$ denote the self-adjoint operator in the Charlier case with parameter $\mu$, i.e.,
\[[D f](x):=\sqrt{\mu(x+1)}f(x+1)-(x+\mu)f(x)+\sqrt{\mu x}f(x-1).\]
We have $D c^{(\mu)}_n=-n c^{(\mu)}_n$ for all $n\geq 1$ (see Example \ref{ex:Charlier}). Thus, $K_{\mathrm{Charlier}(\mu), N}$ coincides with the spectral projection onto $\{\sqrt{2N}^{-1}(D+N)>0\}$. For all $r\in \mathbb{R}$, we consider the limit transition regime where $\mu, N\to \infty$ with $(\mu-N)/\sqrt{N}\to \sqrt{2}r$. For every $x\in \mathbb{Z}_{\geq 0}$ we have 
\[\lim \frac{1}{\sqrt{2N}}(D+N)f(x)=\sqrt{\frac{x+1}{2}}f(x+1)-r f(x)+\sqrt{\frac{x}{2}}f(x-1)=(T-r)f(x),\]
where $T$ is given by Equation \eqref{eq:multi_operator} in the Hermite polynomials case. Moreover, $\sqrt{2N}^{-1}(D+N)$ converges to $T-r$ in the strong resolvent sense. It implies that the strong convergence of the spectral projections. In particular, $K_{\mathrm{Charlier}(\mu), N}$ strongly converges to $K_{\mathrm{DHermite}, (r, \infty)}$.

The same argument implies following results.
\begin{example}[{Meixner $\to$ discrete Hermite (\cite[Theorem 6.3]{BO17})}]\label{ex:M_dH}
    $\sqrt{2\xi\beta}^{-1}(D+(1-\xi)N)$ converges to $T-r$ when $\xi c\to \infty$ and $(1-\xi)N/\sqrt{\xi c}-\sqrt{\xi c}\to -\sqrt{2}r$, where $D$ and $T$ are the self-adjoint operators associated with $\mathrm{Mexiner}(c, \xi)$ and $\mathrm{DHermite}$, respectively. Hence, $K_{\mathrm{Meixner}(c, \xi), N}$ converges to $K_{\mathrm{DHermite}, (r, \infty)}$ strongly in this regime.
\end{example}

Let $s$ denote the linear operator on $\ell^2(\mathbb{Z}_{\geq 0})$ defined by $s\delta_x:=(-1)^x\delta_x$ for all $x\in \mathbb{Z}_{\geq 0}$.
\begin{example}[{Meixner $\to$ discrete Laguerre (\cite[Theorem 6.2]{BO17})}]\label{ex:M_dL}
    $D+(1-\xi)N$ converges to $-s(T-r)s$ when $(1-\xi)N\to r\in \mathbb{R}$, where $D$ and $T$ are the self-adjoint operators associated with $\mathrm{Meixner}(c, \xi)$ and $\mathrm{DLaguerre}(c)$, respectively. Hence, $K_{\mathrm{Meixner}(c, \xi), N}$ converges to $K_{\mathrm{DLaguerre}(c), (-\infty, r)}$ strongly in this regime.
\end{example}

In the case of Krawtchouk, Hahn, and Racah polynomials, the self-adjoint operator $D$ is originally defined on $\ell^2(\{0, \dots, M\})$. However, in the following three examples, $D$ extends into $\ell^2(\mathbb{Z}_{\geq 0})=\ell^2(\{0, \dots, M\})\oplus \ell^2(\{M+1, M+2, \dots\})$, where it coincides with $m_N-1$ on the second direct summand.

\begin{example}[{Krawtchouk $\to$ discrete Hermite (\cite[Theorem 6.4]{BO17})}]\label{ex:K_dH}
    $\sqrt{2pM}^{-1}(D+N)$ converges to $T-r$ when $pM\to \infty$ and $-\sqrt{pM}+N/\sqrt{pM}\to -\sqrt{2}r$, where $D$ and $T$ are the self-adjoint operators associated with $\mathrm{Krawtchouk}_M(p)$ and $\mathrm{DHermite}$, respectively. Hence, $K_{\mathrm{Krawtchouk}_M(p), N}$ converges to $K_{\mathrm{DHermite}, (r, \infty)}$ strongly in this regime.
\end{example}

\begin{example}[{Hahn $\to$ discrete Laguerre (\cite[Theorem 6.5]{BO17})}]\label{ex:H_dL}
    $M^{-1}(D+N(N+a+b+1))$ converges to $-s(T-r)s$ when $N, M\to \infty$ such that $rM/N^2\to 1$, where $D$ and $T$ are the self-adjoint operators associated with $\mathrm{Hahn}_M(a, b)$ and $\mathrm{DLaguerre}(a+1)$, respectively. Hence, $K_{\mathrm{Hahn}_M(a, b), N}$ converges to $K_{\mathrm{DLaguerre}(a+1), (-\infty, r)}$ strongly in this regime.
\end{example}

\begin{example}[{Racah $\to$ discrete Jacobi (\cite[Theorem 6.6]{BO17})}]\label{ex:R_dJ}
    $2M^{-2}(D+N+\alpha+\beta+1)$ converges to $T-r$ when 
    \[N, M\to \infty, \quad \frac{N}{M}\to \sqrt{\frac{1-r}{2}}, \quad \alpha=-M-1, \quad \beta=M+a+1, \quad \gamma=a, \quad \delta=b,\]
    where $D$ and $T$ are the self-adjoint operators in $\mathrm{Racah}_M(\alpha, \beta, \gamma, \delta)$ and $\mathrm{DJacobi}(a, b)$, respectively. Hence, $K_{\mathrm{Racah}_M(\alpha, \beta, \gamma, \delta), N}$ converge to $K_{\mathrm{DJacobi}(a, b), (r, \infty)}$ strongly in this regime.
\end{example}

As observed in Section \ref{sec:LT_by_strong_res_conv}, the strong resolvent convergence pf self-adjoint operators induces the convergence of stochastic processes. We will return to this point in Section \ref{sec:ltcpop}.

\section{Stochastic processes from orthogonal polynomials}\label{sec:stoc_from_op}
We apply the theory of stochastically positive KMS systems developed in the previous sections to the setting from orthogonal polynomials and construct stochastic processes on the point configuration space. The resulting stochastic processes contain two kinds of parameters, which are the inverse temperature $\beta$ and ones concerning orthogonal polynomials. Accordingly, we will study two types of limiting behavior of stochastic processes relating these parameters.

\subsection{Stochastic processes from orthogonal polynomial ensembles}
We return to the same setting in Section \ref{sec:op_hg}. Namely, $\mathfrak{X}\subset \mathbb{R}$ is a uniform lattice equipped with a weight function $w$ with finite moments of all orders. Moreover, the associated orthogonal polynomials $(\tilde p_n(x))_{n\geq 0}$ are of hypergeometric type, and we have $\mathcal{D}\tilde p_n=m_n\tilde p_n$ for all $n\geq 0$. Here, $\mathcal{D}:=\sigma(x)\Delta\nabla+\tau(x)\Delta$ and $m_n:=\tau' n + \sigma'' n(n-1)/2$ for every $n\geq 0$. Similar to the previous section, we define $D:=w^{1/2}\mathcal{D}w^{-1/2}$ and $p_n(x):= \tilde p_n(x)w(x)^{1/2}/\|\tilde p_n\|_{L^2(\mathfrak{X}, w)}$ for any $x\in \mathfrak{X}$ and $n\geq 0$. Thus, $(p_n)_{n\geq 0}$ form an orthonormal basis for $\ell^2(\mathfrak{X})$ and $Dp_n = m_n p_n$ holds. 

\medskip
Let us explain a few notions. We denoted by $C_b(\mathfrak{X})$ the space of bounded (continuous) functions on $\mathfrak{X}$. A linear map $P$ on $C_b(\mathfrak{X})$ is said to be \emph{positive} if $Pf\geq 0$ for any $f\in C_b(\mathfrak{X})$ with $f\geq 0$. In addition, $P\mathbbm{1}=\mathbbm{1}$, then $P$ is called \emph{Markov}. A \emph{Markov semigroup} $(P_t)_{t\geq 0}$ is a semigroup of Markov operators and $P_0$ is the identity operator.  

\medskip
Throughout this section, we assume the following twofold:
\begin{itemize}
    \item $e^{\beta D}$ is a trace class operator for some $\beta>0$,
    \item $(e^{t\mathcal{D}})_{t\geq 0}$ forms a Markov semigroup on $\mathfrak{X}$.
\end{itemize}

\begin{remark}
    Let us recall that for any finite support function $f$ on $\mathfrak{X}$ 
    \[[\mathcal{D}f](x)=\lambda_xf(x+1)-(\lambda_x+\mu_x)f(x)+\mu_xf(x-1),\]
    where $\lambda_x=\sigma(x)+\tau(x)=\sigma(x+1)w(x+1)/w(x)$ and $\mu_x=\sigma(x)$. If $\mu_x>0$ for all $x\in \mathfrak{X}$, then we have $\lambda_x>0$. Moreover, the Markov semigroup generated by $\mathcal{D}$ gives a \emph{birth-and-death process} with a hopping rate of $\lambda_x$ to the right and a hopping rate of $\mu_x$ to the left.
\end{remark}

Let us recall that $\varphi_{-D-\mu, \beta}$ is the quasi-free state on $\mathfrak{A}(\mathfrak{X})$ associated with $K_{-D-\mu, \beta}$. By the first assumption, it is given by Equation \eqref{eq:qf_state_density_op}. Moreover, $\varphi_{-D-\mu, \beta}$ is a $\alpha^{-D-\mu}$-KMS state at $\beta$, where the $\mathbb{R}$-flow $\alpha^{-D-\mu}\colon \mathbb{R}\curvearrowright \mathfrak{A}(\mathfrak{X})$ is given by $\alpha^{-D-\mu}_t(a(h)):=a(e^{-\mathrm{i}t (D+\mu)}h)$ for all $t\in \mathbb{R}$ and $h\in \ell^2(\mathfrak{X})$. The associated family $\{\mathcal{S}_{A_1, \dots, A_n}\}_{A_1, \dots, A_n\in \mathfrak{A}(\mathfrak{X})}$ is given by Equation \eqref{eq:schwinger_func_trace}. 

By the second assumption and the Karlin--McGregor formula (see \cite{KMc57,KMc59}), we have 
\[\det[\langle e^{\beta D}\delta_{x_i}, \delta_{y_j}\rangle \rangle]_{i, j=1}^n\geq 0\] 
for any $x_1>\dots>x_n, y_1>\dots>y_n\in \mathfrak{X}$ and $n\geq 1$. In fact, it is proportional with a positive coefficient to the probability that $n$ independent Markov processes generated by $\mathcal{D}$ start from $x_1>\cdots >x_n$ and reach $y_1>\cdots > y_n$ at time $t$ without coincident. 

The following is a consequence of Theorems \ref{thm:stochastic_process}, \ref{thm:dynamical_correlation}:

\begin{proposition}\label{prop:op_stoc_proc}
    $\varphi_{-D-\mu, \beta}$ is stochastically positive with respect to $\mathfrak{D}(\mathfrak{X})$. The associated stochastic process $(X^{(\beta)}_t)_{t\in [-\beta/2, \beta/2]}$ on $\mathcal{C}(\mathfrak{X})$ by Theorem \ref{thm:stochastic_process} satisfies
    \[\mathbb{P}[(X^{(\beta)}_{t_1})_{x_1}=1, \dots, (X^{(\beta)}_{t_n})_{x_n}=1] = \det[\mathcal{K}_{-D-\mu, \beta}(x_i, t_i; x_j, t_j)]_{i, j=1}^n\]
    for any $x_1, \dots, x_n\in \mathfrak{X}$ and $-\beta/2\leq t_1\leq \cdots \leq t_n\leq \beta/2$. Here, $\mathcal{K}_{-D-\mu, \beta}$ is defined as follows:
    \[\mathcal{K}_{-D-\mu, \beta}(x, t; y, s):=\begin{dcases} \sum_{n=0}^\infty p_n(x)p_n(y)\frac{e^{(\beta+t-s)(m_n+\mu)}}{1+e^{\beta (m_n+\mu)}} & t\leq s, \\ -\sum_{n=0}^\infty p_n(x)p_n(y) \frac{e^{(t-s)(m_n+\mu)}}{1+e^{\beta (m_n+\mu)}} & t>s.\end{dcases}\]
\end{proposition}
\begin{proof}
    The first statement follows from Lemma \ref{lem:cond_stoc_positivity}. Using orthonormal basis $(p_n)_{n=0}^\infty$, we obtain the expression of the space-time correlation kernel $\mathcal{K}_{-D-\mu, \beta}$ by Equation \ref{eq:zero_temp_limit}.
\end{proof}

Let us discuss concrete examples. We suppose that $\mathfrak{X}=\mathbb{Z}_{\geq 0}$ and $w$ is the negative binomial distribution or the Poisson distribution (see Exapmples \ref{ex:Meixner}, \ref{ex:Charlier}). Then, $e^{\beta D}$ is of trace class for any $\beta>0$, i.e., these examples satisfy the above first condition. Furthermore, $\mathcal{D}$ generates a Markov process (birth-and-death process) on $\mathfrak{X}$. See \cite[Corollary 6.6]{BO13}\footnote{For the Meixner case, the parameter $\xi$ is replaced by $r/(1+r)$ in \cite{BO13}. These references did not address the Charlier case, but it follows the same argument as others.}. More precisely, there exists a Markov semigroup $(P_t)_{t\geq 0}$ on $\mathfrak{X}$ satisfying Kolmogorov's backward and forward equations with $\mathcal{D}$. In these case, $(P_t)_{t\geq 0}$ preserves the Banach space $C_0(\mathfrak{X})$ of functions on $\mathfrak{X}$ vanishing at infinity, and the infinitesimal generator of $(P_t|_{C_0(\mathfrak{X})})_{t\geq 0}$ coincides with $\mathcal{D}$.Thus, the above second condition holds true. Namely, we obtain the stochastic process satisfying the determinantal formula in Proposition \ref{prop:op_stoc_proc}.

\medskip
In this section, the stochastic positivity of $\varphi_{-D-\mu}$ relies on the existence of a stochastic dynamics $(e^{t\mathcal{D}})_{t\geq 0}$ on the underlying space $\mathfrak{X}$ and the Karlin--McGregor formula. The Karlin--McGregor formula, which extends stochastic dynamics on $\mathfrak{X}$ to one on $\mathcal{C}(\mathfrak{X})$, has commonly been discussed in previous works as well as our research. In this paper, we have revealed an algebraic aspect of such a construction of stochastic processes on $\mathcal{C}(\mathfrak{X})$. Due to the algebraic nature, we obtained the above determinantal formulas of space-time correlations of stochastic processes on $\mathcal{C}(\mathfrak{X})$. Furthermore, as demonstrated in the next sections, we can analyze limiting behavior of stochastic processes in an algebraic way. 

\subsection{Limiting behavior as temperature tends to zero}
We apply the results in Section \ref{sec:zero_temp_limit} for the current setup. Let us suppose $-\mu \in (m_N, m_{N-1})$. Thus, we have $K_{-D-\mu, \beta}\to K_{w, N}$ as $\beta\to \infty$, and hence, $\varphi_{-D-\mu, \beta}$ converges to $\varphi_{w, N}$ in the weak${}^*$ topology. Here, the quasi-free state $\varphi_{w, N}$ on $\mathfrak{A}(\mathfrak{X})$ was introduced in Section \ref{sec:ope}. 

\begin{proposition}\label{prop:process_zero_temp}
    $\varphi_{w, N}$ is stochastically positive. As $\beta\to\infty$, the stochastic process associated with $\varphi_{-D-\mu, \beta}$ converges to the stochastic process $(X_t)_{t\in \mathbb{R}}$ associated with $\varphi_{w, N}$, and it satisfies 
    \[\mathbb{P}[(X_{t_1})_{x_1}=1, \dots, (X_{t_n})_{x_n}=1] = \det[\mathcal{K}_{-D, N}(x_i, t_i; x_j, t_j)]_{i, j=1}^n\]
    for every $x_1, \dots, x_n \in \mathfrak{X}$ and $t_1\leq \cdots \leq t_n $, where 
    \[\mathcal{K}_{-D, N}(x, t; y, s)=\begin{dcases} \sum_{n=0}^{N-1} e^{-(s-t)m_n} p_n(x)p_n(y) & t\leq s, \\ -\sum_{n=N}^\infty e^{(t-s)m_n} p_n(x)p_n(y) & t>s.\end{dcases}\]
\end{proposition}
\begin{proof}
    The first statement follows from Lemma \ref{lem:convergence}. For any $x, y\in \mathfrak{X}$ and $t, s\in \mathbb{R}$, we have 
    \[R_{-D-\mu}(x, t; y, s)=e^{t\mu}\mathcal{K}_{-D, N}(s, t; y, s)e^{-s\mu}.\]
    Here, the transformation of the kernel multipliying $e^{t\mu}$ and $e^{-s\mu}$ does not affect the determinant in the statement. Namely, we have 
    \begin{align*}
        \mathbb{P}[(X_{t_1})_{x_1}=1, \dots, (X_{t_n})_{x_n}=1] 
        & = \det[R_{-D-\mu}(x_i, t_i; x_j, t_j)]_{i, j=1}^n\\
        &= \det[\mathcal{K}_{-D, N}(x_i, t_i; x_j, t_j)]_{i, j=1}^n.
    \end{align*}
\end{proof}

As mentioned in the previous section, this proposition is applicable to Meixner and Charlier cases. Here, we would like to emphasize that the appearance of space-time correlation kernels in the aforementioned form is well known in both harmonic analysis and random matrix theory. See \cite{BO06, KT} etc. In this paper, we have established an algebraic derivation of such space-time correlation kernels. 

\medskip
Furthermore, we can show the Markov property of the above stochastic process $(X_t)_{t\in \mathbb{R}}$ by an algebraic argument. The perfectness of the orthogonal polynomial ensemble $\mathbb{P}_{w, N}$ plays an essential role.
\begin{proposition}\label{prop:Markov_property}
    The stochastic process $(X_t)_{t\in \mathbb{R}}$ in Proposition \ref{prop:process_zero_temp} has the Markov property.
\end{proposition}
\begin{proof}
    By \cite[Theorem 5.1]{Olshanski20}, $\mathbb{P}_{w, N}$ is perfect, and $K_{w, N}$ is its canonical kernel (see Appendix \ref{sec:perfectness} for more details). Thus, there exists a representation $(T, L^2(\mathcal{C}(\mathfrak{X}), \mathbb{P}_{w, N}))$ of the GICAR algebra $\mathfrak{I}(\mathfrak{X})$ (see Section \ref{sec:DPP_CAR}) such that $\varphi_{w, N}(A)=\langle T(A)\mathbbm{1}, \mathbbm{1}\rangle$ for all $A\in \mathfrak{I}(\mathfrak{X})$ and $T(f)\mathbbm{1}=f$ for all $f\in C(\mathcal{C}(\mathfrak{X}))\cong \mathfrak{D}(\mathfrak{X})$. In particular, $(T, L^2(\mathcal{C}(\mathfrak{X})), \mathbbm{1})$ is the GNS triple associated with $\varphi_{w, N}$, and $\mathbbm{1}$ is cyclic for $T(\mathfrak{D}(\mathfrak{X}))$. Therefore, by Corollary \ref{cor:suf_cond_Markov}, $(X_t)_{t\in \mathbb{R}}$ is Markov.
\end{proof}

\subsection{Limiting behavior concerning parameters of orthogonal polynomials}\label{sec:ltcpop} 
In Section \ref{sec:limit_OP}, we discussed several limiting behavior of orthogonal polynomial ensembles. The essential idea was that their correlation kernels are spectral projections of certain self-adjoint operators, and the convergence of self-adjoint operators induces the convergence of the associated point processes. By Proposition \ref{prop:limit_by_strong_resolvent_convergence}, the associated stochastic processes converge as well as point processes. In particular, by convergences in Examples \ref{ex:M_dH}, \ref{ex:M_dL}, \ref{ex:K_dH}, \ref{ex:H_dL}, \ref{ex:R_dJ}, we obtain stationary processes with respect to discrete Hermite/Laguerre/Jacobi ensembles.

\medskip
Let $T$ denote the self-adjoint operator given by Equation \eqref{eq:multi_operator} in the discrete Hermite case.
\begin{proposition}[stochastic process on the discrete Hermite ensemble at finite temperature]
    The quasi-free state $\varphi_{T-r, \beta}$ on $\mathfrak{A}(\mathbb{Z}_{\geq 0})$ is stochastically positive for all $\beta>0$. The associated stationary process $(X^{(\beta)}_t)_{t\in [-\beta/2, \beta/2]}$ on $\mathcal{C}(\mathbb{Z}_{\geq 0})$ with respect to the discrete Hermite ensemble at inverse temperature $\beta$ satisfies 
    \[\mathbb{P}[(X^{(\beta)}_{t_1})_{x_1}=1, \dots, (X^{(\beta)}_{t_n})_{x_n}=1]=\det[\mathcal{K}_{\mathrm{DHermite}, (r, \infty)}^{(\beta)}(x_i, t_i; x_j, t_j)]_{i, j=1}^n\]
    for all $x_1, \dots, x_n\in \mathbb{Z}_{\geq 0}$ and $t_1\leq \cdots \leq t_n$ in $[-\beta/2, \beta/2]$. Here, $\mathcal{K}_{\mathrm{DHermite}, (r, \infty)}^{(\beta)}$ is given by 
    \[\mathcal{K}_{\mathrm{DHermite}, (r, \infty)}^{(\beta)}(x, t; y, s):=\begin{dcases} \int_\mathbb{R} \widetilde{h_x}(u)\widetilde{h_y}(u) \frac{e^{-(\beta-s+t)(u-r)}}{1+e^{-\beta(u-r)}}e^{-u^2}du & t\geq s, \\ 
    -\int_\mathbb{R}\widetilde{h_x}(u)\widetilde{h_y}(u)\frac{e^{-(t-s)(u-r)}}{1+e^{-\beta(u-r)}}e^{-u^2}du & t>s,\end{dcases}\]
    where $\widetilde{h_x}:=h_x/\|h_x\|$ for all $x=0, 1, \dots$.
\end{proposition}

As $\beta$ tends to infinity, $\varphi_{T-r, \beta}$ converges to the quasi-free state $\varphi_{\mathrm{DHertmite}, (r, \infty)}$ associated with $K_{\mathrm{DHermite}, (r, \infty)}$. The following is a consequence of Lemma \ref{lem:convergence} and Proposition \ref{prop:dynam_corr_zero_temp}:

\begin{proposition}[stochastic process on the discrete Hermite ensemble at zero temperature]\label{prop:dH_zero}
    The quasi-free state $\varphi_{\mathrm{DHertmite}, (r, \infty)}$ on $\mathfrak{A}(\mathbb{Z}_{\geq 0})$ is stochastically positive. The associated stationary process $(X_t)_{t\in \mathbb{R}}$ on $\mathcal{C}(\mathbb{Z}_{\geq 0})$ with respect to the discrete Hermite ensemble satisfies 
    \[\mathbb{P}[(X_{t_1})_{x_1}=1, \dots, (X_{t_n})_{x_n}=1]=\det[\mathcal{K}_{\mathrm{DHermite}, (r, \infty)}(x_i, t_i; x_j, t_j)]_{i, j=1}^n\]
    for all $x_1, \dots, x_n\in \mathbb{Z}_{\geq 0}$ and $t_1\leq \cdots \leq t_n$ in $\mathbb{R}$. Here, $\mathcal{K}_{\mathrm{DHermite}, (r, \infty)}$ is given by 
    \[\mathcal{K}_{\mathrm{DHermite}, (r, \infty)}(x, t; y, s):=\begin{dcases}
        \int_r^\infty \widetilde{h_x}(u)\widetilde{h_y}(u)e^{(s-t)u}e^{-u^2}du & t\leq s, \\ 
        -\int_{-\infty}^r \widetilde{h_x}(u)\widetilde{h_y}(u)e^{-(t-s)u}e^{-u^2}du & t>s.\end{dcases}\]
\end{proposition}

Let $T$ denote the self-adjoint operator given by Equation \eqref{eq:multi_operator} in the discrete Laguerre case with parameter $c>0$, and $s$ is the linear operator on $\ell^2(\mathbb{Z}_{\geq 0})$ defined by $s\delta_x:=(-1)^x\delta_x$ (see Example \ref{ex:M_dL}). We obtain the same results for the discrete Laguerre ensemble. 
\begin{proposition}[stochastic process on the discrete Laguerre ensemble at finite temperature]
    The quasi-free state $\varphi_{-s(T-r)s, \beta}$ on $\mathfrak{A}(\mathbb{Z}_{\geq 0})$ is stochastically positive for all $\beta>0$. The associated stationary process $(X^{(\beta)}_t)_{t\in [-\beta/2, \beta/2]}$ on $\mathcal{C}(\mathbb{Z}_{\geq 0})$ with respect to the discrete Laguerre ensemble at inverse temperature $\beta$ satisfies
    \[\mathbb{P}[(X^{(\beta)}_{t_1})_{x_1}=1, \dots, (X^{(\beta)}_{t_n})_{x_n}=1]=\det[\mathcal{K}_{\mathrm{DLaguerre}(c), (0, r)}^{(\beta)}(x_i, t_i; x_j, t_j)]_{i, j=1}^n\]
    for all $x_1, \dots, x_n\in \mathbb{Z}_{\geq 0}$ and $t_1\leq \cdots \leq t_n$ in $[-\beta/2, \beta/2]$. Here, $\mathcal{K}_{\mathrm{DLaguerre}(c), (0, r)}^{(\beta)}$ is given by
    \[\mathcal{K}_{\mathrm{DLaguerre}(c), (0, r)}^{(\beta)}(x, t; y, s):=\begin{dcases} \int_0^\infty \widetilde{l^{(c)}_x}(u)\widetilde{l^{(c)}_y}(u) \frac{e^{-(\beta-s+t)(u-r)}}{1+e^{-\beta(u-r)}}u^{c-1}e^{-u}du & t\geq s, \\ 
    -\int_0^\infty \widetilde{l^{(c)}_x}(u) \widetilde{l^{(c)}_y}(u)\frac{e^{-(t-s)(u-r)}}{1+e^{-\beta(u-r)}}u^{c-1}e^{-u}du & t>s,\end{dcases}\]
    where $\widetilde{l^{(c)}_x}:=l^{(c)}_x/\|l^{(c)}_x\|$ for all $x=0, 1,\dots$.
\end{proposition}

Let $\varphi_{\mathrm{DLaguerre}(c), (0, r)}$ denote the quasi-free state associated with $K_{\mathrm{DLaguerre}(c), (0, r)}$. Similarly to the above, the following is a consequence of Lemma \ref{lem:convergence} and Proposition \ref{prop:dynam_corr_zero_temp}:
\begin{proposition}[stochastic process on the discrete Laguerre ensemble at zero temperature]\label{prop:dL_zero}
    The quasi-free state $\varphi_{\mathrm{DLaguerre}(c), (0, r)}$ on $\mathfrak{A}(\mathbb{Z}_{\geq 0})$ is stochastically positive. The associated stationary process $(X_t)_{t\in \mathbb{R}}$ on $\mathcal{C}(\mathbb{Z}_{\geq 0})$ with respect to the discrete Laguerre ensemble satisfies
    \[\mathbb{P}[(X_{t_1})_{x_1}=1, \dots, (X_{t_n})_{x_n}=1]=\det[\mathcal{K}_{\mathrm{DLaguerre}(c), (0, r)}(x_i, t_i; x_j, t_j)]_{i, j=1}^n\]
    for all $x_1, \dots, x_n\in \mathbb{Z}_{\geq 0}$ and $t_1\leq \cdots \leq t_n$ in $\mathbb{R}$. Here, $\mathcal{K}_{\mathrm{DLaguerre}(c), (0, r)}$ is given by
    \[\mathcal{K}_{\mathrm{DLaguerre}(c), (0, r)}(x, t; y, s):=\begin{dcases}
            \int_0^r \widetilde{l^{(c)}_x}(u)\widetilde{l^{(c)}_y}(u)e^{(s-t)u}u^{c-1}e^{-u}du & t\leq s, \\ -\int_r^\infty \widetilde{l^{(c)}_x}(u)\widetilde{l^{(c)}_y}(u)e^{-(t-s)u}u^{c-1}e^{-u}du & t>s.\end{dcases}\]
\end{proposition}

Let $T$ denote the self-adjoint operator given by Equation \ref{eq:multi_operator} in the discrete Jacobi case with parameter $a, b>-1$. We obtain the same results as above for the discrete Jacobi ensemble. 
\begin{proposition}[stocahstic process on the discrete Jabobi ensemble at finite temperature]
    The quasi-free state $\varphi_{T-r, \beta}$ on $\mathfrak{A}(\mathbb{Z}_{\geq 0})$ is stochastically positive for all $\beta>0$. The associated stationary process $(X^{(\beta)}_t)_{t\in [-\beta/2, \beta/2]}$ on $\mathcal{C}(\mathbb{Z}_{\geq 0})$ with respect to the discrete Jacobi ensemble at inverse temperature $\beta$ satisfies 
    \[\mathbb{P}[(X^{(\beta)}_{t_1})_{x_1}=1, \dots, (X^{(\beta)}_{t_n})_{x_n}=1]=\det[\mathcal{K}_{\mathrm{DJacobi}(a, b), (r, 1)}^{(\beta)}(x_i, t_i; x_j, t_j)]_{i, j=1}^n\]
    for all $x_1, \dots, x_n\in \mathbb{Z}_{\geq 0}$ and $t_1\leq \cdots \leq t_n$ in $[-\beta/2, \beta/2]$. Here, $\mathcal{K}_{\mathrm{DJacobi}(a, b), (0, r)}^{(\beta)}$ is given by
    \[\mathcal{K}_{\mathrm{DJacobi}(a, b), (r, 1)}^{(\beta)}(x, t; y, s):=\begin{dcases} \int_{-1}^1 \widetilde{j^{(a, b)}_x}(u) \widetilde{j^{(a, b)}_y}(u) \frac{e^{-(\beta-s+t)(u-r)}}{1+e^{-\beta(u-r)}}(1-u)^a(1+u)^bdu & t\geq s, \\ 
    -\int_{-1}^1 \widetilde{j^{(a, b)}_x}(u) \widetilde{j^{(a, b)}_y}(u)\frac{e^{-(t-s)(u-r)}}{1+e^{-\beta(u-r)}}(1-u)^a(1+u)^b du & t>s,\end{dcases}\]
    where $\widetilde{j^{(a, b)}_x}:=j^{(a, b)}_x/\|j^{(a, b)}_x\|$ for all $x=0, 1, \dots$.
\end{proposition}

Let $\varphi_{\mathrm{DJacobi}(a, b), (r, 1)}$ denote the quasi-free state associated with $K_{\mathrm{DJacobi}(a, b), (r, 1)}$. Similarly to the above, the following is a consequence of Lemma \ref{lem:convergence} and Proposition \ref{prop:dynam_corr_zero_temp}:
\begin{proposition}[stochastic process on the discrete Jabobi ensemble at zero temperature]\label{prop:dJ_zero}
    The quasi-free state $\varphi_{\mathrm{DJacobi}(a, b), (r, 1)}$ on $\mathfrak{A}(\mathbb{Z}_{\geq 0})$ is stochastically positive. The associated stationary process $(X_t)_{t\in \mathbb{R}}$ on $\mathcal{C}(\mathbb{Z}_{\geq 0})$ with respect to the discrete Jacobi ensemble satisfies
    \[\mathbb{P}[(X_{t_1})_{x_1}=1, \dots, (X_{t_n})_{x_n}=1]=\det[\mathcal{K}_{\mathrm{DJacobi}(a, b), (r, 1)}(x_i, t_i; x_j, t_j)]_{i, j=1}^n\]
    for all $x_1, \dots, x_n\in \mathbb{Z}_{\geq 0}$ and $t_1\leq \cdots \leq t_n$ in $\mathbb{R}$. Here, $\mathcal{K}_{\mathrm{DJacobi}(a, b), (r, 1)}$ is given by
    \[\mathcal{K}_{\mathrm{DJacobi}(a, b), (r, 1)}(x, t; y, s):=\begin{dcases}
            \int_0^r \widetilde{j^{(a, b)}_x}(u)\widetilde{j^{(a, b)}_y}(u)e^{(s-t)u}(1-u)^a(1+u)^bdu & t\leq s, 
            \\ -\int_r^\infty \widetilde{j^{(a, b)}_x}(u)\widetilde{j^{(a, b)}_y}(u)e^{-(t-s)u}(1-u)^a(1+u)^bdu & t>s.\end{dcases}\]
\end{proposition}

In the above results, the self-adjoint operator $T$ is equivalent to the multiplication operator on $L^2(\mathbb{R}, w)$ by the coordinate function, where $w$ is the weight function associated with the Hermite/Laguerre/Jacobi cases. It implies that $T$ has purely continuous spectrum on the support of a weight function $w$. Thus, Lemma \ref{lem:cond_stoc_positivity} is not applicable in the these three cases. However, we can show that quasi-free states in these cases are stochastically positive through the approximation by discrete orthogonal polynomial ensembles. It breathes that our framework of stochastically positive KMS systems is quite useful in various situations.

By \cite[Theorem 8.2]{Olshanski20}, discrete Hermite/Laguerre/Jacobi ensembles are perfect, and their canonical kernels are given by $K_{\mathrm{DHermite}, (r, \infty)}$, $s K_{\mathrm{DLaguerre}, (0, r)}s$, and $K_{\mathrm{DJacobi}, (r, 1)}$, respectively. Therefore, by the same argument in Proposition \ref{prop:Markov_property}, we conclude the following:
\begin{proposition}\label{prop:Markov_2}
    Let $(X_t)_{t\in \mathbb{R}}$ be a one of stochastic processes in Propositions \ref{prop:dH_zero}, \ref{prop:dL_zero}, \ref{prop:dJ_zero}. Then, $(X_t)_{t\in \mathbb{P}}$ has the Markov property.
\end{proposition}


\subsection{DPPs related to the hypergeometric difference operator}
In the rest of the paper, we discuss DPPs related to the hypergeometric difference operator. In fact, although our framework is applicable even in this case, but we should slightly modify the previous discussion.

Throughout this section, our state space $\mathfrak{X}$ is assumed to be $\mathbb{Z}':=\mathbb{Z}+1/2$. Let $\xi\in (0, 1)$ and $z, z'\in \mathbb{C}$. We suppose that $(z+k)(z'+k)>0$ holds for every $k\in \mathbb{Z}$. As already mentioned, $(z, z')$ is divided into two cases, called principal and complementary (see Example \ref{ex:AL}). We define the \emph{hypergeometric difference operator} $D:=D_{z, z', \xi}$ on $\mathbb{Z}'$ by 
\[[Df](x):=\frac{1}{1-\xi}\left\{a_xf(x+1)-(x+\xi(z+z'+x))f(x)+a_{x-1}f(x-1)\right\},\]
\[a_x:=\sqrt{\xi(z+x+1/2)(z'+x+1/2)} \quad (x\in \mathbb{Z}').\]
By \cite[Proposition 2.2]{Olshanski08}, $D$ defined on the space of finite supported functions in $\ell^2(\mathbb{Z}')$ is essentially self-adjoint. In what follows, its unique self-adjoint extension is denoted by the same symbol $D$. Moreover, by \cite[Proposition 2.4]{Olshanski08}, the spectrum of $D$ consists of eigenvalues with multiplicity one, and they fill $\mathbb{Z}'$. Thus, $e^{-\beta D}$ is never of trace class for all $\beta$. Namely, we do not possess the first requirement in Section \ref{sec:stoc_from_op}.

Let $(\psi_x)_{x\in \mathbb{Z}'}$ be an orthonormal basis for $\ell^2(\mathbb{Z}')$ such that $D\psi_x=x\psi_x$. We define the \emph{discrete hypergeometric kernel} $K:=K_{z, z', \xi}$ on $\mathbb{Z}'\times \mathbb{Z}'$ by 
\[K(x, y):=\sum_{a\in \mathbb{Z}'_{>0}}\psi_a(x)\psi_a(y) \quad (x, y\in \mathbb{Z}').\]
By definition, the integral operator of $K$ is nothing but the spectral projection onto $\{D>0\}$. In particular, it is positive and contractive. Thus, we obtain the associated quasi-free state $\varphi_K$ on $\mathfrak{A}(\mathbb{Z}')$ and a DPP $\mathbb{P}^\mathrm{dHyp}_{z, z', \xi}$ on $\mathbb{Z}'$ with correlation kernel $K$. We call $\mathbb{P}^\mathrm{dHyp}_{z, z', \xi}$ the \emph{discrete hypergeometric process}. In addition, there exists a DPP $\mathbb{P}^\mathrm{dHyp}_{z, z', \xi; \beta}$ on $\mathbb{Z}'$ with correlation kernel $K_{-D, \beta}:=e^{\beta D}(1+e^{\beta D})^{-1}$ for all $\beta>0$, called the discrete hypergeometric process \emph{at inverse temperature $\beta$}. By definition, $\mathbb{P}^\mathrm{dHyp}_{z, z', \xi; \beta}$ weakly converges to $\mathbb{P}^\mathrm{dHyp}_{z, z', \xi}$ as $\beta \to \infty$.

We remark that $K$ and $U_t:=e^{\mathrm{i}t D}$ commute for all $t\in \mathbb{R}$. Thus, $\varphi_K$ is invariant under the $\mathbb{R}$-flow $\alpha^D_t\colon \mathbb{R}\curvearrowright \mathfrak{A}(\mathbb{Z}')$ given by $\alpha^D_t(a(h))=a(U_t h)$ for any $h\in \ell^2(\mathbb{Z}')$. 

Let us consider the GICAR algebra $\mathfrak{I}(\mathbb{Z}')$, and $\varphi_K$ denotes the restriction of $\varphi_K$ to $\mathfrak{I}(\mathbb{Z}')$. Moreover, we denote by $(\pi_K, \mathcal{H}_K, \Omega_K)$ the associated GNS-triple. We remark that the $\mathbb{R}$-flow $\alpha^D$ preserves $\mathfrak{I}(\mathbb{Z}')$. Thus, the invariance of $\varphi_K$ under $\alpha^D$ induces the unique self-adjoint operator $d\Gamma(D)$ on $\mathcal{H}_K$ such that $e^{\mathrm{i}t d\Gamma(D)}\pi_K(A)\Omega_K=\pi_K(\alpha^D_t(A))\Omega_K$ for all $A\in \mathfrak{I}(\mathbb{Z}')$ and $t\in \mathbb{R}$. As shown in \cite[Proposition 4.3]{S24}, the spectrum of $d\Gamma(D)$ coincides with $\{-n\}_{n\in \mathbb{Z}_{\geq 0}}$, and every $-n$ is eigenvalue whose multiplicity is equal to the number of partition of $n$. In particular, $e^{\beta d\Gamma(D)}$ is of trace class for all $\beta>0$. 

We define a new state $\varphi_{-D, \beta}$ on $\mathfrak{I}(\mathbb{Z}')$ by 
\[\varphi_{-D, \beta}(A):=\frac{\mathrm{Tr}(e^{\beta d\Gamma(D)} \pi_K(A))}{\mathrm{Tr}(e^{\beta d\Gamma(D)})} \quad (A\in \mathfrak{I}(\mathbb{Z}')).\]
By the same argument in \cite[Proposition 5.2.23]{BR2}, we can show that $\varphi_{-D, \beta}$ is quasi-free state on $\mathfrak{I}(\mathbb{Z}')$ associated with $K_{-D, \beta}$. Moreover, $\varphi_{-D, \beta}$ is $\alpha^D$-KMS state at $\beta$.

\begin{proposition}
    $\varphi_{-D, \beta}$ is stochastically positive with respect to $\mathfrak{D}(\mathbb{Z}')$. The associated stochastic process $X^{(\beta)}=(X^{(\beta)}_t)_{t\in [-\beta/2, \beta/2]}$ on $\mathcal{C}(\mathbb{Z}')$ is stationary for the discrete hypergeometric process $\mathbb{P}^\mathrm{dHyp}_{z, z', \xi; \beta}$ at inverse temperature $\beta$ and satisfies 
    \[\mathbb{P}[(X^{(\beta)}_{t_1})_{x_1}=1, \dots, (X^{(\beta)}_{t_n})_{x_n}=1]=\det[\mathcal{K}^\mathrm{dHyp}_{z, z', \xi; \beta}(x_i, t_i; x_j, t_j)]_{i, j=1}^n\]
    for all $x_1, \dots, x_n\in \mathbb{Z}'$ and $-\beta/2\leq t_1\leq \cdots \leq t_n\leq \beta/2$, where 
    \[\mathcal{K}^\mathrm{dHyp}_{z, z', \xi; \beta}(x, t; y, s):=\begin{dcases}
        \sum_{a\in \mathbb{Z}'}\frac{e^{(\beta-s+t)a}}{1+e^{\beta a}}\psi_a(x)\psi_a(y) & t\leq s, \\ -\sum_{a\in \mathbb{Z}'}\frac{e^{(t-s)a}}{1+e^{\beta a}}\psi_a(x)\psi_y(y) & t>s.
    \end{dcases}\]
\end{proposition}
\begin{proof}
    As shown in \cite[Section 5.3]{S24}, $d\Gamma(D)$ on $\mathcal{H}_K$ is unitarily equivalent to a Markov generator on $L^2(\mathbb{Y}, M_{z, z', \xi})$, where $\mathbb{Y}$ is the set of Young diagrams and $M_{z, z', \xi}$ is the so-called $z$-measure. Thus, by \cite[Theorem 18.4]{KL81}, $\varphi_{-D, \beta}$ is stochastically positive. The determinantal space-time formula for the stochastic process $X^{(\beta)}$ is follows from Theorem \ref{thm:dynamical_correlation} and Equation \eqref{eq:zero_temp_limit}.
\end{proof}

The following is a consequence of Proposition \ref{prop:dynam_corr_zero_temp}:
\begin{proposition}
    $\varphi_K$ is stochastically positive with respect to $\mathfrak{D}(\mathbb{Z}')$. The associated stochastic process $X=(X_t)_{t\in \mathbb{R}}$ is stationary for the discrete hypergeometric process $\mathbb{P}^\mathrm{dHyp}_{z, z', \xi}$ and satisfies 
    \[\mathbb{P}[(X^{(\beta)}_{t_1})_{x_1}=1, \dots, (X^{(\beta)}_{t_n})_{x_n}=1]=\det[\mathcal{K}^\mathrm{dHyp}_{z, z', \xi}(x_i, t_i; x_j, t_j)]_{i, j=1}^n\]
    for all $x_1, \dots, x_n\in \mathbb{Z}'$ and $t_1\leq \cdots \leq t_n$ in $\mathbb{R}$, where 
    \[\mathcal{K}^\mathrm{dHyp}_{z, z', \xi}(x, t; y, s):=\begin{dcases}
        \sum_{a\in \mathbb{Z}'_{>0}}e^{-(s-t)a}\psi_a(x)\psi_a(y) & t\leq s, \\ -\sum_{a\in \mathbb{Z}'_{<0}}e^{(t-s)a}\psi_a(x)\psi_y(y) & t>s.
    \end{dcases}\]
    Moreover, $X^{(\beta)}$ converges to $X$ in the sense of Proposition \ref{prop:convergence}.
\end{proposition}

We remark that the resulting stochastic process $X$ and the space-time correlation kernel $\mathcal{K}^\mathrm{dHyp}_{z, z', \xi}$ had already appeared in \cite[Theorem A]{BO06}. 

Furthermore, as shown in \cite{Olshanski08}, the discrete hypergeometric difference operator converges to the discrete Bessel difference operator and the discrete sine difference operator in the strong resolvent sense. Therefore, by Proposition \ref{prop:limit_by_strong_resolvent_convergence}, we obtain stationary processes with respect to the discrete Bessel process and the discrete sine process. See also \cite{Esaki,Katori15} for another realization of stationary process on the discrete sine process via relaxation phenomena.

To show the stochastic positivity of $\varphi_{-D, \beta}$, we used the fact that the self-adjoint operator $d\Gamma(D)$ acting on the GNS-representation space associated with $\varphi_K$ is unitarily equivalent to a Markov generator. Thus, in this sense, our construction seems to be a tautology. However, throughout our construction, we can easily derive the above determinantal formula of space-time correlations of stochastic processes.

\section{Conclusions}

In this paper, we further developed the previous work \cite{S24} on the dynamical relationship between DPPs and CAR algebras. Especially, motivated by the work \cite{KL81} by Klein and Landau, we discussed stochastically positive, quasi-free, KMS states on CAR algebras and contracted stochastic processes on point configuration spaces that are stationary with respect to the associated DPPs. Because of an algebraic nature of resulting stochastic processes, we obtained the determinantal formula of space-time correlations, and we developed an algebraic method of analyzing limiting behaviors of stochastic processes. Moreover, we demonstrated that our framework is applicable for various DPPs related to (discrete) orthogonal polynomials. 

On the other hand, there remain several topics that we have not discussed in the paper but that are certainly interesting. Firstly, we did not discuss $q$-hypergeometric orthogonal polynomials, although they appear in various fields of mathematics, including the intersecting area of point processes and representation theory (see \cite{GO}). When $q$-orthogonal polynomials are eigenfunctions of a self-adjoint operator, we can apply our framework and obtain stationary processes with respect to $q$-orthogonal polynomial ensembles and their finite-temperature variations.

Secondly, Pfaffian point processes are quite interesting research topic even in the discrete setting (see \cite{BS,P11}), but we have not yet succeeded in extending the static relationship between Pfaffian point processes and self-dual CAR algebras to the dynamical level. It remains for future work.

Thirdly, one of the most interesting and natural research directions is extending our work to the continuum setting, i.e., developing an algebraic approach to studying stationary processes with respect to point processes on continuum spaces. It is known that such stochastic processes mainly appear in asymptotic representation theory and random matrix theory. In asymptotic representation theory, we are interested in the point processes define on dual spaces of inducitive limits of compact groups, which are continuum spaces. In random matrix theory, random eigenvalues of random matrices is actively studied, and they usually form point processes on $\mathbb{R}$ or $\mathbb{C}$. An operator algebraic approach to constructing stochastic processes related to random matrices (e.g., Dyson's brownian motion) seems very interesting research direction, and it will be open new interaction between random matrix theory and operator algebra theory.

\appendix
\section{Pfaffian point processes and self-dual CAR algebras}
Our interest in the paper is in determinantal point processes rather than on Pfaffian point processes. However, we employ a computational technique involving Pfaffians to derive space-time correlations of stochastic processes. To ensure the paper is self-contained, we summarize the relationship between Pfaffian point processes and self-dual CAR algebras. See also \cite{Koshida}.

\subsection{Pfaffian point processes}
To define a Pfaffian point process, let us recall some notations. For a $2n\times 2n$ skew-symmetric matrix $A=[a_{ij}]_{i, j=1}^{2n}$, its \emph{Pfaffian} $\mathrm{Pf}[A]$ is defined by
\[\mathrm{Pf}[A]:=(-1)^{n(n-1)/2}\sum_{\sigma}\mathrm{sgn}(\sigma)\prod_{j=1}^n a_{\sigma(j)\sigma(j+n)},\]
where $\sigma$ runs over all permutations of $\{1, \dots, 2n\}$ such that $\sigma(1)<\cdots <\sigma(n)$ and $\sigma(j)<\sigma(j+n)$ for all $j=1, \dots, n$. We remark that $\mathrm{Pf}[A]$ is determined only by the upper triangular entries of $A$. Thus, we sometimes write $\mathrm{Pf}(a_{ij})_{1\leq i<j\leq 2n}$. Let $A_{ij}$ ($i, j=1, \dots, n$) denote the $2\times 2$ matrix in $(i, j)$-block of $A$, i.e., $A=(A_{ij})_{i, j=1}^n$ and
\[A_{ij}=\begin{pmatrix} a_{2i-1, 2j-1} & a_{2i-1, 2j} \\ a_{2i, 2j-1} & a_{2i, 2j}\end{pmatrix}.\]
Since $A^T=(A_{ji}^T)_{i, j=1}^n$, we have $A_{ij}^T=-A_{ji}$.

Let $\mathfrak{X}$ be a countable set with the discrete topology. A point process $\mathbb{P}$ on $\mathfrak{X}$ is called a \emph{Pfaffian point process} if there exists a $2\times 2$  matrix-valued function $K\colon \mathfrak{X}\times \mathfrak{X}\to \mathbb{M}_2(\mathbb{C})$ such that $K(x, y)^T=-K(y, x)$ for all $x, y\in \mathfrak{X}$ and for all $x_1, \dots, x_n\in \mathfrak{X}$ we have
\[\rho^{(n)}_\mathbb{P}(x_1, \dots, x_n)=\mathrm{Pf}[K(x_i, x_j)]_{i, j=1}^n.\]
We call $K$ a \emph{correlation kernel} of $\mathbb{P}$. Let $K_{ij}\colon \mathfrak{X}\times \mathfrak{X}\to \mathbb{C}$ ($i, j=1, 2$) be defined by
\[K(x, y)=\begin{pmatrix} K_{11}(x, y) & K_{12}(x, y) \\ K_{21}(x, y) & K_{22}(x, y) \end{pmatrix} \quad (x, y\in \mathfrak{X}).\]
If $K_{11}(x, y)=K_{22}(x, y)=0$ for all $x, y\in \mathfrak{X}$, we have $\mathrm{Pf}[K(x_i, x_j)]_{i, j=1}^n = \det[K_{12}(x_i, x_j)]_{i, j=1}^n$. Thus, a Pfaffian point process is a generalization of a determinantal point process.

\subsection{Self-dual CAR algebras and gauge-invariant quasi-free states}
Let $\mathcal{K}$ be a complex Hilbert space and $\Gamma$ its complex conjugate, i.e., $\Gamma$ is a conjugate-linear, involutive ($\Gamma^2=\mathrm{id}$), unitary map on $\mathcal{K}$. The \emph{self-dual CAR algebra} $\mathfrak{A}_\mathrm{SDC}(\mathcal{K}, \Gamma)$ is a universal unital $C^*$-algebra generated by $b(h)$ ($h\in \mathcal{K}$) such that the mapping $h\in \mathcal{K}\mapsto b(h)\in \mathfrak{A}_\mathrm{SDC}(\mathcal{K}, \Gamma)$ is linear and the \emph{self-dual canonical anti-commutation relation} (self-dual CAR) holds:
\[b(h)b(k)^*+b(k)^*b(h)=\langle h, k\rangle 1, \quad b(h)^*=b(\Gamma h) \quad (h, k\in \mathcal{K}).\]
Here, the inner product is complex linear in the left variable.

A state $\varphi$ on $\mathfrak{A}_\mathrm{SDC}(\mathcal{K}, \Gamma)$ is said to be (\emph{gauge-invariant}) \emph{quasi-free} if we have
\[\varphi(b(h_1)\cdots b(h_{2n+1}))=0, \quad \varphi(b(h_1)\cdots b(h_{2n}))=\mathrm{Pf}[A_\varphi(h_1, \dots, h_{2n})]\]
for all $h_1, \dots, h_{2n+1}\in \mathcal{K}$, where $A_\varphi(h_1, \dots, h_{2n})$ is the unique skew-symmetric $2n\times 2n$ matrix whose upper triangular entries are given as
\[A_\varphi(h_1, \dots, h_{2n})_{ij}:=\varphi(b(h_i)b(h_j)) \quad (1\leq i\leq j \leq 2n).\]
By the Riesz representation theorem, there exists a unique bounded linear operator $S$ on $\mathcal{K}$, called the \emph{covariance operator}, such that
\begin{equation}\label{eq:covariance_operator1}
    \varphi(b(h)^*b(k))=\langle Sk, h\rangle \quad (h, k\in \mathcal{K}),
\end{equation}
\begin{equation}\label{eq:covariance_operator2}
    0\leq S\leq 1, \quad S+\Gamma S\Gamma =1.
\end{equation}
Conversely, if a bounded linear operator $S$ on $\mathcal{K}$ satisfies Equation \eqref{eq:covariance_operator2}, then there exists a unique quasi-free state on $\mathfrak{A}_\mathrm{SDC}(\mathcal{K}, \Gamma)$, denoted by $\varphi_S$, satisfying Equation \eqref{eq:covariance_operator1}. Thus, there exists a bijection between quasi-free states on $\mathfrak{A}_\mathrm{SDC}(\mathcal{K}, \Gamma)$ and covariance operators on $(\mathcal{K}, \Gamma)$. A covariance operator is called a \emph{basic projection} if it is an orthogonal projection.

Let $E$ be a basic projection on $(\mathcal{K}, \Gamma)$ and $\mathcal{H}:=E\mathcal{K}$. Then, $\hat a(h):=b(h)^*\in \mathfrak{A}_\mathrm{SDC}(\mathcal{K}, \Gamma)$ ($h\in \mathcal{H}$) satisfy the CAR over $\mathcal{H}$. Thus, there exists a unital $*$-algebra homomorphism from $\mathfrak{A}(\mathcal{H})$ to $\mathfrak{A}_\mathrm{SDC}(\mathcal{K}, \Gamma)$ assigning $a(h)$ to $\hat a(h)$ for any $h\in \mathcal{H}$. Conversely, we can easily check that $\hat b(h):=a^*(Eh)+a(E\Gamma h) \in \mathfrak{A}(\mathcal{H})$ ($h\in \mathcal{K}$) satisfy the self-dual CAR over $(\mathcal{K}, \Gamma)$. Thus, there exists a unital $*$-algebra homomorphism from $\mathfrak{A}_\mathrm{SDC}(\mathcal{K}, \Gamma)$ to $\mathfrak{A}(\mathcal{H})$ assigning $b(h)$ to $\hat b(h)$ for any $h\in \mathcal{K}$. Therefore, in conclusion, two $C^*$-algebras $\mathfrak{A}_\mathrm{SDC}(\mathcal{K}, \Gamma)$ and $\mathfrak{A}(\mathcal{H})$ are $*$-isomorphic.

In what follows, we freely identify $\mathfrak{A}_\mathrm{SDC}(\mathcal{K}, \Gamma)$ and $\mathfrak{A}(\mathcal{H})$ while a basic projection $E$ is fixed. In particular, we use the same symbols $a(h)$ (resp. $b(h)$) to denote $\hat a(h)$ (resp. $\hat b(h)$).

\medskip
The following is helpful (see Section \ref{sec:density_op_fermion_Fock_sp}) and is the main result in this Appendix.
\begin{lemma}\label{lem:qf_st_calc}
    Let $N\geq 1$ and $X_i:=a^*(h^{(i)}_1)\cdots a^*(h^{(i)}_{n_i})a(k^{(i)}_1)\cdots a(k^{(i)}_{n_i})$ for each $i=1, \dots, N$, where $n_i\geq 1$ and $h^{(i)}_l, k^{(i)}_l\in \mathcal{H}$ for each $l=1, \dots, n_i$ and $i=1, \dots, N$. If $SE=ES$, then
    \[\varphi_S(X_1\cdots X_N)=\det[B_{ij}]_{i, j=1}^N\]
    holds, where $B_{ij}$ is the $n_i\times n_j$ matrix given as
    \[B_{ij}:=\begin{dcases} \begin{bmatrix} \varphi_S(a^*(h^{(i)}_l)a(k^{(j)}_m)) \end{bmatrix}_{\substack{l=1, \dots, n_i, \\ m=1, \dots, n_j}} & i\leq j, \\ - \begin{bmatrix} \varphi_S(a(k^{(j)}_m)a^*(h^{(i)}_l)) \end{bmatrix}_{\substack{l=1, \dots, n_i, \\ j=1, \dots, n_j}} & i>j. \end{dcases}\]
\end{lemma}
\begin{proof}
    We have $\varphi_S(X_1\cdots X_N)=\mathrm{Pf}[A_{ij}]_{i, j=1}^N$, where $A_{ij}$ is $2n_i\times 2n_j$ matrix given as
    \[A_{ij}= \begin{bmatrix} [\varphi_S(a^*(h^{(i)}_l)a^*(h^{(j)}_m))]_{l, m} & [\varphi_S(a^*(h^{(i)}_l)a(k^{(j)}_m))]_{l, m} \\ [\varphi_S(a(k^{(i)}_l)a^*(h^{(k)}_m))]_{l, m} & [\varphi_S(a(k^{(i)}_l)a(k^{(j)}_m))]_{l, m}\end{bmatrix}\]
    and $A_{ji}=-A_{ij}^T$ for each $i<j$, and
    \[A_{ii} = \begin{bmatrix} [\varphi_S(a^*(h^{(i)}_l)a^*(h^{(i)}_m))]_{l, m} & [\varphi_S(a^*(h^{(i)}_l)a(k^{(i)}_m))]_{l, m} \\ -[\varphi_S(a^*(h^{(i)}_m)a(k^{(i)}_l))]_{l, m} & [\varphi_S(a(k^{(i)}_l)a(k^{(i)}_m))]_{l, m} \end{bmatrix}\]
    for each $i=1, \dots, N$. Let $(A_{ij})_{rs}$ denote $(r, s)$-block of $A_{ij}$ for each $r, s = 1, 2$. Since $SE=ES$, for any $h, k\in \mathcal{H}$ we have
    \[\varphi_S(a^*(h)a^*(k))=\langle k, S\Gamma h\rangle = 0,\quad \varphi_S(a(h)a(k))=\langle \Gamma k, Sh\rangle =0.\]
    Thus, $(A_{ij})_{1, 1}$ and $(A_{ij})_{2, 2}$ are null matrices. Hence, we have $\mathrm{Pf}[A_{ij}]_{i, j=1}^N=\det[(A_{ij})_{1, 2}]_{i, j=1}^N$. Moreover, $(A_{ij})_{1, 2} = B_{ij}$ holds for $i\leq j$, and $(A_{ij})_{12} = -(A_{ji})_{21}^T = B_{ij}$ holds for $i>j$. Therefore, we have $\varphi_S(X_1\cdots X_N) = \det[B_{ij}]_{i, j=1}^N$.
\end{proof}

\subsection{Relationship to Pfaffian point processes}
Let $\mathfrak{X}$ be a countable set. We assume that $\mathcal{K}:=\ell^2(\mathfrak{X})\oplus \ell^2(\mathfrak{X})$ and its complex conjugation $\Gamma$ is given by $\Gamma (f\oplus g):=J g\oplus Jf$ for any $f\oplus g\in \mathcal{K}$, where $Jf:=\overline{f}$. We simply denote by $\mathfrak{A}_\mathrm{SDC}(\mathfrak{X})$ the self-dual CAR algebra $\mathfrak{A}_\mathrm{SDC}(\mathcal{K}, \Gamma)$.
We also fix a basic projection $E$ on $(\mathcal{K}, \Gamma)$ given by $E(f\oplus g):=f\oplus 0$ for any $f\oplus g\in \mathcal{K}$. By the previous discussion, we have $\mathfrak{A}_\mathrm{SDC}(\mathfrak{X})\cong \mathfrak{A}(\mathfrak{X})$, and we freely identify these $C^*$-algebras.

Let us recall that $\mathfrak{D}(\mathfrak{X})$ is the $C^*$-subalgebra generated by $a^*_xa_x(=b(\delta_x\oplus 0)b(0\oplus \delta_x))$ ($x\in \mathfrak{X}$), and it is isomorphic to $C(\mathcal{C}(\mathfrak{X}))$. Thus, for every state $\varphi$ on $\mathfrak{A}_\mathrm{SDC}(\mathfrak{X})$, we obtain the point process on $\mathfrak{X}$, denoted by $\mathbb{P}_\varphi$, whose correlation functions are given by Equation \eqref{eq:corr_state}.

\medskip
The following gives the static relationship between Pfaffian point processes and quasi-free states on self-dual CAR algebras (see also \cite[Proposition 1.3]{Koshida}):
\begin{proposition}\label{prop:ex_pfpp}
    Let $\varphi$ be a quasi-free state on $\mathfrak{A}_\mathrm{SDC}(\mathfrak{X})$. The point process $\mathbb{P}_\varphi$ is a Pfaffian point process with correlation kernel $K_\varphi$ given by
    \[K_\varphi(x, y):=\begin{pmatrix} \varphi(a^*_xa^*_y) & \varphi(a^*_xa_y)\\ \varphi(a_xa^*_y)-\delta_{x, y} & \varphi(a_xa_y)\end{pmatrix} \quad (x, y\in \mathfrak{X}).\]
    In particular, if the covariance operator $S$ of $\varphi$ is given by 
    \[S=\begin{pmatrix} S_{11} & S_{12} \\ S_{21} & S_{22} \end{pmatrix},\]
    then $K_\varphi$ is given by
    \[
        K_\varphi(x, y)=\begin{pmatrix}\langle e_y, S_{12}e_x\rangle & \langle e_y, S_{22}e_x\rangle \\ \langle e_y, (S_{11}-1)e_x\rangle & \langle e_y, S_{21}e_x\rangle\end{pmatrix} \quad (x, y\in \mathfrak{X}).
    \]
\end{proposition}
\begin{proof}
    The later statement immediately follows from the first one. By Equation \eqref{eq:corr_state}, it suffices to show that $\varphi(a^*_{x_n}\cdots a^*_{x_1}a_{x_1}\cdots a_{x_n})=\mathrm{Pf}[K_\varphi(x_i, x_j)]_{i, j=1}^n$ for any $x_1, \dots, x_n\in \mathfrak{X}$ and $n\geq1$. If $x_i=x_j$ for some $i, j=1,\dots, n$ with $i\neq j$, then both sides are 0. If $x_1, \dots, x_n$ are distinct, then we have
    \[\varphi(a^*_{x_n}\cdots a^*_{x_1}a_{x_1}\cdots a_{x_n})=\varphi(a^*_{x_1}a_{x_1}\cdots a^*_{x_n}a_{x_n})=\mathrm{Pf}[A_\varphi(e_{x_1}, \Gamma e_{x_1}, \dots, e_{x_n}, \Gamma e_{x_n}))].\]
    Let us recall that for each $i<j$ the $(i, j)$-block of $A_\varphi(e_{x_1}, \Gamma e_{x_1}, \dots, e_{x_n}, \Gamma e_{x_n})$ is given as
    \[\begin{pmatrix}\varphi(a^*_{x_i}a^*_{x_j})& \varphi(a^*_{x_i}a_{x_j}) \\ \varphi(a_{x_i}a^*_{x_j}) & \varphi(a_{x_i}a_{x_j})\end{pmatrix}.\]
    Moreover, the $(1, 2)$-entry of the $(i, i)$-block is $\varphi(a^*_{x_i}a_{x_i})$ for each $i=1, \dots, n$. Thus, the $(2,1)$-entry is $-\varphi(a^*_{x_i}a_{x_i})=\varphi(a_{x_i}a^*_{x_i})-1$, and hence we have
    \[A_\varphi(e_{x_1}, \Gamma e_{x_1}, \dots, e_{x_n}, \Gamma e_{x_n}) = [K_\varphi(x_i, x_j)]_{i, j=1}^n.\]
    Therefore, $\varphi(a^*_{x_n}\cdots a^*_{x_1}a_{x_1}\cdots a_{x_n})=\mathrm{Pf}[K_\varphi(x_i, x_j)]_{i, j=1}^n$ holds true.
\end{proof}

\begin{remark}\label{rem:kernel}
    If $S_{12}=S_{21}=0$, we have
    \begin{align*}
        \varphi(a^*(h_n)\cdots a^*(h_1)a(k_1)\cdots a(k_n))
        & =\mathrm{Pf}[A_\varphi(h_n, \dots, h_1, \Gamma k_1, \dots, \Gamma k_n)] \\
        & =\det[\varphi(a^*(h_i)a(k_j))]_{i, j=1}^n.
    \end{align*}
    Thus, $\varphi$ is a quasi-free state on $\mathfrak{A}(\mathfrak{X})$, and $\mathbb{P}_\varphi$ is determinantal with correlation kernel given as $K_\varphi(x, y)_{12}=\langle S_{22}e_x, e_y\rangle$ for any $x, y\in \mathfrak{X}$.
\end{remark}

Let us return to a arbitrary state $\varphi$ on $\mathfrak{A}_\mathrm{SDC}(\mathfrak{X})$, and $(\pi_\varphi, \mathcal{H}_\varphi, \Omega_\varphi)$ denotes the associated GNS-triple. We denote by $\mathcal{H}_{\varphi|_{\mathfrak{D}(\mathfrak{X})}}$ the closure of $\pi_\varphi(\mathfrak{D}(\mathfrak{X}))\Omega_\varphi$.

Under the isomorphism $\mathfrak{D}(\mathfrak{X})\cong C(\mathcal{C}(\mathfrak{X}))$, we denoted by the same symbol $f$ the element in $\mathfrak{D}(\mathfrak{X})$ associated with $f \in C(\mathcal{C}(\mathfrak{X}))$.

\begin{lemma}\label{lem:L2}
    There exists a unitary map from $\mathcal{H}_{\varphi|_{\mathfrak{D}(\mathfrak{X})}}$ to $L^2(\mathcal{C}(\mathfrak{X}), \mathbb{P}_\varphi)$ assigning $\pi_\varphi(f)\Omega_\varphi$ to $f$ for every $f\in C(\mathcal{C}(\mathfrak{X}))$.
\end{lemma}
\begin{proof}
    By the definition of $\mathbb{P}_\varphi$, the mapping $\pi_\varphi(f)\Omega_\varphi\in \pi_\varphi(\mathfrak{D}(\mathfrak{X}))\Omega_\varphi \mapsto f\in L^2(\mathcal{C}(\mathfrak{X}), \mathbb{P}_\varphi)$ is well defined and isometric. Thus, it can be uniquely extended to an isometric map on $\mathcal{H}_{\varphi|_{\mathfrak{D}(\mathfrak{X})}}$. Moreover, since $\mathbb{P}_\varphi$ is regular, $C(\mathcal{C}(\mathfrak{X}))$ is dense in $L^2(\mathcal{C}(\mathfrak{X}), \mathbb{P}_\varphi)$ (see \cite[Theorem 3.14]{Rudin}), and hence the range of this map coincides with $L^2(\mathcal{C}(\mathfrak{X}), \mathbb{P}_\varphi)$.
\end{proof}

\section{Perfectness of point processes}\label{sec:perfectness}
\subsection{Perfectness of point processes}
In this section, we will discuss the property of discrete point processes, called \emph{perfectness}, due to Olshanski \cite{Olshanski20} and Koshida \cite{Koshida}.

Let $\mathfrak{X}$ be a countable set as in previous sections. We define the gauge-invariant self-dual CAR algebra $\mathfrak{I}_\mathrm{SDC}(\mathfrak{X})$ as follows. First, as similar to the case of $\mathfrak{A}(\mathfrak{X})$, we define a torus group action $\gamma \colon \mathbb{T}\curvearrowright \mathfrak{A}_\mathrm{SDC}(\mathfrak{X})$, called the \emph{gauge action}, is defined by $\gamma_z(b(f\oplus g)):=b(zf\oplus \bar z g)$ for any $f\oplus g\in \ell^2(\mathfrak{X})\oplus \ell^2(\mathfrak{X})$ and $z\in \mathbb{T}$. Then, $\mathfrak{I}_\mathrm{SDC}(\mathfrak{X})$ is defined by 
\[\mathfrak{I}_\mathrm{SDC}(\mathfrak{X}):=\{A\in \mathfrak{A}_\mathrm{SDC}(\mathfrak{X})\mid \gamma_z(A)=A \text{ for all }z\in \mathbb{T}\}.\]
We remark that $\mathfrak{I}_\mathrm{SDC}(\mathfrak{X})$ is isomorphic to $\mathfrak{I}(\mathfrak{X})$ by the isomorphism $\mathfrak{A}_\mathrm{SDC}(\mathfrak{X})\cong \mathfrak{A}(\mathfrak{X})$. In the study of perfectness, $\mathfrak{I}_\mathrm{SDC}(\mathfrak{X})$ plays a fundamental role.

Let $\mathfrak{S}(\mathfrak{X})$ denote the group of finite permutations on $\mathfrak{X}$, i.e., every $\sigma\in \mathfrak{S}(\mathfrak{X})$ is a bijection on $\mathfrak{X}$ such that $\sigma(x)=x$ for any but finitely many $x\in \mathfrak{X}$. We remark that $\mathfrak{S}(\mathfrak{X})$ naturally acts on $\mathcal{C}(\mathfrak{X})$ and $C(\mathcal{C}(\mathfrak{X}))$ by 
\[\sigma\cdot \omega:=(\omega_{\sigma(x)})_{x\in \mathfrak{X}}, \quad [\sigma\cdot f](\sigma^{-1}\cdot \omega)\] 
for any $\omega\in \mathcal{C}(\mathfrak{X})$, $f\in C(\mathcal{C}(\mathfrak{X}))$, and $\sigma\in \mathfrak{S}(\mathfrak{X})$. Thus, we obtain the crossed product $C^*$-algebra $C(\mathcal{C}(\mathfrak{X}))\rtimes \mathfrak{S}(\mathfrak{X})$ (see, e.g., \cite[Section 4.1]{BO}). We suppose that $\mathbb{P}$ is $\mathfrak{S}(\mathfrak{X})$-quasi-invariant. Thus, the following Radon--Nikodym derivative is well defined:
\[\phi(\omega, \sigma):=\frac{d\mathbb{P}\circ \sigma^{-1}}{d\mathbb{P}}(\omega) \quad (\omega \in \mathcal{C}(\mathfrak{X}), \sigma \in \mathfrak{S}(\mathfrak{X})).\]
We obtain a representation $\mathcal{T}_\mathbb{P}\colon C(\mathcal{C}(\mathfrak{X}))\rtimes \mathfrak{S}(\mathfrak{X})\curvearrowright L^2(\mathcal{C}(\mathfrak{X}), \mathbb{P})$ such that
\[[\mathcal{T}_\mathbb{P}(f)h](\omega):=f(\omega)h(\omega), \quad [\mathcal{T}_\mathbb{P}(\sigma)h](\omega):=h(\sigma^{-1}(\omega))\phi(\omega, \sigma)^{1/2} \quad (h\in L^2(\mathcal{C}(\mathfrak{X}), \mathbb{P}))\]
for any $f\in C(\mathcal{C}(\mathfrak{X}))$ and $\sigma \in \mathfrak{S}(\mathfrak{X})$. 

We suppose that $\mathfrak{X}$ is equipped with a linear order, that is, $\mathfrak{X}$ is isomorphic to $\mathbb{Z}_{\geq0}$ or $\mathbb{Z}$ as an ordered set. Let us recall that $C(\mathcal{C}(\mathfrak{X}))$($\cong \mathfrak{D}(\mathfrak{X})$) is naturally embedded into $C(\mathcal{C}(\mathfrak{X}))\rtimes \mathfrak{S}(\mathfrak{X})$ and $\mathfrak{I}_\mathrm{SDC}(\mathfrak{X})$. By \cite{Olshanski20,Koshida}, there exists a surjective $*$-homomorphism
\[\pi\colon C(\mathcal{C}(\mathfrak{X}))\rtimes \mathfrak{S}(\mathfrak{X})\to \mathfrak{I}_\mathrm{SDC}(\mathfrak{X})\]
such that $\pi|_{C(\mathcal{C}(\mathfrak{X}))}=\mathrm{id}_{C(\mathcal{C}(\mathfrak{X}))}$. Moreover, $\ker(\pi)\subset \ker(\mathcal{T}_\mathbb{P})$ holds. Therefore, the representation $(\mathcal{T}_\mathbb{P}, L^2(\mathcal{C}(\mathfrak{X}), \mathbb{P}))$ of $C(\mathcal{C}(\mathfrak{X}))\rtimes \mathfrak{S}(\mathfrak{X})$ factors through $\mathfrak{I}_\mathrm{SDC}(\mathfrak{X})$. Namely, we obtain the representation $T_\mathbb{P}\colon \mathfrak{I}_\mathrm{SDC}(\mathfrak{X})\curvearrowright L^2(\mathcal{C}(\mathfrak{X}), \mathbb{P})$ such that $\mathcal{T}_\mathbb{P}=T_\mathbb{P}\circ \pi$. Using this representation, we obtain a state $\varphi_\mathbb{P}$ on $\mathfrak{I}_\mathrm{SDC}(\mathfrak{X})$ by $\varphi_\mathbb{P}(A):=\langle T_\mathbb{P}(A)\mathbbm{1}, \mathbbm{1}\rangle$ for any $A\in \mathfrak{I}_\mathrm{SDC}(\mathfrak{X})$, where $\mathbbm{1}$ is the constant function equal to 1. By Lemma \ref{lem:L2}, $\mathbbm{1}$ is cyclic for $T_\mathbb{P}(C(\mathcal{C}(\mathfrak{X})))$, (that is, for $T_\mathbb{P}(\mathfrak{I}_\mathrm{SDC}(\mathfrak{X}))$). Thus, $(T_\mathbb{P}, L^2(\mathcal{C}(\mathfrak{X}), \mathbb{P}), \mathbbm{1})$ is the unique GNS-triple associated with $\varphi_\mathbb{P}$.

The definition of the perfectness is as follows: A $\mathfrak{S}(\mathfrak{X})$-quasi-invariant point process $\mathbb{P}$ on $\mathfrak{X}$ is said to be \emph{perfect} if the state $\varphi_\mathbb{P}$ is quasi-free. Then, the associated covariance operator, denoted by $S_\mathbb{P}$, is called the \emph{canonical} covariance operator of $\mathbb{P}$.

Moreover, the restriction $\varphi_\mathbb{P}|_{C(\mathcal{C}(\mathfrak{X}))}$ to $C(\mathcal{C}(\mathfrak{X}))$ coincides with the expectation with respect to $\mathbb{P}$. Thus, for any $x_1, \dots, x_n\in \mathfrak{X}$ and $n\geq1$ we have
\begin{align*}
    \rho^{(n)}_\mathbb{P}(x_1, \dots, x_n)
    &=\varphi_\mathbb{P}(a^*_{x_n}\cdots a^*_{x_1}a_{x_1}\cdots a_{x_n})\\
    &=\mathrm{Pf}[K_\mathbb{P}(x_i, x_j)]_{i, j=1}^n,
\end{align*}
where $K_\mathbb{P}$ is given in Proposition \ref{prop:ex_pfpp} for $S_\mathbb{P}$. Thus, $\mathbb{P}$ is a Pfaffian point process, $K_\mathbb{P}$ is called the \emph{canonical} kernel. 

\begin{example}[{\cite[Theomrem 5.1]{Olshanski20}}]
    The $N$-point orthogonal polynomial ensemble $\mathbb{P}_{w, N}$ with weight function $w$ on $\mathfrak{X}$ is perfect, and the normalized Christoffel--Darboux kernel $K_{w, N}$ given in Equation \eqref{eq:CD_kernel} is its canonical kernel.
\end{example}

\subsection{Observation from the modular theory}
We keep the same notation in the previous section. Let $\varphi$ be an arbitrary state of $\mathfrak{I}_\mathrm{SDC}(\mathfrak{X})$ and $(\pi_\varphi, \mathcal{H}_\varphi, \Omega_\varphi)$ the associated GNS-triple. We denote by $\mathfrak{M}$ the von Neumann algebra generated by $\pi_\varphi(\mathfrak{I}_\mathrm{SDC}(\mathfrak{X}))$, that is, $\mathfrak{M}$ is the closure of $\pi_\varphi(\mathfrak{I}_\mathrm{SDC}(\mathfrak{X}))$ with respect to the weak operator topology. Let $\widehat \varphi$ denote the state of $\mathfrak{M}$ given by $\widehat\varphi(A):=\langle A\Omega_\varphi, \Omega_\varphi\rangle$ for any $A\in \mathfrak{M}$. 

In this section, we assume that $\Omega_\varphi$ is \emph{separating} for $\mathfrak{M}$. Namely, $A\Omega_\varphi=0$ implies $A=0$ for any $A\in \mathfrak{M}$. We summarize basic facts in the modular theory. See \cite[Section 2.5]{BR1}, \cite[Section 5.3]{BR2} for more details.

Since $\Omega_\varphi$ is separating for $\mathfrak{M}$, the conjugate linear operator $S_0$ given by $S_0A\Omega_\varphi:=A^*\Omega_\varphi$ ($A\in \mathfrak{M}$) is well defined on the dense domain $\mathfrak{M}\Omega_\varphi$. Moreover, it is known that $S_0$ is closable, and $S$ denotes its closure. Let $S=J\Delta^{1/2}$ be the polar decomposition, where $\Delta$ is a positive self-adjoint operator, called the \emph{modular operator}, and $J$ is a complex conjugation, called the \emph{modular conjugation}. The Tomita--Takesaki theorem (see \cite[Theorem 2.5.14]{BR1}) states that $\Delta^{\mathrm{i}t}\mathfrak{M}\Delta^{-\mathrm{i}t}=\mathfrak{M}$ holds for all $t\in \mathbb{R}$. Moreover, $\widehat\varphi$ satisfies the KMS condition at $-1$ with respect to the \emph{modular automorphism group} $\sigma\colon \mathbb{R}\curvearrowright\mathfrak{M}$ defined by $\sigma_t:=\mathrm{Ad}\Delta^{\mathrm{i}t}$. We remark that $\widehat \varphi$ is tracial (i.e., $\widehat\varphi(AB)=\widehat\varphi(BA)$ holds for any $A, B\in \mathfrak{M}$) if the modular automorphism group is trivial. Conversely, if $\widehat\varphi$ is tracial, then $S_0$ is isometric on $\mathfrak{M}\xi_\varphi$, and hence, $\Delta$ becomes the identity operator. Thus, the modular automorphism group is trivial.

We assume that a $\mathfrak{S}(\mathfrak{X})$-quasi-invariant point process $\mathbb{P}$ on $\mathfrak{X}$ is perfect, and $\varphi_\mathbb{P}$ denotes the associated quasi-free state. Then, $\varphi_\mathbb{P}$ is not tracial. In fact, we can easily show that $\varphi_\mathbb{P}$ is tracial if and only if $S_\mathbb{P}=1/2$. It implies that $\mathbb{P}$ is the product measure of the Bernoulli measure with parameter $1/2$. However, it is known that the product measure of a Bernoulli measure is not perfect (see \cite[Remark 4.11]{Olshanski20}).

\begin{proposition}\label{prop:nec_cond_perfectness}
    If a point process $\mathbb{P}$ on $\mathfrak{X}$ is perfect, then $\mathbbm{1}\in L^2(\mathcal{C}(\mathfrak{X}), \mathbb{P})$ is not separating for the von Neumann algebra $\mathfrak{M}$ generated by $T_\mathbb{P}(\mathfrak{I}_\mathrm{SDC}(\mathfrak{X}))$. Moreover, $\ker(1-S_\mathbb{P})\neq \{0\}$, or equivalently, $\ker(S_\mathbb{P})\neq \{0\}$.
\end{proposition}
\begin{proof}
    Since $\mathbb{P}$ is perfect, $(T_\mathbb{P}, L^2(\mathcal{C}(\mathfrak{X})), \mathbbm{1})$ is the GNS-triple associated with $\varphi_\mathbb{P}$. By \cite[Corollary 4.10]{Araki70}, $\mathbbm{1}$ is separating for $\mathfrak{M}$ if $\ker(1-S_\mathbb{P})=\{0\}$. Thus, the later statement follows from the former one. We show that $\mathbbm{1}$ is never separating for $\mathfrak{M}$. In fact, if $\mathbbm{1}$ is a separating vector, then, as we saw, the conjugate linear operator $S_0$ given by $S_0A\mathbbm{1}:=A^*\mathbbm{1}$ ($A\in \mathfrak{M}$) is well defined. In particular, we have $S_0f=S_0T_\mathbb{P}(f)\mathbbm{1}=\overline{f}$ for any $f\in C(\mathcal{C}(\mathfrak{X}))$. Thus, $S_0$ is isometric on $C(\mathcal{C}(\mathfrak{X}))\subset L^2(\mathcal{C}(\mathfrak{X}), \mathbb{P})$. However, $C(\mathcal{C}(\mathfrak{X}))$ is dense in $L^2(\mathcal{C}(\mathfrak{X}), \mathbb{P})$ since $\mathbb{P}$ is regular (see \cite[Theorem 3.14]{Rudin}). Thus, $S_0$ is uniquely extended to a complex conjugation on $L^2(\mathcal{C}(\mathfrak{X}), \mathbb{P})$. Therefore, the modular operator associated with the pair $(\mathfrak{M}, \mathbbm{1})$ is the identity operator. However, it is a contradiction since $\varphi$ is not tracial. Thus, $\mathbbm{1}$ is not a separating vector for $\mathfrak{M}$.
\end{proof}

Let us discuss the determinantal case. Namely, we assume that there exists a positive contraction operator $K_\mathbb{P}$ on $\ell^2(\mathfrak{X})$ such that the canonical covariance operator $S_\mathbb{P}$ has the form 
\[S_\mathbb{P}:=\begin{pmatrix} 1-JK_\mathbb{P}J & 0\\ 0 & K_\mathbb{P} \end{pmatrix}.\]
Thus, $\mathbb{P}$ is a DPP with correlation kernel $K_\mathbb{P}$ (see Remark \ref{rem:kernel}). By Proposition \ref{prop:nec_cond_perfectness}, we obtain the following:
\begin{corollary}\label{cor:nec_cond_perfectness}
    If $\mathbb{P}$ is perfect and its canonical covariance operator has the above form, then $\ker(K_\mathbb{P})\neq \{0\}$ or $\ker(1-K_\mathbb{P})\neq \{0\}$.
\end{corollary}
\begin{proof}
    We remark that $\ker(1-S_\mathbb{P})=\{0\}$ if and only if $\ker(K_\mathbb{P})=\{0\}$ and $\ker(1-K_\mathbb{P})=\{0\}$. Thus, the statement immediately follows from Proposition \ref{prop:nec_cond_perfectness}.
\end{proof}

We define $\mathcal{C}_0(\mathfrak{X}):=\{\omega\in \mathcal{C}(\mathfrak{X}) \mid \sum_{x\in \mathfrak{X}}\omega_x<\infty\}$, i.e., $\mathcal{C}_0(\mathfrak{X})$ is the set of finite simple point configurations on $\mathfrak{X}$. Let $L$ be a trace class operator on $\ell^2(\mathfrak{X})$ and assume that $\det L_\omega$ is nonnegative for any $\omega\in \mathcal{C}_0(\mathfrak{X})$, where $L_\omega$ is the restriction of $L$ to $\mathrm{span}\{\delta_{x} \mid \omega_x=1\}$. By assumption, we have
\[\det(1+L)=\sum_{\omega\in \mathcal{C}_0(\mathfrak{X})}\det L_\omega<\infty.\]
Thus, we obtain a point process $\mathbb{P}$, called the \emph{$L$-ensemble}, by
\[\mathbb{P}(\omega):=\frac{\det L_\omega}{\det(1+L)} \quad (\omega\in\mathcal{C}_0(\mathfrak{X})).\]
It is known that $\mathbb{P}$ is determinantal with correlation kernel $K:=L(1+L)^{-1}$. 

The following are consequences of Corollary \ref{cor:nec_cond_perfectness}.

\begin{corollary}
    Let $L$ be a self-adjoint operator of trace class. If $\ker(K)=\{0\}$, or equivalently, $\ker(L)=\{0\}$, then $K$ is never a canonical kernel of the associated $L$-ensemble.
\end{corollary}
\begin{proof}
    Since $1-K=(1+L)^{-1}$, its kernel must be $\{0\}$. Thus, if $K$ is a canonical kernel of the $L$-ensemble, we have $\ker(K)\neq\{0\}$. It is contradict to assumption.
\end{proof}

\begin{corollary}
    Let $H$ be a self-adjoint operator on $\ell^2(\mathfrak{X})$ such that $L:=e^{-\beta H}$ is of trace class for some $\beta\in \mathbb{R}$. Then, $K:=e^{-\beta H}(1+e^{\beta H})^{-1}$ is never a canonical kernel of the associated $L$-ensemble.
\end{corollary}

}

\end{document}